\documentclass[11pt]{amsart}

\usepackage{hyperref}
\usepackage{amssymb,amsfonts,amsmath,amsopn,amstext,amscd,latexsym,amsthm}
\usepackage{bbm}
\usepackage[all,cmtip]{xy}

\usepackage{enumerate}
\usepackage[shortlabels]{enumitem}

\theoremstyle{plain}

\begin{document}

\def\a{\alpha} 
 \def\b{\beta}
 \def\e{\epsilon}
 \def\d{\delta}
  \def\D{\Delta}
 \def\c{\chi}
 \def\k{\kappa}
 \def\g{\gamma}
 \def\t{\tau}
\def\ti{\tilde}
 \def\N{\mathbb N}
 \def\Q{\mathbb Q}
 \def\Z{\mathbb Z}
 \def\C{\mathbb C}
 \def\F{\mathbb F}
 \def\ovF{\overline\F}
 \def\bfN{\mathbf N}
 \def\cG{\mathcal G}
 \def\cT{\mathcal T}
 \def\cX{\mathcal X}
 \def\cY{\mathcal Y}
 \def\cC{\mathcal C}
 \def\cD{\mathcal D}
 \def\cZ{\mathcal Z}
 \def\cO{\mathcal O}
 \def\cW{\mathcal W}
 \def\cL{\mathcal L}
 \def\bfC{\mathbf C}
 \def\bfZ{\mathbf Z}
 \def\bfO{\mathbf O}
 \def\G{\Gamma}
 \def\go{\rightarrow}
 \def\do{\downarrow}
 \def\ra{\rangle}
 \def\la{\langle}
 \def\fix{{\rm fix}}
 \def\ind{{\rm ind}}
 \def\rfix{{\rm rfix}}
 \def\diam{{\rm diam}}
 \def\uni{{\rm uni}}
 \def\diag{{\rm diag}}
 \def\Irr{{\rm Irr}}
 \def\Syl{{\rm Syl}}
 \def\Gal{{\rm Gal}}
 \def\Tr{{\rm Tr}}
 \def\M{{\cal M}}
 \def\cE{{\mathcal E}}
\def\td{\tilde\delta}
\def\tx{\tilde\xi} 
\def\DC{D^\circ}
 
\def\Ker{{\rm Ker}}
 \def\rank{{\rm rank}}
 \def\soc{{\rm soc}}
 \def\Cl{{\rm Cl}}
 \def\A{{\sf A}}
 \def\sP{{\sf P}}
 \def\sQ{{\sf Q}}
 \def\SSS{{\sf S}}
  \def\SQ{{\SSS^2}}
 \def\St{{\sf {St}}}
 \def\p{\ell}
 \def\ps{\ell^*}
 \def\SC{{\rm sc}}
 \def\supp{{\rm supp}}
  \def\cR{{\mathcal R}}
 \newcommand{\tw}[1]{{}^#1}
 
 \def\Sym{{\rm Sym}}
 \def\PSL{{\rm PSL}}
 \def\SL{{\rm SL}}
 \def\Sp{{\rm Sp}}
 \def\GL{{\rm GL}}
 \def\SU{{\rm SU}}
 \def\GU{{\rm GU}}
 \def\SO{{\rm SO}}
 \def\PO{{\rm P}\Omega}
 \def\Spin{{\rm Spin}}
 \def\PSp{{\rm PSp}}
 \def\PSU{{\rm PSU}}
 \def\PGL{{\rm PGL}}
 \def\PGU{{\rm PGU}}
 \def\Iso{{\rm Iso}}
 \def\GO{{\rm GO}}
 \def\Ext{{\rm Ext}}
 \def\E{{\cal E}}
 \def\l{\lambda}
 \def\ve{\varepsilon}
 \def\Lie{\rm Lie}
 \def\s{\sigma}
 \def\O{\Omega}
 \def\o{\omega}
 \def\ot{\otimes}
 \def\op{\oplus}
 \def\oc{\overline{\chi}}
 \def\pf{\noindent {\bf Proof.$\;$ }}
 \def\Proof{{\it Proof. }$\;\;$}
 \def\no{\noindent}
 %{\quad\sqbox\vspace{1mm}\par}
\def\hal{\unskip\nobreak\hfil\penalty50\hskip10pt\hbox{}\nobreak
 \hfill\vrule height 5pt width 6pt depth 1pt\par\vskip 2mm}

 \renewcommand{\thefootnote}{}

\newtheorem{theorem}{Theorem}
 \newtheorem{thm}{Theorem}[section]
 \newtheorem{prop}[thm]{Proposition}
 \newtheorem{lem}[thm]{Lemma}
 \newtheorem{lemma}[thm]{Lemma}
 \newtheorem{defn}[thm]{Definition}
 \newtheorem{cor}[thm]{Corollary}
 \newtheorem{coroll}[theorem]{Corollary}
 \newtheorem{rem}[thm]{Remark}
 \newtheorem{exa}[thm]{Example}
 \newtheorem{cla}[thm]{Claim}

\numberwithin{equation}{section}
\parskip 1mm

\title{Surjective word maps and Burnside's $p^aq^b$ theorem}

\author[Guralnick, Liebeck, O'Brien, Shalev, and Tiep]
{Robert M. Guralnick \and Martin W. Liebeck \and E.A. O'Brien \and Aner Shalev \and Pham Huu Tiep }
\date{}

\thanks{The first author was partially supported by  NSF
grants DMS-1001962, DMS-1302886, and the Simons Foundation Fellowship 
224965.  He also thanks the Institute for Advanced Study for its support.
The third author was partially supported by the Marsden
Fund of New Zealand via grant UOA 105.
The fourth author was supported by ERC Advanced Grant 247034,
ISF grant 1117/13 and the Vinik Chair of Mathematics which he holds.
The fifth author was partially supported by the NSF grant DMS-1201374, the Simons Foundation
Fellowship 305247, the Mathematisches Forschungsinstitut Oberwolfach, and the EPSRC. 
Parts of the paper were written while the fifth author visited the Department of Mathematics, Harvard 
University, and Imperial College, London. He thanks Harvard University and 
Imperial College for generous hospitality and stimulating environments.}
\thanks{The authors thank Frank L\"ubeck for providing them with the proof of Lemma \ref{main2-slu34}.}

\begin{abstract} 
We prove surjectivity of certain word maps on finite non-abelian simple groups. 
More precisely, we prove the following:
if $N$ is a product of two prime powers, 
then the word
map $(x,y) \mapsto x^Ny^N$ is surjective on every finite non-abelian simple group; 
if $N$ is an odd integer, then the word map $(x,y,z) \mapsto x^Ny^Nz^N$ is surjective 
on every finite quasisimple group. 
These generalize classical theorems of Burnside and Feit-Thompson. 
We also prove asymptotic results about the surjectivity of the 
word map $(x,y) \mapsto x^Ny^N$ that depend on the number of prime factors of
the integer $N$.   
\end{abstract}

\maketitle

 % \footnotetext{2010 {\it Mathematics Subject Classification: }
 % }

\tableofcontents

\section{Introduction}

The theory of word maps on finite non-abelian 
simple groups -- that is, maps of the form $(x_1,\ldots ,x_k)
\mapsto w(x_1,\ldots ,x_k)$ for some word $w$ in the free group $F_k$ of
rank $k$ -- has attracted much attention. It was
shown in \cite[1.6]{LS} that for a given nontrivial word $w$, every
element of every sufficiently large finite simple group $G$ can be
expressed as a product of $C(w)$ values of $w$ in $G$, where $C(w)$
depends only on $w$; and this has been improved to $C(w)
= 2$ in \cite{LarSh, LarShT, Sh}. Improving $C(w)$ to 1 is not possible in
general, as is shown by power words $x_1^n$, which cannot be
surjective on any finite group of order non-coprime to $n$. 

Certain word maps are surjective on all groups -- namely, those in cosets of 
the form $x_1^{e_1} \ldots x_k^{e_k}F_k'$ where the $e_i$ are integers with 
${\rm gcd}(e_1,\ldots,e_k)=1$
(see \cite[3.1.1]{S}). 
The word maps for a small number of other words have been shown to be 
surjective on all finite simple groups.
These include the commutator word $[x_1,x_2]$,
whose surjectivity was conjectured by Ore in 1951 and proved in 
2010 (see \cite{ore} and the references therein).

The main result of this paper is the following.

\begin{theorem}\label{main1}
Let $p,q$ be primes, let $a,b$ be non-negative integers, and let $N=p^aq^b$. 
The word map $(x,y) \mapsto x^Ny^N$ is surjective on all finite (non-abelian) simple groups.
\end{theorem}

This result generalizes various theorems.

First, it implies the classical Burnside $p^aq^b$-theorem stating that groups 
of this order are soluble. Indeed, if $G$ is a non-soluble group of order
$N = p^aq^b$, then $G$ has a non-abelian composition factor $S$ whose order
divides $N$. Thus $S$ is a (non-abelian) finite simple group satisfying
the identity $x^N=1$, so the word map $x^Ny^N$ on $S$ has the trivial
image $\{ 1 \}$, contradicting Theorem \ref{main1}.

Theorem \ref{main1} also implies the surjectivity of $x^2y^2$ and more
generally of the words $x^{p^a}y^{p^a}$ (for a prime $p$), as
established in \cite{GM, LOST}. 

In \cite[Cor.\ 1.5]{GM} it is shown that $x^{6^a}y^{6^a}$ is surjective on all
(non-abelian) finite simple groups, again a particular case
of Theorem \ref{main1}.

This theorem is best possible in the sense that it cannot be 
extended to the case where $N$ is a product of three or more prime powers, 
since such a number can be the exponent of a simple group. Indeed, the
smallest example is that of $\A_5$.

If $N_1, N_2$ are positive integers such that $N_1N_2$ is 
divisible by at most two primes, then $x^{N_1}y^{N_2}$ is surjective on all 
(non-abelian) finite simple groups, since 
$(x^{N_2})^{N_1}(y^{N_1})^{N_2} = x^{N_1N_2}y^{N_1N_2}$ is surjective
by Theorem \ref{main1}. 

%\vspace{4mm} 
But some more general questions, including the following, have a negative answer. 
 If $N$ is not divisible by the exponent of a finite simple group $G$, is $x^Ny^N$ surjective on $G$? If $N$ is odd, is 
$x^Ny^N$ surjective on all finite non-abelian simple groups? If $N = p^aq^b$ for some primes $p,q$,
is $x^Ny^N$ surjective on all finite quasisimple groups, or does it hit at least all non-central elements
of every quasisimple group? See Remark \ref{general}.

However, we prove the following result which generalizes the celebrated Feit-Thompson theorem:

\begin{theorem}\label{main2}
Let $N$ be an odd positive integer. The word map $(x,y,z) \mapsto x^Ny^Nz^N$ is surjective on all finite quasisimple groups. In fact, every element
of every finite quasisimple group is a product of three $2$-elements.
\end{theorem}

As mentioned above, this result is best possible in the sense that
it does not hold for $x^Ny^N$; it also implies the surjectivity
of $x^{N_1}y^{N_2}z^{N_3}$ for odd numbers $N_1, N_2, N_3$.

A key ingredient of our proof of Theorem \ref{main2} is the construction of certain $2$-elements in simple groups $G$ of 
Lie type in odd characteristic that are regular if $G$ is classical (see \S7.2) and almost regular if $G$ is exceptional
(see \S7.4). This construction may be useful in other situations. There are other results of the same 
flavor as the second statement of Theorem \ref{main2}, such as \cite[Theorem 3.8]{GT3} where $p$-elements are
considered instead of $2$-elements. There is also considerable literature on the case of involutions, 
see e.g.\ \cite{Mac} and the references therein. 
These imply results like Theorem \ref{main2} with longer products
$x_1^N x_2^N \ldots x_t^N$, where $N$ is not divisible by the exponent of the simple group in question, see for example \ \cite[Corollary 3.9]{GT3}.

\smallskip 
Recall that the main results of \cite{LarSh, LarShT} assert that, given two non-trivial words $w_1$ and $w_2$, the product 
$w_1w_2$ is surjective on all finite non-abelian simple groups of {\it sufficiently large} order (depending on $w_1$ and $w_2$).   
In particular, once we fix a positive integer $N$, the word $x^Ny^N$ is surjective on all sufficiently large simple groups.
Theorem \ref{main1} (and \ref{main2}) shows that, for all $N$ of the prescribed form, the word map 
$x^Ny^N$ (respectively $x^Ny^Nz^N$) is in fact surjective on {\it all} simple groups (respectively 
quasisimple groups).
 
As mentioned above, one cannot generalize Theorem \ref{main1} for products of more than two prime powers. However,
we prove results of that flavor by imposing asymptotic conditions on the simple groups. To formulate   
these results, define $\pi(N) = k$ and 
$\O(N) = \sum^k_{i=1}\alpha_i$ if the integer $N$ has the prime factorization 
$N = \prod^k_{i=1}p_i^{\alpha_i}$ (with $p_1 < \ldots < p_k$ and $\alpha_i > 0$). 

\begin{theorem}\label{main3}
Given a positive integer $k$, there is some $f(k)$ such that for all positive integers $N$ with $\pi(N) \leq k$,
the word map $(x,y) \mapsto x^Ny^N$ is surjective on all finite simple groups $S$, 
where $S$ is either an alternating group $\A_n$ with $n \geq f(k)$, or a 
simple Lie-type group of rank $\geq f(k)$ and defined over $\F_q$ with $q \geq f(k)$.
\end{theorem}

\begin{theorem}\label{main4}
Given a positive integer $k$, there is some $g(k)$ such that for all positive integers $N$ with 
$\O(N) \leq k$, the word map $(x,y) \mapsto x^Ny^N$ is surjective on all finite simple groups $S$, where $S$ is either an alternating group $\A_n$ with $n \geq g(k)$, or a 
simple Lie-type group of rank $\geq g(k)$.
\end{theorem}

Neither Theorem \ref{main3} nor \ref{main4} holds for finite simple Lie-type 
groups of bounded rank, cf.\ Example \ref{bounded}. It remains an open question whether Theorem \ref{main3} holds
for finite simple Lie-type groups of unbounded rank over fields of bounded size.

\smallskip
We use the notation of \cite{KL} for finite groups of Lie type.
For $\e = \pm$, the group $\SL^\e_n(q)$ is $\SL_n(q)$ when $\e = +$ and $\SU_n(q)$ when $\e = -$, and similarly for $\GL^\e_n(q)$, $\PSL^\e_n(q)$.
Also, $E_6^\e(q)$ is $E_6(q)$ if $\e = +$ and
$\tw2 E_6(q)$ if $\e = -$.  We use the convention that if $\e = \pm$ then expressions such as $q-\e$ mean $q-\e 1$.

\section{Preliminaries}
The following plays a key role in our proofs.

\begin{thm}{\rm \cite[Theorem 1.1]{GT3}}\label{prime1}
Let $\cG$ be a simple simply connected algebraic group in characteristic 
$p > 0$ and let $F~:~\cG \to \cG$ be a Frobenius endomorphism such that 
$G := \cG^F$ is quasisimple. There exist (not necessarily distinct)
primes $r,s_1,s_2$, all different from $p$, and 
regular semisimple $x,y \in G$ such that $|x| = r$, 
$y$ is an $\{s_1,s_2\}$-element, and $x^G\cdot y^G \supseteq G \setminus \bfZ(G)$.
In fact $s_1 = s_2$ unless $\cG$ is of type $B_{2n}$ or $C_{2n}$.
\end{thm} 

Throughout the paper, by a {\it finite simple group of Lie type in characteristic $p$} we mean
a simple non-abelian group $S = G/\bfZ(G)$ for some $G = \cG^F$ as in Theorem \ref{prime1}.
In this notation, let $q = p^f$ denote the common absolute value of the eigenvalues of $F$ acting on the character group of an $F$-stable maximal torus (so that $f$ is half-an-integer if 
$G$ is a Suzuki-Ree group). For each group $G$
%mentioned in Theorem \ref{prime1} 
and $S = G/\bfZ(G)$, we refer to the set 
$\{r,s_1,s_2\}$ specified in the proof of \cite[Theorem 1.1]{GT3} as $\cR(G)$ and $\cR(S)$. 
 
\begin{cor}\label{prime2}
In the notation of Theorem \ref{prime1}, let $S = G/\bfZ(G)$ be simple non-abelian. 
%Then the following statements hold.

\begin{enumerate}[\rm(i)]

\item Theorem \ref{main1} holds for $S$, unless possibly $N = p^at^b$ with $t \in \{r,s_1,s_2\}$.

\item Suppose $N = p^at^b$ for some prime $t$ and $|\cX| < |G|/2$, where $\cX$ is 
the set of all elements of $G$ of order divisible by $p$ or by $t$. The word map
$(x,y) \mapsto x^Ny^N$ is surjective on $G$.
\end{enumerate}
\end{cor}

\pf
(i) By \cite[Corollary, p.\ 3661]{EG}, every non-central element of $G$ is a product of two $p$-elements.
Hence Theorem \ref{main1} holds for $S$ if $p \nmid N$. On the other hand, if 
$N = p^at^b$ with $t \notin \{r,s_1,s_2\}$, then the elements $x$ and $y$ in Theorem 
\ref{prime1} are $N$th powers, so Theorem \ref{main1} again holds for $S$. 

(ii) Let $g \in G$.  By assumption, $|G \setminus \cX| > |G|/2$, so 
$g(G \setminus \cX) \cap (G \setminus \cX) \neq \emptyset$.  Hence 
$g = xy^{-1}$ for some $x,y \in G \setminus \cX$. Note that every element of $G \setminus \cX$ is an 
$N$th power, whence the claim follows.
\hal

Recall that if $a \geq 2$ and $n \geq 3$ are integers and $(a,n) \neq (2,6)$, then $a^n-1$ has 
a {\it primitive prime divisor}, i.e.\ a prime divisor that does not divide $\prod^{n-1}_{i=1}(a^i-1)$,
cf.\ \cite{Zs}. In what follows, we fix one such prime divisor for given $(a,n)$ and denote it 
by $\p(a,n)$. Next we record the 
primes $r,s_1,s_2$ mentioned in Theorem \ref{prime1} in Table \ref{primes} (for larger groups 
$G$). The third column
of Table \ref{primes} contains one entry precisely when $s_1 = s_2$. 
\begin{table}[htb]
\[
\begin{array}{|c|c|c|c|}
\hline 
G & r & s_1,s_2 & (n,q) \neq \\
\hline
\begin{array}{c}\SL_n(q),\\ n \geq 4 \end{array}
   & \p(p,nf) & \p(p,(n-1)f) & (6,2),(7,2), (4,4) \\ \hline

\begin{array}{c}\SU_n(q),\\ n \geq 5 \hbox{ odd} \end{array}
   & \p(p,2nf) & \begin{array}{ll} \p(p,(n-1)f), & n \equiv 1 \bmod 4\\
                                                  \p(p,(n-1)f/2), & n \equiv 3 \bmod 4 \end{array}
    & (7,4) \\ \hline
\begin{array}{c}\SU_n(q),\\ n \geq 4 \hbox{ even} \end{array}
   & \p(p,(2n-2)f) & \begin{array}{ll} \p(p,nf), & n \equiv 0 \bmod 4\\
                                                       \p(p,nf/2), & n \equiv 2 \bmod 4 \end{array}
    & (4,2),(6,4) \\ \hline   

\begin{array}{c}\Sp_{2n}(q),\\ \Spin_{2n+1}(q),\\ n \geq 3 \hbox{ odd} \end{array}
   & \p(p,2nf) & \p(p,nf) & (3,4) \\ \hline
\begin{array}{c}\Sp_{2n}(q),\\ \Spin_{2n+1}(q),\\ n \geq 6 \hbox{ even} \end{array}
   & \p(p,2nf) & \p(p,nf), \p(p,nf/2) & (6,2), (12,2) \\ \hline
\Sp_{24}(2) & 241 & 13, 7 & \\   
\Sp_{12}(2) & 13 & 3, 7 & \\  
   \hline

\begin{array}{c}\Spin^+_{2n}(q),\\ n \geq 4 \end{array}
   & \p(p,(2n-2)f) & \begin{array}{ll}\p(p,nf), & n \hbox{ odd}\\
                                                      \p(p,(n-1)f), & n \hbox{ even} \end{array} &  (4,2)\\ \hline   

\begin{array}{c}\Spin^-_{2n}(q),\\ n \geq 4 \end{array}
   & \p(p,2nf) & \p(p,(2n-2)f) & (4,2) \\ \hline

\tw2 B_2(q^2) & \p(2,8f) & \p(2,8f) & q^2 > 8\\ \hline
\tw2 G_2(q^2) & \p(3,12f) & \p(3,12f) & q^2 > 27\\ \hline
\tw2 F_4(q^2) & \p(2,24f) & \p(2,12f) & q^2 > 8\\ \hline
G_2(q) & \p(p,3f) & \p(p,3f) & q \neq 2,4\\ \hline
\tw3 D_4(q) & \p(p,12f) & \p(p,12f) & \\ \hline
F_4(q) & \p(p,12f) & \p(p,8f) & \\ \hline
E_6(q)_{\SC} & \p(p,9f) & \p(p,8f) & \\ \hline
\tw2 E_6(q)_{\SC} & \p(p,18f) & \p(p,8f) & \\ \hline
E_7(q)_{\SC} & \p(p,18f) & \p(p,7f) & \\ \hline
E_8(q) & \p(p,24f) & \p(p,20f) & \\ \hline

\end{array}
\]
\caption{Special primes for simple groups of Lie type} \label{primes}
\end{table} 

% We frequently use the following statements.

\begin{lem}\label{basic}
Let $G$ be a finite group, fix $g_1,g_2 \in G$, and let $g \in G$.

{\rm (i)} Then $g \in g_1^G \cdot g_2^G$ if and only if
$$\sum_{\chi \in \Irr(G)}\frac{\chi(g_1)\chi(g_2)\oc(g)}{\chi(1)} \neq 0.$$
In particular, $g \in g_1^G\cdot g_2^G$ if 
$$\left|\sum_{\chi \in \Irr(G),~\chi(1) > 1}\frac{\chi(g_1)\chi(g_2)\oc(g)}{\chi(1)} \right|
     < \left|\sum_{\chi \in \Irr(G),~\chi(1)=1}\frac{\chi(g_1)\chi(g_2)\oc(g)}{\chi(1)} \right|.$$
     
{\rm (ii)} For $D \in \N$,
$$\left|\sum_{\chi \in \Irr(G),~\chi(1) \geq D}\frac{\chi(g_1)\chi(g_2)\oc(g)}{\chi(1)}\right| 
    \leq \frac{1}{D}\left(|\bfC_G(g_1)|\cdot|\bfC_G(g_2)|\cdot|\bfC_G(g)|\right)^{1/2}.$$
\end{lem}

\pf
%(i) is well known. 
The first is the well known result of Frobenius.
For (ii), note that $|\chi(g)| \leq |\bfC_G(g)|^{1/2}$ for $\chi \in \Irr(G)$ 
by the second orthogonality relation for complex characters. By the Cauchy-Schwarz inequality,
$$\sum_{\chi \in \Irr(G)}|\chi(g_1)\chi(g_2)| 
    \leq \left(\sum_{\chi \in \Irr(G)}|\chi(g_1)|^2 \cdot \sum_{\chi \in \Irr(G)}|\chi(g_2)|^2\right)^{1/2} 
    = (|\bfC_G(g_1)|\cdot|\bfC_G(g_2)|)^{1/2}.$$
\hal

\begin{lem}\label{alt-spor}
Theorem \ref{main1} holds for all alternating groups $\A_n$, $5 \leq n \leq 18$, and for
all $26$ sporadic finite simple groups.
\end{lem}

\pf
For each of these groups $G$ and for every two primes $p,q$ dividing $|G|$, we 
verify that each $g \in G$ can be written as a product of two $\{p,q\}'$-elements.
We do this by applying Lemma \ref{basic} to the character table
of the relevant group. 
Some of these character tables are available in the
Character Table Library of {\sf GAP} \cite{GAP};
the remainder were constructed directly using
the {\sc Magma} \cite{Magma} implementation of the
algorithm of Unger \cite{Unger}.
\hal

\begin{prop}\label{alt1}
Theorem \ref{main1} holds for $S = \A_n$ if $n \geq 19$.
\end{prop}

\pf
Since $n \geq 19$, there are at least $6$ consecutive integers in the interval
$[\lfloor 3n/4 \rfloor, n]$. In particular, we can find an odd integer $m$ such that 
$\lfloor 3n/4 \rfloor \leq m < m+4 \leq n$. Suppose now that $N = p^aq^b$. Among $m$, $m+2$, and 
$m+4$, at most one integer is divisible by $p$, and similarly for $q$. Hence there is some 
$\ell \in \{m,m+2,m+4\}$ that is coprime to $N$. According to \cite[Corollary 2.1]{B}, each 
$g \in \A_n$ is a product of two $\ell$-cycles. Since every $\ell$-cycle is an $N$th power
in $S$, we are done.
\hal

\begin{prop}\label{alt2}
Given a positive integer $k$, there is some $f(k)$ such that for all $n \geq f(k)$ and 
for all positive integers $N$ with at most $k$ distinct prime factors, the word map 
$(x,y) \mapsto x^Ny^N$ is surjective on $S = \A_n$.
\end{prop}

\pf
Choosing $f(k)$ large enough, we see by the prime number theorem that, for every $n \geq f(k)$, 
the interval $[3n/4,n]$ contains at least $k+1$ distinct primes $p_1, \ldots, p_{k+1}$. Given
a positive integer $N$ with at most $k$ distinct primes factors, at least one of the $p_i$'s, 
call it $\ell$, 
does not divide $N$, whence all $\ell$-cycles are $N$th powers. Hence the claim follows from
\cite[Corollary 2.1]{B}.
\hal

%The next simple observation is helpful in various situations.

\begin{lem}\label{real1}
If $g$ is a real element of a finite group $G$, then $g$ is a product of two
$2$-elements of $G$.
\end{lem}

\pf
By assumption, $g^{-1} = xgx^{-1}$ for some $x \in G$. Replace $x$ by 
$x^{|x|_{2'}}$ to obtain a $2$-element. Now 
$$xgxg = x^{2} \cdot x^{-1}gx \cdot g = x^2 \cdot g^{-1} \cdot g = x^2,$$
so $xg$ is a $2$-element as well. Since $g = x^{-1} \cdot xg$, the claim follows.
\hal

In particular, the following is an immediate consequence of Lemma \ref{real1}:

\begin{cor}\label{real2}
If $G$ is a finite real group and $N$ is an odd integer, then the word map
$(x,y) \mapsto x^Ny^N$ is surjective on $G$.
\end{cor}

\begin{cor}\label{main-real}
Let $q = p^f$ be an odd prime power. Theorem  \ref{main1} holds for the following simple 
groups:

\begin{enumerate}[\rm(i)]
\item $\PSp_{2n}(q)$ and $\Omega_{2n+1}(q)$, where $n \geq 3$ and $q \equiv 1 \bmod 4$;

\item $\PO^+_{4n}(q)$, where $n \geq 3$ and $q \equiv 1 \bmod 4$;

\item $\PO^{-}_{4n}(q)$, where $n \geq 2$;

\item $\PO^+_8(q)$, $\O_9(q)$, and $\tw3 D_4(q)$. 
\end{enumerate}
If $N$ is an arbitrary odd integer, then the word map $(x,y) \mapsto x^Ny^N$ is surjective on 
each of these groups. The same conclusion holds for $G = \Spin^-_{4n}(q)$ with $n \geq 2$,
and $G = \Omega^+_{4n}(q)$ with $n \geq 2$ and $q \equiv 1 \bmod 4$, and $G = \Omega^+_8(q)$.
\end{cor}

\pf
By \cite[Theorem 1.2]{TZ3}, all of these groups $G$ are real, whence 
the statement follows from Corollary \ref{real2} when $N$ is odd. If 
$G$ is simple and $N$ is even, then the statement follows from Corollary \ref{prime2}(i).
\hal

Corollary \ref{main-real} implies that Theorem \ref{main1} holds for many simple 
symplectic or orthogonal groups over $\F_q$ when $q \equiv 1 \bmod 4$. To handle the 
groups over $\F_q$ with $q \equiv 3 \bmod 4$ or $2|q$, we use the following result:

\begin{prop}\label{ratio}
Let $S$ be a non-abelian simple group of Lie type in characteristic $p$. Suppose $N = p^at^b$ with $t \in \cR(S)$,
where $\cR(S)$ is defined after Theorem \ref{prime1}. The word map $(x,y) \mapsto x^Ny^N$ is 
surjective on $G$, where $S = G/\bfZ(G)$ and
$G$ is one of the following groups:

\begin{enumerate}[\rm(i)]
\item $\Sp_{2n}(q)$, where $2|q \geq 8$ and $2 \nmid n \geq 3$; 

\item $\Sp_{2n}(q)$, where $2 \nmid q  \geq 11$ and $2 \nmid n  \geq 3$; 

\item $\O_{2n+1}(q)$, where $2 \nmid q \geq 7$, $2 \nmid n \geq 3$, and
$(n,q) \neq (3,7)$; 

\item $\Omega^\pm_{2n}(q)$, where $n \geq 4$, $q \geq 5$, and 
$n \neq 5,7$ when $q=5$.
\end{enumerate}
\end{prop}
 
\pf
By Corollary \ref{prime2}(ii), it suffices to show that $|\cX| < |G|/2$ for
$\cX = \cX_p \cup \cX_t$, where $\cX_s$ is  the set of all elements of $G$ that have order divisible by 
$s$ for $s \in \{p,t\}$. 
We use \cite[Theorem 2.3]{GL} which states that $|\cX_p|/|G| < c(q)$,
where 
$$c(q) :=  \left\{ \begin{array}{ll}2/(q-1) + 1/(q-1)^2, & G = \Sp_{2n}(q), ~2|q\\ 
   3/(q-1) + 1/(q-1)^2, & G = \Sp_{2n}(q), ~2 \nmid q \\ 2/(q-1) + 2/(q-1)^2, & G = \Omega_{2n+1}(q)
     \mbox{ or }\Omega^\pm_{2n}(q). 
    \end{array} \right.$$
(Note that this result applies to $G$ since $\bfZ(G)$ is a $p'$-group.)
To estimate $|\cX_t|$, observe that every nontrivial $t$-element $x$ of  $G$ 
is regular semisimple, 
with $\bfC_G(x)$ being a conjugate $T^g$ of a fixed maximal torus $T$ of $G$. Hence
if $y \in \cX_t$ has the $t$-part equal to $x$ then $y \in T^g$. 
It follows that 
$$|\cX_t|/|G| < |T|/|\bfN_G(T)|.$$ 
%by \cite[Lemma 2.1]{GL} (as $1 \in T$ but $1 \notin \cX_t$). 
For cases (i)--(iii), 
$\bfN_G(T)/T$ contains a cyclic group of odd order $n$.
Moreover, since the central involution of the Weyl group of $G$ inverts $T$, 
cf.\ \cite[Proposition 3.1]{TZ3}, $|\bfN_G(T)/T|$ is even.
It follows that $2n$ divides $|\bfN_G(T)/T|$. If in addition $G \neq \O_7(7)$, then
$$\frac{|\cX|}{|G|} \leq \frac{|\cX_t|}{|G|} + \frac{|\cX_p|}{|G|} < \frac{1}{2n} + c(q) < 0.49.$$
%respectively. 
%When $G = \O_7(7)$, note that $t = 19$ or $43$ and 
%$|T:P_t| \geq 9$, respectively $4$. It follows that $|\bfN_G(T)/P_t|  \geq 4 \cdot 6$ and so 
%$|\cX|/|G| <  0.44$ -- needs to be revised.

In case (iv), we may by Corollary \ref{main-real} assume 
that $n \geq 5$. Note that $T$ is constructed using two kinds of cyclic maximal tori.
The first is  
$$T_1 = \SO^+_2(q^m) \cap G \leq \SO^+_{2m}(q)$$
with $m$ odd, where 
$$\bfN_{\Omega^+_{2m}(q)}(T_1)/T_1 \cong  C_m.$$ 
The second is  
$$T_2 = \SO^-_2(q^m) \cap G \leq \SO^-_{2m}(q),$$
where 
$$\bfN_{\Omega^-_{2m}(q)}(T_2)/T_2  \hookleftarrow C_m.$$   
Furthermore, $m \in \{n-1,n\}$.  
Hence, if $q \geq 7$, or $q = 5$ and $n \geq 9$, then 
$$\frac{|\cX|}{|G|} \leq \frac{|\cX_t|}{|G|} + \frac{|\cX_p|}{|G|} <
    \frac{1}{n-1}+ \frac{1}{q-1} + \frac{2}{(q-1)^2} \leq 1/2,$$
as desired. If $q = 5$ and $n = 6,8$ then we are done by Corollary \ref{main-real}.
%In the remaining cases where $q=5$ and $n=5,7$, note that $|T/P_t| \geq 2$. (Indeed, if $m = n$,
%then $|T| = (5^n \pm 1)/2$ is divisible by $3$ or $2$ and so $|P_t| \leq |T|/2$. If $m = n-1$, then,
%since $m$ is even, $T$ is constructed using $T_2 = \SO^-_2(5^m)$, and
%$$T = (\SO^-_2(5^m) \times \SO^\e_2(5)) \cap G,$$
%where $G = \O^\e_{2n}(5)$ and $\e = \pm$. Since $P_t < T_2$,  $|P_t| \leq |T|/4$.)  It follows 
%that $|\cX_t|/|G| \leq 1/2(n-1) \leq 1/8$ and so $|\cX|/|G| < 1/2$. -- need to
%be revised.
\hal

\begin{lem}\label{orbit}
Let $G$ be a finite group and let $g \in G$. If $\cO$ is a $\Gal(\overline\Q/\Q)$-orbit on 
$$\{ \chi \in \Irr(G) \mid \chi(g) \neq 0\},$$ 
then 
$$\sum_{\chi \in \cO}|\chi(g)|^2 \geq |\cO|.$$
In particular,
$$|\{ \chi \in \Irr(G) \mid \chi(g) \neq 0\}| \leq |\bfC_G(g)|.$$ 
\end{lem}

\pf
Note that $\prod_{\chi \in \cO}\chi(g)$ is a nonzero algebraic integer fixed by 
$\Gal(\overline\Q/\Q)$, whence it is a nonzero integer. The Cauchy-Schwarz
 inequality implies that 
$$\sum_{\chi \in \cO}|\chi(g)|^2 \geq |\cO| \cdot |\prod_{\chi \in \cO}\chi(g)|^{2/|\cO|} \geq |\cO|.$$
%For the second statement, 
Let $\cO_1, \ldots, \cO_t$ denote all of the distinct 
$\Gal(\overline\Q/\Q)$-orbits on $\{ \chi \in \Irr(G) \mid \chi(g) \neq 0\}$. The first 
statement implies that 
$$|\bfC_G(g)| = \sum_{\chi \in \Irr(G),~\chi(g) \neq 0}|\chi(g)|^2 = 
   \sum^t_{i=1}\sum_{\chi \in \cO_i}|\chi(g)|^2 \geq \sum^t_{i=1}|\cO_i|.$$
\hal
 
\begin{rem}\label{general}
{\rm Some natural generalizations of Theorem \ref{main1} are false.}

{\rm (i)} It is not true that for every $N = p^aq^b$ the word map $(x,y) \mapsto x^Ny^N$ is always surjective on every quasisimple group $G$, or at least hits all the non-central elements of $G$.  {\rm For instance, 
if $N = 20$, then this map does not hit any element of order $5$ in $G = \SL_2(5)$ (indeed, 
$x^{20}$ has order $1$ or $3$ in $G$, and if $x \in G$ has order $3$ then 
$\{1\} \cup x^G \cup x^G \cdot x^G$ does not contain any element of order $5$ of $G$).}

\smallskip
{\rm (ii)} It is not true that for every odd integer $N$ the word map $(x,y) \mapsto x^Ny^N$ is always surjective on every non-abelian simple group $G$. {\rm For instance, consider a
prime power $q > 3$ where $q \equiv 3 \bmod 8$ and set $G : = \PSL_2(q)$ and $N := q(q^2-1)/8$. 
Note that $x^N$ has order $1$ or $2$ for every $x \in G$. 
%Note that for any $x \in G$, $x^N$ has order $1$ or $2$. 
It follows that every element of $G$ that is hit by the 
word map $(x,y) \mapsto x^Ny^N$ is either an involution or a product of two involutions, so 
it is real. On the other hand, the nontrivial unipotent elements of $G$ are not real. The same arguments show that
the word map $(x,y) \mapsto x^Ny^N$ is not surjective on the Ree group $G = \tw2 G_2(q)$, if $q = 3^{2a+1} > 3$ and 
$N = |G|_{2'}$. 
It is an open question whether these two families of simple groups exhaust all the simple groups $G$ on which
the word map $(x,y) \mapsto x^Ny^N$ is not surjective for some odd $N$.}
\end{rem}

\begin{exa}\label{bounded}
{\em By \cite[Corollary 4.2]{AST}, there are infinitely primes $p$ such that $\O(p^2-1) \leq 21$. For every 
such prime $p$, the exponent of $\PSL_2(p)$ divides $N_p := p(p^2-1)$, so 
the word map $(x,y) \mapsto x^{N_p}y^{N_p}$ cannot be surjective on $\PSL_2(p)$ (its image consists only of the identity element); on the other hand, $\pi(N_p) \leq \O(N_p) \leq 22$. Thus 
neither Theorem \ref{main3} nor \ref{main4} holds for finite simple groups of Lie type and 
bounded rank.}
\end{exa} 
 
\section{Centralizers of unbreakable elements}

\subsection{Symplectic and orthogonal groups}

\begin{defn}\label{unbr-spo}
{\em Let $\Cl(V) = \Sp(V)$ or $\O (V)$ be a finite symplectic or orthogonal group. An element $x $ of $\Cl(V)$ is {\it breakable} if there is a proper, nonzero, non-degenerate subspace $U$ of $V$ such that $x  = x_1x_2 \in \Cl(U) \times \Cl(U^\perp)$ (with $x_1 \in \Cl(U)$, $x_2 \in \Cl(U^\perp)$), and either 

(i) $\Cl(U)$ and $\Cl(U^\perp)$ are both perfect, or

(ii) $\Cl(U^\perp)$ is perfect and $x_1 = \pm 1_{U}$.

\noindent
Otherwise, $x$ is {\it unbreakable}.}
\end{defn}

\begin{lemma}\label{spunb}
Let $G = \Sp_{2n}(q) = \Sp(V)$ with $n\ge 2$, and assume that $n\ge 4$ if $q=3$ and that $n\ge 7$ if $q=2$.
If $x \in G$ is unbreakable, then $|\bfC_G(x)|\le N$ where $N$ is as in Table $\ref{spbd}$.
\end{lemma}

\begin{table}[htb]
\[
\begin{array}{l|l|l}
\hline 
n & q & N \\
\hline
\hbox{odd} & q>3, q\hbox{ odd} & q^{2n-1}(q^2-1) \\
                     & q>3, q\hbox{ even} & 2q^{2n}(q+1) \\
                      & q=3 & 24 \cdot 3^{2n-2} \\
\hbox{even} & q>3, q\hbox{ odd} & 2q^n \\
                     & q>3, q\hbox{ even} & q^{2n}(q^2-1) \\
                      & q=3 & 48\cdot 3^{2n+1} \\
\hbox{any} & q=2 & 9\cdot 2^{2n+9} \\
\hline
\end{array}
\]
\caption{Upper bounds for symplectic groups} \label{spbd}
\end{table} 

\pf Assume first that $x$ is unipotent and $q$ is odd. By \cite[3.12]{LSei}, $V\downarrow x$ is an orthogonal sum of non-degenerate subspaces of the form $W(m)$ and $V(2m)$, where $x$ acts on $W(m)$ as $J_m^2$, the sum of two Jordan blocks of size $m$, and on $V(2m)$ as $J_{2m}$. Moreover, if $m$ is even then $W(m) \cong V(m)^2$ as $x$-modules. For $q>3$ the symplectic group $\Sp(V(m))$ is perfect for every $m\ge 2$, so the unbreakability of $x$ implies that 
$V\downarrow x$ is either $W(n)$ with $n$ odd, or $V(2n)$. The corresponding orders of $\bfC_G(x)$ are given by \cite[7.1]{LSei}, and the largest are those in Table \ref{spbd} for $q>3$ odd. If $q=3$ then $\Sp_2(3)$ is not perfect, so there are more unbreakable possibilities for $x$:
\[
\begin{array}{l|l}
\hline
V\downarrow x & |\bfC_G(x)| \\
\hline
V(2n) & 2\cdot 3^n \\
V(2n-2)+V(2) & 4\cdot 3^{n+2} \\
W(n) \,(n\hbox{ odd}) & 24\cdot 3^{2n-2} \\
W(n-1)+V(2)\,(n\hbox{ even}) & 48\cdot 3^{2n+1} \\
\hline
\end{array}
\]
Again, the values of $|\bfC_G(x)|$ are given by \cite[7.1]{LSei}, and the largest are those in Table \ref{spbd}.

Next assume $x$ is unipotent and $q$ is even. Again, $V\downarrow x$ is an orthogonal sum of non-degenerate subspaces of the form $W(m)$ and $V(2m)$ (see \cite[Chapter 6]{LSei}). If $q\ge 4$, the unbreakability of $x$ implies that 
$V\downarrow x$ is either $W(n)$ or $V(2n)$. The corresponding orders of $\bfC_G(x)$ are given by \cite[7.3]{LSei}, and the largest are those in Table \ref{spbd} for $q>2$ even. If $q=2$ then neither $\Sp_2(2)$ nor $\Sp_4(2)$ is perfect, so for $n\ge 7$, the possible $V \downarrow x$ for unbreakable $x$ are of the form $X+Y$, where $X = W(n-k)$ or $V(2n-2k)$ and $Y = W(k)$, $V(2k)$ or $V(2)^k$ for some $k\le 2$. By \cite[7.3]{LSei}, the largest centralizer order occurs for $W(n-2)+W(2)$, and is at most $9\cdot 2^{2n+9}$, as in Table \ref{spbd}.

%We have now handled the case where $x$ is unipotent. 
Now suppose $x$ is not unipotent and write $x=su$ with semisimple part $s$ and unipotent part $u$. If $s \in \bfZ(G)$ then the argument for the unipotent case above applies, so assume $s \not \in \bfZ(G)$. Then 
\[
\bfC_G(s) = \Sp_{2r}(q) \times \Sp_{2t}(q) \times \prod \GL_{a_i}^{\e_i}(q^{b_i}),
\]
where $2r,2t$ are the dimensions of the 1- and $-1$- eigenspaces of $s$ (with $t=0$ for $q$ even), and $r+t+2\sum a_ib_i = n$. 

If $q>3$ then the unbreakability of $x$ implies that $r=t=0$ and $a_1b_1 = n$; write $a=a_1,b=b_1$. Moreover, in $\bfC_G(s) = \GL_{a}^{\e}(q^{b})$, $u$ must be a single Jordan block $J_{a}$. So from 
\cite[7.1]{LSei}, 
$|\bfC_G(x)| = |\bfC_{\bfC_G(s)}(u)| = (q^b-\e)q^{b(a-1)} \le q^n+1$, giving the result in this case.

Now consider $q=3$. As $x$ is unbreakable, either $2r$ or $2t$ is equal to $2n-2$, or $a_1b_1 \in \{n-1,n\}$. In the former case, $u = u_1u_2 \in \bfC_G(s) = \Sp_{2n-2}(3) \times H$ with $H = \Sp_2(3)$ or $\GU_1(3)$, and unbreakability forces $V_{2n-2} \downarrow {u_1}$ to be $W(n-1)$ ($n$ even) or $V(2n-2)$. Now \cite[7.1]{LSei} shows that $|\bfC_G(x)|$ is less than the bound in Table \ref{spbd}. In the latter case $u = u_1u_2 \in \bfC_G(s) = \GL_a^\e(q^b) \times H$ with either $ab=n$, $H=1$,  or $ab=n-1$, $H \in \{ \Sp_2(3), \GU_1(3)\}$. If $ab=n$, unbreakability forces $u_1$ to be $J_a$ or $(J_{a-1},J_1)$;  likewise if $ab=n-1$, then $u_1 = J_a$. In either case $|\bfC_G(x)|$ is less than the bound in Table \ref{spbd}.

Finally, suppose $q=2$. Here unbreakability forces either $r \ge n-2$ or $a_1b_1 \ge n-2$. If $r\ge n-2$ then $u = u_1u_2 \in \bfC_G(s) = \Sp_{2r}(2) \times H$ with $H \le \Sp_{2n-2r}(2)$, and $V_{2r}\downarrow u_1$ is $V(2r)$, $W(r)$, or 
$V(2n-4)+V(2)$ ($r=n-1$) or $W(n-2)+V(2)$ ($r=n-1$). The largest possible value of $|\bfC_G(x)|$ is less than the value $9\cdot 2^{2n+9}$ in Table \ref{spbd}. If $a_1b_1 = n-k\ge n-2$ then, writing $a=a_1,b=b_1$, we see that $u = u_1u_2 \in \Sp_{2k}(2) \times \GL_a^\e(q^b)$. The largest value of $|\bfC_G(x)|$ occurs when $a=n, b=1, \e=-1$ and $u = u_2 = (J_{n-2},J_1^2)$; here $\bfC_G(x) = \bfC_{\GU_n(2)}(u)$ again has order less than the bound in Table \ref{spbd}. \hal
 
\begin{lemma}\label{orthogunb}
Let $G = \O(V) = \O_{2n}^\e(q)$ ($n\ge 4$) or $G = \O(V)= \O_{2n+1}(q)$ ($n\ge 3$, $q$ odd), and assume 
further that $\dim V \ge 13$ if $q\le3$.
If $x \in G$ is unbreakable, then $|\bfC_G(x)|\le M$, where $M$ is as in Table $\ref{orthogbd}$.
\end{lemma}

\begin{table}[htb]
\[
\begin{array}{l|l}
\hline 
 q & M \\
\hline
 q>3 & q^{2n-2}(q+1)^2 \\
q=2 & 3\cdot 2^{2n+6} \\
q=3 & 2^6\cdot 3^{2n+4}\,(\dim V = 2n) \\
        & 2^4\cdot 3^{2n+3}\,(\dim V = 2n+1) \\
\hline
\end{array}
\]
\caption{Upper bounds for orthogonal groups} \label{orthogbd}
\end{table} 

\pf First consider the case where $q\geq 4$ is even, so $G = \O_{2n}^\e(q)$. 

Assume $x$ is unipotent. By \cite[Chapter 6]{LSei}, $V \downarrow x$ is an orthogonal sum of non-degenerate subspaces of the form $V(2k)$ (a single Jordan block $J_{2k} \in \GO_{2k}^\e(q)\setminus \O_{2k}^\e(q)$) and $W(k)$ (two singular Jordan blocks $J_k^2 \in \O_{2k}^+(q)$). Since $x$ is unbreakable, $V\downarrow x$ is $W(n)$ or $V(2n-2k)+V(2k)$ for some $k$. The order of $\bfC_G(x)$ is given by \cite[7.1]{LSei}, and the largest value occurs for $W(n)$. It is $q^{2n-3}|\Sp_2(q)|$ for $n$ even, and $q^{2n-2}|\SO_2^\pm(q)|$ for $n$ odd; the former is less than the bound in Table \ref{orthogbd} for $q>3$. 

If $x = su$ is non-unipotent with semisimple part $s$ and unipotent part $u$, then $\bfC_G(s) = \O_{2k}^\d(q) \times \prod \GL_{a_i}^{\e_i}(q^{b_i})$ with $2k = \dim \bfC_V(s)$ and $k+\sum a_ib_i = n$. As each $\GL_{a_i}^{\e_i}(q^{b_i}) \le \O_{2a_ib_i}(q)$, the unbreakability of $x$ implies that either $k \ge n-1$ or $a_1b_1 \ge n-1$. In the former case $u = u_1u_2 \in \bfC_G(s) =  \O_{2n-2}^\d(q) \times \GL_1^\nu(q)$, and as in the previous paragraph $|\bfC_{ \O_{2n-2}^\d(q)}(u_1)|$ is at most $q^{2n-5}|\Sp_2(q)|$, which gives the conclusion. In the latter case $u = u_1u_2 \in 
\bfC_G(s) =  \O_{2k}^\d(q) \times \GL_a^\nu(q^b)$ with $k\le 1$ and $ab = n-k$, and unbreakability forces $u_2 \in  \GL_a^\nu(q^b)$ to be either $J_a$, or $(J_{a-1},J_1)$ with $a=n, b=1$. Then $\bfC_G(x) = \bfC_{\bfC_G(s)}(u)$ has smaller order than the bound in Table \ref{orthogbd}.

Now consider the case where $q \ge 5 $ is odd. 

For $x$ unipotent, $V \downarrow x$ is an orthogonal sum of non-degenerate spaces $W(2k)$ (namely, $J_{2k}^2 \in \O_{4k}^+(q)$) and $V(2k+1)$ (namely, $J_{2k+1} \in \O_{2k+1}(q)$). The unbreakability of $x$ implies that $V\downarrow x = W(n)$ or $V(2n+1)$, giving the conclusion by \cite[7.1]{LSei}. 

For $x=su$ non-unipotent, write 
\[
\bfC_G(s) =  (\O_a(q) \times \O_b(q) \times \prod \GL_{a_i}^{\e_i}(q^{b_i})) \cap G,
\]
where $a = \dim \bfC_V(s), b=\dim \bfC_V(-s)$ and $a+b+\sum 2a_ib_i = \dim V$. As $\GL_r^\e(q) \le \SO_{2r}(q)$ and $s$ has determinant one, $b$ is even. If $a \ne 0$ then $V_a \downarrow u$ is either $W(2k)$ or $V(2k+1)$ and $x$ is breakable. Hence $a=0$. Moreover, $-1 \in \O_{4k}^+(q)$ (see \cite[2.5.13]{KL}), 
so if $u_0$ is a unipotent element of type $W(2k)$, then $-u_0 \in \O_{4k}^+(q)$. Hence by unbreakability, if $b\ne 0$ then either $b=\dim V$ and $V_b\downarrow u = W(n)$, or $V_b \downarrow u$ is a sum of an even number of spaces $V(2k_i+1)$. The former case satisfies the conclusion as above, so assume the latter holds. If there are more than two of the spaces $V(2k_i+1)$, then there exist $i,j$ such that the discriminant of $V(2k_i+1)+V(2k_j+1)$ is a square; if $u_1$ is the projection of $u$ to this space then $-u_1 = -(J_{2k_1+1},J_{2k_j+1}) \in \O_{2k_i+2k_j+2}(q)$, contradicting unbreakability. Hence either $b=0$ or $V_b \downarrow u$ is a sum of two spaces $V(2k_i+1)$. Likewise, the projection of $u$ to a factor 
$\GL_{a_i}^{\e_i}(q^{b_i}))$ has at most two Jordan blocks; here, the only extra point to note is that if $b_i = 1$ and there are three blocks $J_1,J_k,J_l$ with the projection of $s$ to the $J_1$ block giving an element of $\O_2(q)$, then the projection of $s$ to the other blocks gives elements of $\O_{2k}(q),\O_{2l}(q)$, and $x$ is breakable. 

It follows from all these observations together with \cite[7.1]{LSei} that the largest value of $|\bfC_G(x)|$ occurs when either $b=\dim V$ and $V\downarrow u = V(n)^2$ ($n$ odd), or $\bfC_G(s) = \GU_n(q)\cap G$ and $u = (J_{n/2}^2) \in \GU_n(q)$ ($n$ even). In either case $|\bfC_G(x)| \le q^{2n-2}(q+1)^2$, as in Table \ref{orthogbd}. 

Next suppose $q=3$. Following the proof of the $q=3$ case of \cite[5.15]{ore}, for $\dim V = 2n$ the largest possibility for $|\bfC_G(x)|$ is as in Table \ref{orthogbd}, and arises when $x$ is unipotent and $V\downarrow x = W(2)+W(n-2)$; note that the larger bound given in \cite[5.15]{ore} occurs when $x = -u$ with $V\downarrow u = V(1)^4+W(n-2)$, but this element $x$ is breakable according to our definition (which is different from the definition in \cite{ore}). For $\dim V = 2n+1$ the largest value of $|\bfC_G(x)|$ is as in \cite[5.15.]{ore}.

Finally, if $q=2$ the proof of \cite[5.15]{ore} gives the bound in Table \ref{orthogbd}. \hal

\begin{lemma}\label{blocks}
{\rm (i)} Let $q=2$ or $3$, and let $G = Sp(V)$ or $\O(V)$ with the assumptions on $\dim V$ as in Lemmas $\ref{spunb}$ and $\ref{orthogunb}$. Let $\bar V = V \otimes_{\F_q} \bar \F_q$ and let $\a \in \bar \F_q$ satisfy either $\a^{q-1}=1$ or $\a^{q+1}=1$. If $x \in G$ is unbreakable, then $\dim {\rm Ker}_{\bar V}(x-\a I) \le 4$.

{\rm (ii)} Let $q=5$ and let $G = \O(V) = \O^{\pm}_{2n}(5)$ with $n\ge 5$.  Let $\bar V = V \otimes_{\F_q} \bar \F_q$ and let $\a \in \bar \F_q$ satisfy $\a^{q-1}=1$ or $\a^{q+1}=1$. If $x \in G$ is unbreakable, then $\dim {\rm Ker}_{\bar V}(x-\a I) \le 2$.
\end{lemma}

\pf (i) For $\a = \pm 1$ the lemma implies that the number of unipotent Jordan blocks of $\pm x$ is at most 4, which follows from the proofs of Lemmas \ref{spunb} and \ref{orthogunb}. In the other case, $\a$ has order $q+1$. A Jordan block of $x$ on $\bar V$ with eigenvalue $\a$ and dimension $k$ corresponds to a non-degenerate subspace $W$ of $V$ of dimension $2k$ such that $x^W$ lies in $\Sp(W)$ or $\SO(W)$. Hence the unbreakability of $x$ implies that there can be no more than four such blocks. 

\vspace{2mm}
(ii)  If $x = \pm u$ with $u$ unipotent, then the proof of Lemma 3.3 (for the case where $q \ge 5$ is odd) shows that $V\downarrow u$ is $W(n)$ or $V(2k_1+1)+V(2k_2+1)$ for some $k_1,k_2$, giving the result in this case. Now suppose $x=su$ with semisimple part $s \ne \pm 1$, and let $\bfC_G(s) = \O_a(5) \times \O_b(5) \times \prod \GL_{a_i}^{\e_i}(5^{b_i})$ as in Lemma 3.3. That proof shows that $a=0$, $b$ is even, $V_b\downarrow u$ is the sum of zero or two odd-dimensional spaces $V(2k_i+1)$, and the projection of $u$ to each factor  $\GL_{a_i}^{\e_i}(5^{b_i})$ has at most 2 Jordan blocks. The conclusion of (ii) follows. \hal

\subsection{Linear and unitary groups}
\begin{defn}\label{br-slu}
 {\em (i) An element of the general linear group $\GL_n(2)$ is 
{\it breakable} if it lies in a natural subgroup of the form $\GL_a(2) \times \GL_b(2)$ where $a+b=n$, 
$1\le a\le b$ and $a,b \ne 2$. 
 
(ii) An element of the unitary group $\GU_n(2)$ is 
{\it breakable} if it lies in a natural subgroup of the form $\GU_a(2) \times \GU_b(2)$ where $a+b=n$, $1\le a\le b$ and $a,b \ne 2, 3$. 

(iii) An element of the general linear or unitary group $\GL^\e_n(3)$  is 
{\it breakable} if it lies in a natural subgroup of the form $\GL^\e_a(3) \times \GL^\e_b(3)$ where $a+b=n$, $1\le a\le b$ and $a,b \ne 2$.

(iv) If $q \geq 4$, then an element of $\GL^\e_n(q)$  
is {\it breakable} % (or {\it decomposable}) -- EOB never used 
if it lies in a natural subgroup of the form 
$\GL^\e_a(q) \times \GL^\e_b(q)$ where $a+b=n$ and $1\le a\le b$.}
\end{defn} 

%Note that in case (iv) of Definition \ref{br-slu}, 
If $q \geq 4$ and $x \in G = \GL^\e_n(q)$ is unbreakable, then  
\begin{equation}\label{cent-slu}
  |\bfC_G(x)| \leq \left\{ \begin{array}{ll}q^n-1, & \e = +,\\ q^{n-1}(q+1), & \e = -, \end{array} \right.
\end{equation}
%if $x \in G = \GL^\e_n(q)$ is unbreakable 
(cf.\ \cite[Lemma 6.7]{ore} for the case $\e = -$).  
\begin{lemma}\label{LUL-sl2}
If $n\ge 7$ and $x \in G=\GL_n(2)$ is unbreakable, then either 

{\rm (i)} $|\bfC_G(x)| \le 2^{n+2}$, or 

{\rm (ii)} $|\bfC_G(x)| = 9 \cdot 2^n$, $2|n$, and $x \in \GL_{n/2}(4)$.
\end{lemma}

\pf Suppose first that $x$ is unipotent. As it is unbreakable, $x$ has Jordan form $J_n$, or 
$J_{n-2}+J_2$. The order of $\bfC_G(x)$ is given by \cite[7.1]{LSei}, and the maximum possible order is $2^{n+2}$, which occurs in the last case.

Now assume that $x = su$ where $s \ne 1$ is the semisimple part and $u$ the unipotent part of $x$.  Then 
\[
\bfC_G(s) = \prod_i \GL_{a_i}(2^{b_i}),
\]
where $\sum a_ib_i = n$. Moreover, since $x \in \bfC_G(s)$ is unbreakable, we may assume 
$a_1b_1 \in \{n,n-2\}$, and write $a=a_1,b=b_1$. If $ab=n$ then $b\geq 2$. A Jordan block 
$J_c$ of $u$ as an element of $\GL_a(2^b)$ lies in a natural subgroup $\GL_{cb}(2)$, so the unbreakability of $x$ forces the Jordan form of $u$ in $\GL_a(2^b)$ to be $J_a$ or $J_{a-1}+J_1$ (with $b=2$ in the latter case). By \cite[7.1]{LSei}, $|\bfC_G(x)| = |\bfC_{\GL_a(2^b)}(u)|$ is
$2^{b(a-1)}(2^b-1) < 2^n$ in the former case, and it is 
$2^{ab}\cdot |\GL_1(2^b)|^2 = 9 \cdot 2^n$ in the latter case, in which case also 
$2|n$ and $x \in \bfC_G(s) = \GL_{n/2}(4)$. If $ab = n-2$, then 
$\bfC_G(s) \leq \GL_a(2^b) \times \GL_2(2)$ and the Jordan form of $u$ in the first factor must be 
$J_a$, whence 
\[
|\bfC_G(x)| \le 2^{b(a-1)}|\GL_1(2^b)||\GL_2(2)| = (2^{n-2}-2^{n-2-b})\cdot 6 < 2^{n+2},
\]
giving the result in this case.  \hal

\begin{lemma}\label{LUL}
If $x \in G=\GU_n(2)$ is unbreakable, then $|\bfC_G(x)| \le 2^{n+4}\cdot 3^2$ if 
$n \geq 10$ and $|\bfC_G(x)| \leq 2^{48}$ if $n = 9$.
\end{lemma}

\pf 
(i) Consider the case $n \geq 10$.
Suppose first that $x$ is unipotent. As it is unbreakable, $x$ has Jordan form $J_n$, $J_{n-2}+J_2$ or $J_{n-3}+J_3$. The order of $\bfC_G(x)$ is given by \cite[7.1]{LSei}, and the maximum possible order is $2^{n+4}\cdot 3^2$, which occurs in the last case.

Suppose that $x = su$ where $s \ne 1$ is the semisimple part and $u$ the unipotent part of $x$. If $s \in \bfZ(G)$ then the argument of the previous paragraph applies. If $s \not \in \bfZ(G)$, then 
\[
\bfC_G(s) = \prod \GU_{a_i}(2^{b_i}) \times \prod \GL_{c_i}(2^{2d_i}) \le \prod \GU_{a_ib_i}(2) \times \prod \GU_{2c_id_i}(2),
\]
where $\sum a_ib_i + 2\sum c_id_i = n$, and all $b_i$ are odd. Moreover, since $x \in \bfC_G(s)$ is unbreakable, either $a_1b_1$ or $2c_1d_1$ lies in the set $\{n,n-2,n-3\}$.

Suppose $a_1b_1 \in \{n,n-2,n-3\}$, and write $a=a_1,b=b_1$. If $ab=n$ then $b > 1$
since $s \not \in \bfZ(G)$, so $b\ge 3$ (as $b$ is odd). A Jordan block $J_c$ of $u$ as an element of $\GU_a(2^b)$ lies in a natural subgroup $\GU_{cb}(2)$, so the unbreakability of $x$ forces the Jordan form of $u$ in $\GU_a(2^b)$ to be $J_a$ or $J_{a-1}+J_1$ (with $b=3$ in the latter case). By \cite[7.1]{LSei}, the largest possible value of $|\bfC_G(x)| = |\bfC_{\GU_a(2^b)}(u)|$ occurs in the latter case, 
and is 
$2^{ab}\cdot |\GU_1(2^b)|^2 = 2^n\cdot 9^2$, proving the result in this case. If $ab = n-2$, then $\bfC_G(s) = \GU_a(2^b) \times \GU_2(2)$ and the Jordan form of $u$ in the first factor must be $J_a$, whence 
\[
|\bfC_G(x)| \le 2^{b(a-1)}|\GU_1(2^b)||\GU_2(2)| = (2^{n-2}+2^{n-2-b})\cdot 18 < 2^n\cdot 3^2,
\]
giving the result in this case. Similarly, if $ab=n-3$ then 
\[
\begin{array}{ll}
|\bfC_G(x)| & \le |\bfC_{\GU_a(2^b)}(J_a)||\GU_3(2)| = 2^{b(a-1)}(2^b+1)\cdot 2^33^4 \\
               & =  (2^{n-3}+2^{n-3-b})\cdot 2^33^4 < 2^{n+4}\cdot 3^2.
\end{array}
\]

Now suppose $2c_1d_1 \in \{n,n-2,n-3\}$, and write $c=c_1,d=d_1$. If $d=1$ then the projection of $s$ in $\GL_c(2^{2d})$ is a central element of order 3 which is central in a natural subgroup $\GU_{2c}(2)$, so $\bfC_G(s)$ has a factor $\GU_{2c}(2)$ rather than $\GL_c(2^2)$. Hence $d>1$. As above, the unbreakability of $x$ forces $u$ to have Jordan form $J_c$ as an element of $\GL_c(2^{2d})$. Hence 
\[
|\bfC_G(x)| \le |\bfC_{\GL_c(2^{2d})}(J_c)|\cdot |\GU_{n-2cd}(2)|,
\]
which is a maximum when $cd = n-3$, in which case $|\bfC_G(x)| \le 2^{2d(c-1)}(2^{2d}-1)\cdot |\GU_3(2)|$ which is less than $2^n\cdot 3^4$. This completes the proof. 

\smallskip
(ii) Suppose now that $n = 9$. Assume first that $x = su$ where $s \in \bfZ(G)$ and $u$ is unipotent. As $x$ is unbreakable, $u$ has Jordan form $J_9$, $J_7+J_2$, $J_6+J_3$ or $J_3^3$. The largest centralizer is that of $J_3^3$, which has order $2^{18}|\GU_3(2)|$, less than $2^{48}$.

Now suppose $x=su$ with semisimple part $s \not \in \bfZ(G)$. Then $\bfC_G(s)$ is as described above. Assuming that $|\bfC_G(x)|\ge 2^{48}$, the only possibility is that $\bfC_G(s) = \GU_7(2)\times \GU_2(2)$ (note that 
$\GU_8(2)\times GU_1(2)$ is not possible as this would imply that $x$ is breakable). If 
$|\bfC_G(x)| = |\bfC_{C_G(s)}(u)|\ge 2^{48}$, then $u$ projects to the identity in $\GU_7(2)$; but then $x$ is breakable, a contradiction.
\hal

\begin{lemma}\label{LUL3}
If $n\ge 7$ and $x \in G=\GL^\e_n(3)$ is unbreakable, then $|\bfC_G(x)| \le 3^{n+2}\cdot 2^4$.
\end{lemma}

\pf For $x$ unipotent the largest centralizer occurs when $x = (J_{n-2},J_2)$ and has order $ 3^{n+2}\cdot 2^4$ by \cite[7.1]{LSei}. 

Suppose $x = su$ is non-unipotent. If $s \in \bfZ(G)$ the bound of the previous paragraph applies, so assume $s \not \in \bfZ(G)$. The possibilities for $\bfC_G(s)$ are: 
\[
\begin{array}{l}
\e = +:\; \bfC_G(s) = \prod \GL_{a_i}(3^{b_i}) \\
\e=-:\; \bfC_G(s) =  \prod \GU_{a_i}(3^{b_i}) \times \prod \GL_{c_i}(3^{2d_i})
\end{array}
\]
where $\sum a_ib_i = n$ for $\e=+$, and 
$\sum a_ib_i + 2\sum c_id_i = n$ and all $b_i$ are odd for $\e=-$.
As in the previous proof, the unbreakability assumption implies that $a_1b_1 \in \{n-2,n\}$ for $\e=+$, and  either $a_1b_1$ or $2c_1d_1$ is in $\{n-2,n\}$ for $\e=-$. Now we argue as in the previous lemma that none of the possibilities for $u \in \bfC_G(s)$ give a larger centralizer order than $ 3^{n+2}\cdot 2^4$. 
\hal

\section{Theorem \ref{main1} for linear and unitary groups}
\subsection{General inductive argument}
Recall $\cR(S)$ from \S2, and the notion of unbreakability from Definition \ref{br-slu}. 

\begin{defn}\label{cond-slu}
{\em Given a prime power $q = p^f$, $\e = \pm$, and an integer $N = p^at^b$ with $t \nmid (q-\e)$ 
a prime. We say that $G = \GL^\e_n(q)$ satisfies 

{\rm (i)} the condition $\sP(N)$ if every 
$g \in G$ can be written as $g=x^Ny^N$ for some $x,y \in G$ with $x^N \in \SL^\e_n(q)$; and

{\rm (ii)} the condition $\sP_u(N)$ if every {\it unbreakable} 
$g \in G$ can be written as $g=x^Ny^N$ for some $x,y \in G$ with $x^N \in \SL^\e_n(q)$.}
\end{defn}

First we prove an extension of Theorem \ref{prime1} for $\GL^\e_n(q)$:

\begin{prop}\label{slu-generic}
Let $G = \GL^\e_n(q)$ with $n \geq 4$, $q = p^f$, and let $t \nmid p(q-\e)$ be a prime not contained in 
$\cR(\SL^\e_n(q))$. Then $\sP(N)$ holds for $G$ and for all $N = p^at^b$.   
\end{prop}

\pf 
(i) First we consider the generic case: $\cR(\SL^\e_n(q)) = \{r,s_1=s_2\}$ and 
$r$ and $s=s_1=s_2$ are listed in Table \ref{primes}. In particular, $r = \p(q,n)$ and 
$s_1 = \p(q,n-1)$ when $\e = +$. When $\e = -$, interchanging $r$ and $s$ if necessary, we may assume that
$r$ divides $q^n-\e^n$ but not
$\prod^{n-1}_{i=1}(q^i-\e^i)$ (so $r$ is a primitive prime divisor of $(\e q)^n-1$), and similarly,
$s$ divides $q^{n-1}-\e^{n-1}$ but not $\prod_{1 \leq i \leq n,~i \neq n-1}(q^i-\e^i)$.

Since $N$ is coprime to $q-\e$, every central element of $G$ can be written as an $N$th power. 
So it suffices to prove $\sP(N)$ for every non-central $g \in G$.
Fix a regular semisimple $g_1 \in G$ of order $r$, in particular $\det(g_1) = 1$, 
and a regular semisimple $h \in \GL^\e_{n-1}(q)$ of order $s$. We can choose 
$d\in \GL^\e_1(q)$ such that $\det(g_2) = \det(g)$ for $g_2 := \diag(h,d)$.  Since both $g_1$ and 
$g_2$ have order coprime to $N$, it suffices to show that $g \in g_1^G \cdot g_2^G$. To this end
we apply Lemma \ref{basic}(i).

Consider a character $\chi \in \Irr(G)$ with $\chi(g_1)\chi(g_2) \neq 0$. It follows that 
$\chi(1)$ is neither of $r$-defect $0$ nor of $s$-defect $0$. On the other hand, the order
of the centralizer of every non-central semisimple element of $\GL^\e_n(q)$ is either coprime to
$r$ or coprime to $s$. Hence the Lusztig classification of 
irreducible characters of $G$ \cite{DL} implies that $\chi$ belongs to the rational series $\cE(G,(z))$
labeled by a central semisimple $z \in G^* \cong G$. It follows that $\chi = \lambda\psi$,
where $\lambda(1) = 1$ and $\psi$ is a unipotent character of $G$. Moreover, as shown in
the proof of \cite[Theorems 2.1--2.2]{MSW}, $\psi$ is either $1_G$ or $\St$, the Steinberg character of 
$G$. Since $\det(g_1) = 1$ and $\det(g_2) = \det(g)$ by our choice, $\lambda(g_1) = 1$ and
$\lambda(g_2)\overline\lambda(g) = 1$ for all linear $\lambda \in \Irr(G)$. Finally, since 
$g \notin \bfZ(G)$ and $|\St(g_i)| = 1$, 
$$\sum_{\chi \in \Irr(G)}\frac{\chi(g_1)\chi(g_2)\oc(g)}{\chi(1)}
   = (q-\e)\left(1+\frac{\St(g)}{\St(1)}\right) > 0,$$
so we are done.
 
\smallskip
(ii) The same arguments apply to the non-generic cases
$$(n,q,\e) = (4,4,+),~(6,4,-),(7,4,-),$$
if we choose $\cR(\SL^\e_n(q))$ to be $\{17,7\}$, $\{41,7\}$, or $\{113,7\}$, respectively.
In the remaining cases
$$(n,q,\e) = (6,2,+),~(7,2,+),~(4,2,-),$$
the statement follows from \cite[Lemma 2.12]{GT3} if we choose $\cR(\SL^\e_n(q))$ to be 
$\{31\}$, $\{127\}$, or $\{5\}$, respectively (note that $\GU_4(2) \cong C_3 \times \SU_4(2)$). 
\hal

Our proof of Theorem \ref{main1} for linear and unitary groups relies on the following inductive argument: 

\begin{prop}\label{ind-slu}
Fix a prime power $q = p^f$, an integer $n \geq 4$, and $\e = \pm$. Suppose that there is an integer 
$n_0 \geq 3$ such that the following statements hold:
\begin{enumerate}[\rm(i)]
\item Let $1 \leq k \leq n_0$ with $k \neq 2$ if $q = 2,3$, and $k \neq 3$ if $(q,\e) = (2,-)$. Then  
$\sP_u(N)$ holds for $\GL^\e_k(q)$ for every $N = p^at^b$ with $t$ prime and 
$t \nmid p(q-\e)$.
\item For each $k$ with $n_0 < k \leq n$, $\sP_u(N)$ holds for $\GL^\e_k(q)$ and for every 
$N = p^at^b$ with $t \in \cR(\SL^\e_k(q))$.
\end{enumerate}
If $N = s^at^b$ for some primes $s,t$, then the word map $(u,v) \mapsto u^Nv^N$ is surjective on 
$\PSL^\e_n(q)$.  
\end{prop}

\pf
By Corollary \ref{prime2}, we need to consider only the case $N = p^at^b$ with 
$t \in \cR(\SL^\e_n(q))$; in particular, $t \nmid (q-\e)$. 
It suffices to show $\sP(N)$ holds for $G := \GL^\e_n(q)$ and this 
choice of $N$. Indeed, in this case every 
$g \in \SL^\e_n(q)$ can be written as $x^Ny^N$ with $\det(x^N) = \det(y^N) = 1$. Since 
$\gcd(N,q-\e) = 1$, it follows that $x, y \in \SL^\e_n(q)$.

By (ii), $\sP_u(N)$ holds for $G$. Consider a 
breakable $g \in G$  and write it as $\diag(g_1, \ldots ,g_m)$ lying in the 
natural subgroup
$$\GL^\e_{k_1}(q) \times \ldots \times \GL^\e_{k_m}(q).$$
Here, $1 \leq k_i < n$, and if $k_i \leq n_0$ then $k=k_i$ fulfills the conditions imposed on $k$ in (i).
Furthermore, each $g_i$ is unbreakable.
Hence, according to (i), $\sP_u(N)$ holds for $\GL^\e_{k_i}(q)$ if $k_i \leq n_0$. 
If $k_i > n_0$, then by (ii) and Proposition \ref{slu-generic}, $\sP_u(N)$ holds for $\GL^\e_{k_i}(q)$
as well. Thus we can write $g_i = x_i^Ny_i^N$ with $x_i,y_i \in \GL^\e_{k_i}(q)$ and $\det(x_i^N) = 1$. Letting
$$x := \diag(x_1, \ldots ,x_m),~~y := \diag(y_1, \ldots, y_m)$$
we deduce that $g = x^Ny^N$ and $\det(x^N) = 1$. Thus $\sP(N)$ holds for $G$, as desired. 
\hal

%In fact, the above proof also establishes the following criterion that we use to prove 
%$\sP(N)$ for some small rank groups:
%
%\begin{lem}\label{ind-slu-small}
%Given a prime power $q = p^f$, an integer $n$, $\e = \pm$, and an integer $N = p^at^b$ 
%with $t$ a prime and $t \nmid p(q-\e)$. Suppose that there is an integer 
%$n_0 \leq n$ such that the following statements hold:
%
%\begin{enumerate}[\rm(i)]
%\item Suppose $1 \leq k \leq n_0$, $k \neq 2$ if $q = 2,3$, and $k \neq 3$ if $(q,\e) = (2,-)$. Then  
%$\sP(N)$ holds for $\GL^\e_k(q)$; and
%\item For each $k$ with $n_0 < k \leq n$, $\sP_u(N)$ holds for $\GL^\e_k(q)$.
%\end{enumerate}
%Then, $\sP(N)$ holds for $\GL^\e_n(q)$.
%\end{lem}

\subsection{Induction base}
\begin{lem}\label{slu-small2}
Let $q = p^f\geq 2$, $\e = \pm$, and $N = r^as^b$ for some primes $r,s$. Suppose 
that $S = \PSL^\e_k(q)$ is simple and $k = 2$ or $3$. Then the map 
$(u,v) \mapsto u^Nv^N$ is surjective on $S$. 
\end{lem}

\pf
By Corollary \ref{prime2}(i), we need to consider only the case $N = p^as^b$. Let
$S = \PSL_3(q)$. By \cite[Theorem 7.3]{GM}, $S \setminus \{1\} \subseteq CC$ where 
$C = x^S$ or $y^S$, $|x| = (q^2+q+1)/d$ and $|y| = (q^2-1)/d$, with $d = \gcd(3,q-1)$. 
In particular, $|x|$ and $|y|$ are coprime. Hence at least one of $x,y$ has order coprime 
to $N$, so it is an $N$-power in $S$, whence we are done. $\PSU_3(q)$ can be 
treated similarly using \cite[Theorem 7.1]{GM}. If $S = \PSL_2(q)$ with $q \geq 7$ odd, then 
by \cite[Theorem 7.1]{GM},  $S \setminus \{1\} \subseteq CC$ with $C = x^S$ or $y^S$,
$|x| = (q+1)/2$ and $|y| = (q-1)/2$, so we can argue similarly. Finally, assume that
$S = \SL_2(q)$ with $q \geq 4$ even.  If $s \nmid (q-1)$, then $S \setminus \{1\} \subseteq CC$ 
with $C = x^S$ and $|x| = q-1$ by \cite[Theorem 7.1]{GM}, so we are done. Assume 
$s|(q-1)$. Using the character table, we check that $S \setminus \{1\} \subseteq y^S \cdot (y^2)^S$ 
if $|y| = q+1$, so we are done again. 
\hal

\begin{lem}\label{slu-small1}
Let $q = p^f\geq 4$, $\e = \pm$, and $N = p^at^b$ for a prime $t \nmid p(q-\e)$. Then $\sP_u(N)$
holds for $G = \GL^\e_k(q)$ with $1 \leq k \leq 3$. 
\end{lem}

\pf
Clearly the statement holds for $k = 1$. Suppose $k > 1$ and 
let $g \in G$ be unbreakable. Let $\rho \in \F_q^\times$ and 
let $\ve \in \C^\times$ have order $q-1 \geq 3$. To establish $\sP_u(N)$ for $g$, we exhibit some $N'$-elements
$g_1, g_2$ of $G$ such that $g \in g_1^G \cdot g_2^G$ and at least one of $g_1,g_2$ has determinant $1$.
 
\smallskip
(i) Consider the case $G = \GL_2(q)$.
Since $g$ is unbreakable, it belongs to class $B_1$ or $A_2$, in the notation of \cite{St}.
In the first case, $g$ lies in
a torus of order $q^2-1$, and we define $g_1 = \diag(\rho,\rho^{-1})$, and 
$g_2 = \diag(1,\rho^i)$ if $\det(g) = \rho^i \neq 1$, or $g_2 = g_1$ if $\det(g) = 1$. 
Using \cite[Table II]{St}, it is easy to check that  
$$\sum_{\chi \in \Irr(G)}\frac{\chi(g_1)\chi(g_2)\oc(g)}{\chi(1)}
   = (q-1)\left(1+\frac{1}{q}\right) > 0.$$
Since $g_1$ and $g_2$ are $N'$-elements, we are done.   
Suppose now that $g \in A_2$, i.e.\ $g = zu$ with $z \in \bfZ(G)$ and $u$ a regular unipotent element.
Since $z$ is the $N$th power of some central element of $G$, it suffices to show that 
$u \in g_1^G \cdot g_2^G$ where we again choose $g_2 = g_1$. Using \cite[Table II]{St}, 
$$\sum_{\chi \in \Irr(G)}\frac{\chi(g_1)\chi(g_2)\oc(g)}{\chi(1)}
   = (q-1)\left(1-\frac{1}{2(q+1)}\sum_{0 \leq m \neq n \leq q-2}(\ve^{m-n}+\ve^{n-m})^2\right) = 
   \frac{4(q-1)}{q+1},$$
so we are done again.

The same arguments apply in the case $G = \GU_2(q)$, where we choose $g_2 = g_1^2$ if 
$g = zu$ and $u$ is a regular unipotent element.   

\smallskip
(ii) Consider the case $G = \GL_3(q)$, Since $g$ is unbreakable, $g$ belongs to class $C_1$
(so $g$ lies in a maximal torus of order $q^3-1$) or  
$A_3$ (i.e.\ $g$ is a scalar multiple of a regular unipotent element), 
in the notation of \cite{St}. First suppose that $t \neq \p(q,3)$. By Lemma \ref{slu-det} (below) we 
can find a regular semisimple $g_1 \in \GL_3(q)$ of order $\p(q,3)m$ 
such that $\det(g_1) = \det(g)$ and all prime divisors of $m$ divide $q-1$. Note that 
$g_1$ belongs to class $C_1$.
Also, define $g_2 = \diag(1,\rho,\rho^{-1}) \in \SL_3(q)$ belonging to class $A_6$. 
Using \cite[\S3]{St}, it is easy to check that  
$$\left|\sum_{\chi \in \Irr(G)}\frac{\chi(g_1)\chi(g_2)\oc(g)}{\chi(1)}\right|
   > (q-1)\left(1-\frac{2}{q(q+1)}-\frac{1}{q^3}\right) > 0.$$
Since $g_1$ and $g_2$ are $N'$-elements, we are done.   
Suppose now that $t = \p(q,3)$. We choose $h$ to be a regular semisimple element of 
order $q+1$ in $\SL_2(q)$ and define $g_1 := \diag(h,\det(g))$ so that it belongs to class $B_1$.
Using $g_2$ as in (i), we observe that 
$$\left|\sum_{\chi \in \Irr(G)}\frac{\chi(g_1)\chi(g_2)\oc(g)}{\chi(1)}\right|
   > (q-1)\left(1-\frac{1}{q^3}-\frac{3(q-2)}{2(q^2+q+1)} \right) > 0,$$
so we are done again.
 
The same arguments apply in the case $G = \GU_3(q)$.  
\hal

\subsection{Induction step: Generic case}

We need the following simple observation:

\begin{lem}\label{slu-det}
Let $G = \GL^\e_n(q)$ with $n \geq 3$ and let $T$ be a cyclic torus of order $q^n-\e^n$ of $G$.
Suppose there is a prime $s$ that divides $q^n-\e^n$ but not $\prod^{n-1}_{i=1}(q^i-\e^i)$. 
For every $g \in G$, 
%we can find 
there exists a regular semisimple $h \in T$ of order $sm$ for some $m \in \N$ such that 
$\det(h) = \det(g)$ and all prime divisors of $m$ divide $q-\e$. 
\end{lem}

\begin{proof}
Let $D \cong C_{q-\e}$ denote the image of $G$ under the determinant map $\det$. Note 
that $\det$ maps $T$ onto $D$. The condition on $s$ implies that every $x \in T$ of 
order divisible by $s$ is regular semisimple and $s \nmid (q-\e)$.  It follows that 
$\det$ maps $T_1 \geq \bfO_s(T)$ into $1$ and $T_2$ onto $D$, where $T = T_1 \times T_2$,
$|T_1|$ is coprime to $q-\e$, and all prime divisors of $|T_2|$ divide $q-\e$. Hence we can
choose $x \in \bfO_s(T)$ of order $s$ and $y \in T_2$ such that $\det(y) = \det(g)$ and 
set $h := xy$.
\end{proof}

\begin{prop}\label{sl-large}
Suppose $G = \GL_n(q)$ with $n \geq 4$, $q = p^f \geq 4$, and $t \in \cR(\SL_n(q))$. 
Then $\sP_u(N)$ holds for $G$ and for every $N = p^at^b$. 
\end{prop}

\pf
Consider an unbreakable $g \in G$ and a regular semisimple $g_1 \in \SL_n(q)$ of order $s \in \cR(G) \setminus \{t\}$. Denote
$$\Irr(G/[G,G]) = \{ \lambda_i \mid 0 \leq i \leq q-2\}.$$    

\smallskip
(i) First we consider the case $n \geq 6$. Choose 
$$D = \frac{(q^n-1)(q^{n-1}-q^2)}{(q-1)(q^2-1)}.$$
By \cite[Theorem 3.1]{TZ1}, every irreducible character of $\SL_n(q)$ of degree less than $D$ is either 
the principal character, or an irreducible Weil character, and it is well known that each of these characters 
extends to $G$. It follows that the characters in $\Irr(G)$ of degree less than $D$ are exactly the 
$q-1$ linear characters $\lambda_i$ and $(q-1)^2$ irreducible Weil characters 
$\tau_{i,j}$, $0 \leq i,j \leq q-2$, where 
$$\tau_{i,j} = \lambda_j\tau_{i,0},~\tau_{i,0}(1) = \frac{q^n-1}{q-1} - \delta_{i,0}.$$ 
Using Lemma \ref{slu-det}, we can choose  a regular semisimple $g_2 \in G$ of order $sm$ where all prime divisors of $m$ divide $q-1$ and $\det(g_2) = \det(g)$. In particular, 
$$|\bfC_G(g_i)| \leq q^n-1,$$ 
and
\begin{equation}\label{sl41}
  \sum^{q-2}_{i=0}\lambda_i(g_1)\lambda_i(g_2)\overline\lambda_i(g) = q-1.
\end{equation}  
By (\ref{cent-slu}) and Lemma \ref{basic},
\begin{equation}\label{sl42}
  \left|\sum_{\chi \in \Irr(G),~\chi(1) \geq D}\frac{\chi(g_1)\chi(g_2)\oc(g)}{\chi(1)}\right|
    \leq \frac{(q^n-1)^{3/2}}{D} 
    \leq (q-1)\left(1-\frac{1}{q^2+q+1} \right).
\end{equation}
Fix a primitive $(q-1)$th root of unity $\delta \in \F_q^\times$ and a primitive $(q-1)$th root of unity 
$\td \in \C^\times$. Relabeling $\tau_{i,j}$ if necessary, 
$$\tau_{i,0}(x) = \frac{1}{q-1}\sum^{q-2}_{l=0}\td^{il} q^{e(x,\delta^l)}-2\delta_{i,0}$$
for every $x \in G$, where $e(x,\alpha)$ denotes the dimension of the $\alpha$-eigenspace of 
$x$ on the natural module $\F_q^n$ for $G$. The choice of $g_i$ and the unbreakability of $g$ 
ensure that $e(y,\delta^l)$ is at most $1$ 
for $y \in \{g,g_1,g_2\}$ and $0 \leq l \leq q-2$, 
and in fact it can equal $1$ for at most one value $l_0$. In particular,
$$-1 = \frac{1}{q-1}(q-1)-2 \leq \tau_{0,0}(y) \leq \frac{1}{q-1}(q+q-2)-2 = 0.$$
Consider $i > 0$. If such $l_0$ exists, then
$$\tau_{i,0}(y) = \frac{1}{q-1}\left(\td^{il_0}(q-1)+\sum^{q-2}_{l=0}\td^{il}\right) = \td^{il_0}.$$
If no such $l_0$ exists, then   
$$\tau_{i,0}(y) = \frac{1}{q-1}\sum^{q-2}_{l=0}\td^{il} = 0.$$
We have shown that 
\begin{equation}\label{weil-sl1}
  |\tau_{i,j}(y)| \leq 1. 
\end{equation}  
It follows that if $n \geq 5$ then  
\begin{equation}\label{weil-sl2}
  \left|\sum_{0 \leq i,j \leq q-2}\frac{\tau_{i,j}(g_1)\tau_{i,j}(g_2)\overline\tau_{i,j}(g)}{\tau_{i,j}(1)}\right|
    \leq \frac{(q-1)^3}{q^n-q} < \frac{q-1}{q(q^2+q+1)}.
\end{equation}    
Together with (\ref{sl42}), this implies that
$$\left|\sum_{\chi \in \Irr(G),~\chi(1) > 1}\frac{\chi(g_1)\chi(g_2)\oc(g)}{\chi(1)}\right|
    \leq (q-1)\left(1-\frac{1}{q^2+q+1} + \frac{1}{q(q^2+q+1)}\right) < q-1.$$  
Hence $g \in g_1^G \cdot g_2^G$ by (\ref{sl41}) and Lemma \ref{basic}(i). 
Since both $g_1$ and $g_2$ have order coprime to $N$, we are done. 

\smallskip
(ii) Next we consider the case $n = 5$. Setting $s' := \p(q,3)$ and using Lemma \ref{slu-det},
 we can choose a regular semisimple 
$h \in \GL_3(q)$ of order $s'm$, where all prime divisors of $m$ divide $q-1$ and 
$\det(h) = \det(g)$. Also, let $h' \in \GL_2(q)$ be conjugate (over $\overline\F_q$) to 
$\diag(\beta,\beta^{-1})$, where $\beta \in \overline\F_q^\times$ has order $q+1$.  Setting
$g_2 = \diag(h,h')$, the orders of $g_1$ and $g_2$ are coprime to $N$, 
$\det(g_2) = \det(g)$, and 
$e(g_2,\delta^l) = 0$ for $0 \leq l \leq q-2$.   In particular, (\ref{weil-sl1}) and 
(\ref{weil-sl2}) hold. Next, we choose $D = q^4(q^5-1)/(q-1)$, yielding
\begin{equation}\label{sl43}
  \left|\sum_{\chi \in \Irr(G),~\chi(1) \geq D}\frac{\chi(g_1)\chi(g_2)\oc(g)}{\chi(1)}\right|
    \leq \frac{(q^5-1)^{3/2}}{D} 
    \leq \frac{q-1}{q^{3/2}} \leq \frac{q-1}{8}.
\end{equation}

Now, using \cite{Lu}, we check that if $\psi \in \Irr(\SL_5(q))$ has positive $s$-defect and 
positive $s'$-defect and $\psi(1) < D$, then either $\psi$ is the principal character or a Weil character, 
or $s = \p(q,4)$ and
$\psi$ is the unique character of degree $q^2(q^5-1)/(q-1)$. In either case, $\psi$ extends to 
$G$. In fact, in the latter case, an extension $\varphi$ of $\psi$ to $G$ is the unipotent character
labeled by the partition $(3,2)$ (see \cite[\S13.8]{C}). On the other hand, $\tau_{0,0}$ is the unipotent 
character of $G$ labeled by the partition $(4,1)$. It follows by \cite[Lemma 5.1]{GT1} that
$$\varphi = (1_G + \tau_{0,0} + \varphi) - (1_G + \tau_{0,0}) = \rho_2 - \rho_1,$$
where $\rho_i$ is the permutation character of the action of $G$ on the set of $i$-dimensional 
subspaces of the natural module $\F_q^5$ for $i = 1,2$. Therefore,
$$\varphi(g_1) = \rho_2(g_1) - \rho_1(g_1) = 0-1 = -1,~~
    \varphi(g_2) = \rho_2(g_2) - \rho_1(g_2) = 1-0 = 1.$$
Also, the extensions of $\psi$ to $G$ are $\varphi\lambda_i$, $0 \leq i \leq q-2$, and 
$|\varphi(g)| \leq (q^5-1)^{1/2}$ by (\ref{cent-slu}). Certainly, $\chi(g_1)\chi(g_2) = 0$ 
unless $\chi$ has positive $s$-defect and positive $s'$-defect. Hence, combining with 
(\ref{weil-sl2}), we deduce that 
$$\left|\sum_{\chi \in \Irr(G),~1 < \chi(1) < D}\frac{\chi(g_1)\chi(g_2)\oc(g)}{\chi(1)}\right|
    \leq \frac{q-1}{q(q^2+q+1)} + (q-1)\frac{(q^5-1)^{1/2}}{\psi(1)} 
    \leq \frac{q-1}{32}.$$
Together with (\ref{sl43}), this implies that
$$\left|\sum_{\chi \in \Irr(G),~\chi(1) > 1}\frac{\chi(g_1)\chi(g_2)\oc(g)}{\chi(1)}\right|
    \leq (q-1)\left(\frac{1}{8} + \frac{1}{32}\right) < q-1,$$
so we are done as before.      

\smallskip
(iii) Here we consider the case $n = 4$. Since $g$ is unbreakable, $g$ belongs to 
class $A_5$, $C_2$, or $E_1$, in the notation of \cite{St}. In the two latter cases, note that the 
$G$-conjugacy class of such an element $g$ is completely determined
by $|g|$ and the eigenvalues of $g$ acting on 
$\overline\F_q^4$. On the other hand, $G$ contains a natural subgroup 
$H \cong \GL_2(q^2)$, and $H$ contains an element $h$ with the same spectrum and 
order as $g$. Hence we may assume $g =h \in H$. As $N = p^at^b$ and $t \nmid (q^2-1)$, we 
can now apply Lemma \ref{slu-small1} (if $h$ is unbreakable) to get $g = x^Ny^N$ for some 
$x \in \SL_2(q^2) < \SL_4(q)$ and $y \in H$. Such a decomposition certainly exists if $h$ is 
breakable in $H$ (i.e.\ $h \in \GL_1(q^2) \times \GL_1(q^2)$).

It remains therefore to consider the case $g \in A_5$, i.e.\ $g = zu$, where $z \in \bfZ(G)$ and
$u$ is a regular unipotent element. By \cite[Corollary 8.3.6]{C}, $|\chi(g)| \leq 1$ for all
$\chi \in \Irr(G)$. 
Choosing $D = (q-1)(q^3-1)$ and $g_2$ of order $sm$ as in (i), 
by the Cauchy-Schwarz inequality, 
$$ \left|\sum_{\chi \in \Irr(G),~\chi(1) \geq D}\frac{\chi(g_1)\chi(g_2)\oc(g)}{\chi(1)}\right|
    \leq \frac{(q^4-1)}{(q-1)(q^3-1)} < 1.35.$$
Using \cite{St}, we check that all irreducible characters of $G$ of degree less than $D$ are
linear or Weil characters. Hence (\ref{weil-sl1}) implies that
$$\left|\sum_{\chi \in \Irr(G),~1 < \chi(1) < D}\frac{\chi(g_1)\chi(g_2)\oc(g)}{\chi(1)}\right|
    \leq \frac{(q-1)^{3}}{q^4-q} < 0.11.$$
It follows that        
$$\left|\sum_{\chi \in \Irr(G),~\chi(1) > 1}\frac{\chi(g_1)\chi(g_2)\oc(g)}{\chi(1)}\right|
    < 1.35 + 0.11 = 1.46 < q-1,$$     
so we are done.
\hal

\begin{prop}\label{su-large}
Suppose $G = \GU_n(q)$ with $n \geq 4$, $q = p^f \geq 4$, and $t \in \cR(\SU_n(q))$. 
Then $\sP_u(N)$ holds for $G$ and for every $N = p^at^b$. 
\end{prop}

\pf
Consider an unbreakable $g \in G$ and a regular semisimple $g_1 \in \SU_n(q)$ of order $s \in \cR(G) \setminus \{t\}$. Denote
$$\Irr(G/[G,G]) = \{ \lambda_i \mid 0 \leq i \leq q\}.$$    

\smallskip
(i) First we consider the case $n \geq 6$. If $n \geq 7$, then using Lemma \ref{slu-det}, we can choose  a regular semisimple $g_2 \in G$ of order $sm$ where all prime divisors of $m$ divide $q+1$ and 
$\det(g_2) = \det(g)$. If $n = 6$, then we set $s' := \p(q,6) \geq 7$ and use Lemma \ref{slu-det} to get a 
regular semisimple $h \in \GU_3(q)$ of order $s'm$, where $\det(h) = \det(g)$ and 
all prime divisors of $m$ divide $q+1$. We also set $h':= (h_{s'})^{-1}$ and $g_2 = \diag(h,h')$. 
Then $g_2 \in G$ is regular semisimple, and $\det(g_2) = \det(g)$. In either case 
$$|\bfC_G(g_i)| \leq (q^{n-1}+1)(q+1).$$ 
Choose 
$$D = \left\{ \begin{array}{ll}(q^n-(-1)^n)(q^{n-1}-q^2)/(q+1)(q^2-1), & n \geq 7,\\
                                            (q+1)(q^3+1)(q^5+1)/2, & n = 6. \end{array} \right.$$
If $n \geq 7$, then by \cite[Theorem 4.1]{TZ1}, every irreducible character of $\SU_n(q)$ of degree 
less than $D$ is either the principal character, or an irreducible Weil character, and each of these characters extends to $G$, cf.\ \cite[Lemma 4.7]{TZ2}. In this case, the characters in $\Irr(G)$ of degree less than $D$ are exactly the 
$q+1$ linear characters $\lambda_i$, and $(q+1)^2$ irreducible Weil characters 
$\zeta_{i,j}$, $0 \leq i,j \leq q$, where 
$$\zeta_{i,j} = \lambda_j\zeta_{i,0},~\zeta_{i,0}(1) = \frac{q^n-(-1)^n}{q+1} + (-1)^n\delta_{i,0}.$$ 
Suppose $n = 6$ and $\psi \in \Irr(\SU_6(q))$ has positive $s$-defect $0$ and positive 
$s'$-defect. Using \cite{Lu}, we check that either $\psi$ is the principal character of a Weil
character, or $\psi(1) \geq D$. Again, if $\chi \in \Irr(G)$, $\chi(1) < D$, and 
$\chi(g_1)\chi(g_2) \neq 0$, then $\chi$ is either a linear character, or a Weil character.  

The choice of $g_1$ and $g_2$ ensures that 
\begin{equation}\label{su41}
  \sum^{q}_{i=0}\lambda_i(g_1)\lambda_i(g_2)\overline\lambda_i(g) = q+1.
\end{equation}  
By (\ref{cent-slu}) and Lemma \ref{basic},
\begin{equation}\label{su42}
  \left|\sum_{\chi \in \Irr(G),~\chi(1) \geq D}\frac{\chi(g_1)\chi(g_2)\oc(g)}{\chi(1)}\right|
    < \frac{((q+1)(q^{n-1}+1))^{3/2}}{D} 
    < \frac{2(q+1)}{3}.
\end{equation}
Fix a primitive $(q+1)$th root $\xi \in \F_{q^2}^\times$ of unity and a primitive $(q+1)$th root 
$\tx \in \C^\times$ of unity. Relabeling $\zeta_{i,j}$ if necessary, 
$$\zeta_{i,0}(x) = \frac{(-1)^n}{q+1}\sum^{q}_{l=0}\tx^{il} (-q)^{e(x,\xi^l)}$$
for every $x \in G$, where $e(x,\alpha)$ denotes the dimension of the $\alpha$-eigenspace of 
$x$ on the natural module $\F_{q^2}^n$ for $G$. 
As before, the choice of $g_i$ and the unbreakability of $g$ ensure that 
$e(y,\xi^l)$ is at most $1$ 
for $y \in \{g,g_1,g_2\}$ and $0 \leq l \leq q$,
and in fact it can equal $1$ for at most one value $l_0$. If such $l_0$ exists, then
$$(-1)^n\zeta_{i,0}(y) = \frac{1}{q+1}\left(\tx^{il_0}(-q-1)+\sum^{q}_{l=0}\tx^{il}\right) = 
    \delta_{i,0}-\tx^{il_0}.$$
If no such $l_0$ exists, then   
$$\zeta_{i,0}(y) = \frac{(-1)^n}{q+1}\sum^{q}_{l=0}\tx^{il} = (-1)^n\delta_{i,0}.$$
We have shown that 
\begin{equation}\label{weil-su1}
  |\zeta_{i,j}(y)| \leq 1. 
\end{equation}  
It follows that if $n \geq 5$ then  
\begin{equation}\label{weil-su2}
  \left|\sum_{0 \leq i,j \leq q}\frac{\zeta_{i,j}(g_1)\zeta_{i,j}(g_2)\overline\zeta_{i,j}(g)}{\zeta_{i,j}(1)}\right|
    \leq \frac{(q+1)^3}{q^n-q} \leq \frac{(q+1)^2}{q(q-1)^2}.
\end{equation}    
Together with (\ref{su42}), this implies that
$$\left|\sum_{\chi \in \Irr(G),~\chi(1) > 1}\frac{\chi(g_1)\chi(g_2)\oc(g)}{\chi(1)}\right|
    \leq (q+1)\left(\frac{2}{3} + \frac{1}{7}\right) < q+1.$$  
Hence $g \in g_1^G \cdot g_2^G$ by (\ref{su41}) and Lemma \ref{basic}(i). 
Since both $g_1$ and $g_2$ have order coprime to $N$, we are done. 

\smallskip
(ii) Next we consider the case $n = 5$. Setting $s' := \p(q,6)$ and using Lemma \ref{slu-det} we can choose a regular semisimple $h \in \GU_3(q)$ of order $s'm$, where all prime divisors of $m$ divide $q+1$ and 
$\det(h) = \det(g)$. Also, let $h' \in \GU_2(q)$ be conjugate (over $\overline\F_q$) to 
$\diag(\alpha,\alpha^{-1})$, where $\alpha \in \F_q^\times$ has order $q-1$.  Setting
$g_2 = \diag(h,h')$, the orders of $g_1$ and $g_2$ are coprime to $N$, 
$\det(g_2) = \det(g)$, and 
$e(g_2,\xi^l) = 0$ for $0 \leq l \leq q$.   In particular, (\ref{weil-su1}) and 
(\ref{weil-su2}) hold. Next, we choose $D = q^4(q^5+1)/(q+1)$, yielding
\begin{equation}\label{su43}
  \left|\sum_{\chi \in \Irr(G),~\chi(1) \geq D}\frac{\chi(g_1)\chi(g_2)\oc(g)}{\chi(1)}\right|
    \leq \frac{(q+1)(q^4+1)(q^4(q+1))^{1/2}}{D} 
    < \frac{q+1}{5}.
\end{equation}

Now, using \cite{Lu}, we check that if $\psi \in \Irr(\SU_5(q))$ has positive $s$-defect and 
positive $s'$-defect and $\psi(1) < D$, then either $\psi$ is the principal character or a Weil character, 
or $s = \p(q,4)$ and $\psi$ is the unique character of degree $q^2(q^5+1)/(q+1)$. In either case, 
$\psi$ extends to $G$. In fact, in the latter case, an extension $\varphi$ of $\psi$ to $G$ is the unipotent character labeled by the partition $(3,2)$ (see \cite[\S13.8]{C}). Letting $\sigma$ be the unipotent 
character of $G$ labeled by the partition $(3,1,1)$, of degree $q^3(q^2+1)(q^2-q+1)$, we check that
$$\rho = 1_G + \varphi + \sigma$$
is the (rank 3) permutation character of the action of $G$ on the set of isotropic $1$-dimensional 
subspaces of the natural module $\F_{q^2}^5$, cf.\ \cite[Table 2]{ST}. Note that $\sigma$ has $s$-defect $0$ and 
$s'$-defect $0$. It follows that $\sigma(g_1) = \sigma(g_2) = 0$, so
$$\varphi(g_1) = \rho(g_1) - 1 = 0-1 = -1,~~
    \varphi(g_2) = \rho(g_2) - 1 = 2-1 = 1.$$
Also, the extensions of $\psi$ to $G$ are $\varphi\lambda_i$, $0 \leq i \leq q$, and 
$|\varphi(g)| \leq (q^4(q+1))^{1/2}$ by (\ref{cent-slu}). Certainly, $\chi(g_1)\chi(g_2) = 0$ 
unless $\chi$ has positive $s$-defect and positive $s'$-defect. Hence, combining with 
(\ref{weil-su2}), we deduce that 
$$\left|\sum_{\chi \in \Irr(G),~1 < \chi(1) < D}\frac{\chi(g_1)\chi(g_2)\oc(g)}{\chi(1)}\right|
    \leq \frac{(q+1)^2}{q(q-1)^2} + (q+1)\frac{(q^4(q+1))^{1/2}}{\psi(1)} 
    < \frac{q+1}{7}.$$
Together with (\ref{su43}), this implies that
$$\left|\sum_{\chi \in \Irr(G),~\chi(1) > 1}\frac{\chi(g_1)\chi(g_2)\oc(g)}{\chi(1)}\right|
    \leq (q+1)\left(\frac{1}{5} + \frac{1}{7}\right) < q+1,$$
so we are done as before.      

\smallskip
(iii) Here we consider the case $n = 4$. Since $g$ is unbreakable, $g$ belongs to 
class $A_5$, $C_2$, or $E_1$, in the notation of \cite{Noz1}. In the two latter cases, note that the 
$G$-conjugacy class of such an element $g$ is completely determined 
by $|g|$ and the eigenvalues of $g$ acting on 
$\overline\F_q^4$. On the other hand, $G$ contains a natural subgroup 
$H \cong \GL_2(q^2)$, and $H$ contains an element $h$ with the same spectrum and 
order as $g$. Hence we may assume $g =h \in H$. As $N = p^at^b$ and $t \nmid (q^2-1)$, we 
can now apply Lemma \ref{slu-small1} (if $h$ is unbreakable) to get $g = x^Ny^N$ for some 
$x \in \SL_2(q^2) < \SL_4(q)$ and $y \in H$. Such a decomposition certainly exists if $h$ is 
breakable in $H$ (i.e.\ $h \in \GL_1(q^2) \times \GL_1(q^2)$).

It remains therefore to consider the case $g \in A_5$, i.e.\ $g = zu$, where $z \in \bfZ(G)$ and
$u$ is a regular unipotent element. By \cite[Corollary 8.3.6]{C}, $|\chi(g)| \leq 1$ for all
$\chi \in \Irr(G)$. Choosing $D = (q+1)(q^3+1)$ and $g_2$ of order $sm$ as in (i) (when $n \geq 7$), by the Cauchy-Schwarz inequality, 
$$ \left|\sum_{\chi \in \Irr(G),~\chi(1) \geq D}\frac{\chi(g_1)\chi(g_2)\oc(g)}{\chi(1)}\right|
    \leq \frac{(q^3+1)(q+1)}{(q+1)(q^3+1)} = 1 \leq \frac{q+1}{5}.$$
Using \cite{Noz1}, we check that all irreducible characters of $G$ of degree less than $D$ are
linear or Weil characters. Hence (\ref{weil-su1}) implies that
$$\left|\sum_{\chi \in \Irr(G),~1 < \chi(1) < D}\frac{\chi(g_1)\chi(g_2)\oc(g)}{\chi(1)}\right|
    \leq \frac{(q+1)^{2}}{q(q-1)^2} < \frac{q+1}{7}.$$
It follows that        
$$\left|\sum_{\chi \in \Irr(G),~\chi(1) > 1}\frac{\chi(g_1)\chi(g_2)\oc(g)}{\chi(1)}\right|
    < (q+1)\left(\frac{1}{5} + \frac{1}{7}\right) < q+1,$$     
so we are done.
\hal

\begin{cor}\label{slu-large}
Theorem \ref{main1} holds for $G = \PSL^\e_n(q)$ with $q = p^f \geq 4$, $\e = \pm$, and 
$n \geq 2$.
\end{cor}

\pf
The case $n = 2,3$ follows from Lemma \ref{slu-small2}. If $n \geq 4$, then 
we choose $n_0 = 3$ and apply Proposition \ref{ind-slu}. Note that condition (i) of
that proposition is satisfied by Lemma \ref{slu-small1}, and (ii) holds by Propositions
\ref{sl-large} and \ref{su-large}. Hence we are done by Proposition \ref{ind-slu}.
\hal

\subsection{Induction step: Small fields}
\begin{prop}\label{sl2}
Suppose $G = \GL_n(2)$ with $n \geq 8$ and $t \in \cR(G)$. Then $\sP_u(N)$ holds for 
$G$ and for every $N = 2^at^b$. 
\end{prop}

\pf
Consider an unbreakable $g \in G$ and choose 
$$D = (2^n-1)(2^{n-1}-4)/3.$$
By \cite[Theorem 3.1]{TZ1}, $\Irr(G)$ contains exactly two characters of degree less than $D$:
namely, $1_G$ and $\tau$. In fact $\tau(1) = 2^n-2$ and $\rho = \tau+1_G$ is the permutation 
character of the action of $G$ on the set of nonzero vectors of the natural module 
$V = \F^n_2$. Choose regular semisimple elements 
$g_1=g_2$ of order $s \in \cR(G) \setminus \{t\}$; in particular, $|\bfC_G(g_i)| \leq 2^n-1$.  
Note that $\rho(g_i) \in \{0,1\}$, so $|\tau(g_i)| \leq 1$. Also, $|\bfC_G(g)| \leq 9 \cdot 2^{n}$
by Lemma \ref{LUL-sl2}. It follows that
$$\frac{|\tau(g_1)\tau(g_2)\tau(g)|}{\tau(1)} 
    \leq \frac{3 \cdot 2^{n/2}}{2^n-2} < 0.189.$$
If $n \geq 9$, then by Lemma \ref{basic}(ii) 
$$\left|\sum_{\chi \in \Irr(G),~\chi(1) \geq D}\frac{\chi(g_1)\chi(g_2)\oc(g)}{\chi(1)}\right|
    \leq \frac{(2^n-1) \cdot 3 \cdot 2^{n/2}}{D} = 
    \frac{9 \cdot 2^{n/2}}{2^{n-1}-4} < 0.809.$$
If $n = 8$ and $|\bfC_G(g)| \leq 2^{n+2}$, then     
$$\left|\sum_{\chi \in \Irr(G),~\chi(1) \geq D}\frac{\chi(g_1)\chi(g_2)\oc(g)}{\chi(1)}\right|
    \leq \frac{(2^n-1) \cdot 2 \cdot 2^{n/2}}{D} = 
    \frac{6 \cdot 2^{n/2}}{2^{n-1}-4} < 0.775.$$
Thus, in each of these cases,
$$\left|\sum_{\chi \in \Irr(G),~\chi(1) > 1}\frac{\chi(g_1)\chi(g_2)\oc(g)}{\chi(1)}\right|
    < 0.809 + 0.189 = 0.998,$$
whence $g \in g_1^G \cdot g_2^G$ by Lemma \ref{basic}(i). 
Since both $g_1$ and $g_2$ have order coprime to $N$, we
are done in these cases. In the remaining case, by Lemma \ref{LUL-sl2}, 
$G = \GL_8(2)$ and $g \in H := \GL_4(4) = \bfZ(H) \times S$ with $\bfZ(H) \cong C_3$ and
$S \cong \SL_4(4)$. Thus we can write $g = zh$ with $z \in \bfZ(H)$ and $h \in S$.
Applying Corollary \ref{slu-large} to $\SL_4(4)$, we deduce that $h = x^Ny^N$ for some $x,y \in S$. 
Certainly, $z = z_1^N$ for some $z_1 \in \bfZ(H)$. It follows that 
$g = (z_1x)^Ny^N$, and we are done again.
\hal

\begin{prop}\label{sl3}
Suppose $G = \GL_n(3)$ with $n \geq 8$ and $t \in \cR(\SL_n(3))$. Then $\sP_u(N)$ holds for 
$G$ and for every $N = 3^at^b$. 
\end{prop}

\pf
Consider an unbreakable $g \in G$, so 
$|\bfC_G(g)| \leq 3^{n+2} \cdot 2^4$ by Lemma \ref{LUL3}. 
First, we use Lemma \ref{slu-det} to get a regular semisimple element $g_1$ of order $sm$, where 
$s \in \cR(G) \setminus \{t\}$, $m$ is a $2$-power, and $\det(g_1) = \det(g)$. Next we fix a 
regular semisimple 
$h \in \SL_{n-2}(3)$ of order $s' = \p(3,n-2)$ and 
$h' \in \SL_2(3)$ of order $4$, and set $g_2 := \diag(h,h')$. In particular, 
$|\bfC_G(g_i)| \leq 3^n-1$. Also,  we choose 
$$D = 3^{3n-9}.$$
By Lemma \ref{basic}(ii), 
$$\left|\sum_{\chi \in \Irr(G),~\chi(1) \geq D}\frac{\chi(g_1)\chi(g_2)\oc(g)}{\chi(1)}\right|
    \leq \frac{(3^n-1) \cdot 4 \cdot 3^{(n+2)/2}}{3^{3n-9}} 
    < \frac{4}{3^{3n/2-10}} \leq \frac{4}{9}.$$
Now we estimate character ratios for $\chi \in \Irr(G)$ with $\chi(1) < D$ and $\chi(g_1)\chi(g_2) \neq 0$.
The latter condition implies that $\chi$ has positive $s$-defect and positive $s'$-defect.  Applying 
\cite[Theorem 3.4]{BK}, $\chi$ can be only one of the following: 

$\bullet$ two linear characters $\lambda_{0,1}$, 

$\bullet$ two of the four Weil characters $\tau_{i,j}$ with $0 \leq i \leq 1$ (see the proof of 
Proposition \ref{sl-large} for their definition), and, possibly, 

$\bullet$ two characters $\varphi_{0,1} = \varphi\lambda_{0,1}$. Here, $\varphi$ is the unipotent character of $G$ labeled by the partition $(n-2,2)$, of degree $(3^n-1)(3^{n-1}-9)/16$. 

The elements $g_{1,2}$ have the property that $e(g_i,\delta) \leq 1$ for all $\delta \in \F_3^\times$,
with equality attained at most once. Hence, the estimate (\ref{weil-sl1}) holds.  It follows
that 
$$\sum_{0 \leq i,j \leq 1}\frac{|\tau_{i,j}(g_1)\tau_{i,j}(g_2)\overline\tau_{i,j}(g)|}{\tau_{i,j}(1)} 
    \leq \frac{2 \cdot 4 \cdot 3^{n/2+1}}{(3^n-3)/2} \leq \frac{8}{13}.$$
On the other hand, $\tau_{0,0}$ is the unipotent 
character of $G$ labeled by the partition $(n-1,1)$. It follows by \cite[Lemma 5.1]{GT1} that
$$\varphi = (1_G + \tau_{0,0} + \varphi) - (1_G + \tau_{0,0}) = \rho_2 - \rho_1,$$
where $\rho_i$ is the permutation character of the action of $G$ on the set of $i$-dimensional 
subspaces of the natural module $\F_3^n$ for $i = 1,2$. Observe that $\rho_2(g_1) = 0$ and 
$\rho_1(g_1) = 0$ or $1$. Therefore, 
$$|\varphi(g_1)| = |\rho_2(g_1) - \rho_1(g_1)| = |\rho_1(g_1)| \leq 1,~~
    \varphi(g_2) = \rho_2(g_2) - \rho_1(g_2) = 1-0 = 1.$$
This implies that  
$$\left|\sum^1_{i=0}\frac{\varphi_i(g_1)\varphi_i(g_2)\overline\varphi_i(g)}{\varphi_i(1)}\right|
    \leq \frac{2 \cdot 4 \cdot 3^{n/2+1}}{(3^n-1)(3^{n-1}-9)/16} \leq
    \frac{128}{(3^{n/2-1}-1)(3^{n-1}-9)} < 0.003.$$
In summary, 
$$\left|\sum_{\chi \in \Irr(G),~\chi(1) > 1}\frac{\chi(g_1)\chi(g_2)\oc(g)}{\chi(1)}\right|
    < \frac{4}{9} + \frac{8}{13} + 0.003 <  1.07.$$
Our choice of $g_1$ and $g_2$ ensures that
$$\sum^{1}_{i=0}\lambda_i(g_1)\lambda_i(g_2)\overline\lambda_i(g) = 2.$$
Hence $g \in g_1^G \cdot g_2^G$ by Lemma \ref{basic}(i), so we are done since
$|g_1|$ and $|g_2|$ are both coprime to $N$.
\hal

\begin{prop}\label{su3}
Suppose $G = \GU_n(3)$ with $n \geq 7$ and $t \in \cR(\SU_n(3))$. Then $\sP_u(N)$ holds for 
$G$ and for every $N = 3^at^b$. 
\end{prop}

\pf
Consider an unbreakable $g \in G$, so 
$|\bfC_G(g)| \leq 3^{n+2} \cdot 2^4$ by Lemma \ref{LUL3}. 
First, we use Lemma \ref{slu-det} to get a regular semisimple $g_1$ of order $sm$, where $s \in \cR(G) \setminus \{t\}$, $m$ is a 
$2$-power, and $\det(g_1) = \det(g)$. Then we choose 
$$s' := \left\{ \begin{array}{ll}\p(q,2n-4), & n \equiv 1 \bmod 2,\\
                                             \p(q,n-2),   & n \equiv 2 \bmod 4,\\
                                   \p(q,(n-2)/2), & n \equiv 0 \bmod 4, \end{array}\right.$$   
(with $q = 3$). Note that $s'|(q^{n-2}-(-1)^{n-2})$ but $s' \nmid \prod^{n}_{i=1,i \neq n-2}(q^i-(-1)^i)$.                                             
Next, we fix $\alpha \in \ovF_3^\times$ of order $q^{n-2}-(-1)^n$ and choose  
a regular semisimple $h \in \GU_{n-2}(3)$ that is conjugate over $\ovF_3$ to 
$$\diag\left(\alpha,\alpha^{-q},\alpha^{q^2}, \ldots ,\alpha^{(-q)^{n-3}}\right).$$
Note that $\det(h) \in \F_9^\times$ has order $4$. Hence there is some $\beta \in \F_9^\times$ of order $q^2-1$ so that $\det(h) = \beta^2$. We fix $h'  = \diag(\beta,\beta^{-q}) \in \GU_2(3)$,
and set $g_2 := \diag(h,h')$. In particular, $g_2 \in \SU_n(3)$ is $s'$-singular, $g_i$ is an 
$N'$-element and $|\bfC_G(g_i)| \leq 4(3^{n-1}+1)$ for $i = 1,2$.

Recall the Weil characters $\zeta_{i,j}$, $0 \leq i,j \leq q$ defined in the proof of Proposition
\ref{su-large}. Fix $\xi \in \F_{q^2}^\times$ of order $q+1$. 
The elements $g_{1,2}$ have the property that $e(g_i,\xi^l) \leq 1$ for all $0 \leq l \leq q$,
with equality attained at most once. Hence, the estimate (\ref{weil-su1}) holds for $y = g_i$.
Also, 
\begin{equation}\label{weil-su3}
  e(g,\xi^l) \leq n/2
\end{equation}  
whenever $n \geq 7$. (Indeed, otherwise 
$U = \Ker(g-\xi^l \cdot 1_V)$ has dimension $\geq (n+1)/2$ in 
the natural $G$-module $V := \F_{q^2}^n$. It follows that $U$ cannot be totally singular, 
so $U$ contains at least one anisotropic vector $u$. In this case, $g$ fixes the decomposition
$$V = \la u \ra_{\F_{q^2}} \oplus (\la u \ra_{\F_{q^2}})^\perp.$$
In other words,  $g \in \GU_1(q) \times \GU_{n-1}(q)$, so $g$ is breakable, a contradiction.)  
As $n \geq 7$, we deduce that $e(g,\xi^l) \leq n-4$, whence
\begin{equation}\label{weil-su31}
  |\zeta_{i,j}(g)| \leq \frac{(q+1)q^{n-4}}{q+1} =  q^{n-4},
\end{equation}  
so
\begin{equation}\label{weil-su4}
  \sum_{0 \leq i,j \leq q}\frac{|\zeta_{i,j}(g_1)\zeta_{i,j}(g_2)\overline\zeta_{i,j}(g)|}{\zeta_{i,j}(1)} 
    \leq \frac{16 \cdot 3^{n-4}}{(3^n-3)/4} < 0.8.
\end{equation}    
 
Choosing 
$$D = \left\{ \begin{array}{rl}(3^{n}-1)(3^{n-1}-1)(3^{n-2}-27)/896, & n \geq 8,\\
                                            3^{16}, & n = 7, \end{array} \right.$$
by Lemma \ref{basic}(ii) 
\begin{equation}\label{su31}
  \left|\sum_{\chi \in \Irr(G),~\chi(1) \geq D}\frac{\chi(g_1)\chi(g_2)\oc(g)}{\chi(1)}\right|
    \leq \frac{4(3^{n-1}+1) \cdot 4 \cdot 3^{(n+2)/2}}{D} 
    < 0.76.
\end{equation}    
Now we estimate character ratios for $\chi \in \Irr(G)$ with $\chi(1) < D$ and $\chi(g_1)\chi(g_2) \neq 0$.
The latter condition implies that $\chi$ has positive $s$-defect and positive $s'$-defect.  Applying 
\cite[Proposition 6.6]{ore} for $n \geq 8$, $\chi$ can be only one of 
the following:

$\bullet$ $4$ linear characters $\lambda_{i}$, $0 \leq i \leq 3$; 

$\bullet$ (at most $12$ of the) $16$ Weil characters $\zeta_{i,j}$ with $0 \leq i \leq 3$, and 

$\bullet$ $4$ characters $\varphi_{i} = \varphi\lambda_{i}$, $0 \leq i \leq 3$, 
if $s | (q^{n-1}+(-1)^n)$. Here, $\varphi$ is the unipotent character of $G$ labeled by the partition 
$(n-2,2)$, of degree 
$$\varphi(1) = (3^n-(-1)^n)(3^{n-1}+9(-1)^n)/32.$$ 
This conclusion also holds for $n = 7$. (Indeed, 
for $n = 7$, using \cite{Lu} we can check that if $\sigma \in \Irr(\SU_7(q))$ has 
positive $s$-defect and positive $s'$-defect and $\sigma(1) < D$, then $\sigma$ is the 
restriction to $\SU_7(q)$ of one of the above characters of $\GU_7(q)$.)

Let $\psi$ denote the unipotent  character of $G$ labeled by the partition $(n-2,1,1)$, of degree
$$\psi(1) = (3^n+3(-1)^n)(3^n-9(-1)^n)/32.$$
It is well known, see e.g. \cite[Table 2]{ST}, that $\rho := 1_G + \varphi + \psi$ is the permutation character of the action of $G$ on the set of isotropic $1$-dimensional 
subspaces of the natural module $V$.  Recall we need to consider $\varphi$ only when
$s|(q^{n-1}+(-1)^n)$, so $\psi$ has $s$-defect $0$ and $s'$-defect $0$. In particular, 
$\psi(g_1) = \psi(g_2) = 0$. Therefore, 
$$\varphi(g_1) = \rho(g_1) - 1 = 0-1 = -1,~~
    \varphi(g_2) = \rho(g_2) - 1 = 2-1 = 1.$$
Since $|\varphi(g)| \leq 4 \cdot 3^{n/2+1}$, 
$$\left|\sum^q_{i=0}\frac{\varphi_i(g_1)\varphi_i(g_2)\overline\varphi_i(g)}{\varphi_i(1)}\right|
    \leq \frac{4 \cdot 4 \cdot 3^{n/2+1}}{(3^n-1)(3^{n-1}-9)/32}  < 0.05.$$
Together with (\ref{weil-su3}) and (\ref{su31}), this implies that
$$\left|\sum_{\chi \in \Irr(G),~\chi(1) > 1}\frac{\chi(g_1)\chi(g_2)\oc(g)}{\chi(1)}\right|
    <  0.8+ 0.76 + 0.05 =  1.61.$$
Our choice of $g_1$ and $g_2$ ensures that 
$$\sum^{q}_{i=0}\lambda_i(g_1)\lambda_i(g_2)\overline\lambda_i(g) = 4.$$
Hence $g \in g_1^G \cdot g_2^G$ by Lemma \ref{basic}(i), so we are done since
$|g_1|$ and $|g_2|$ are both coprime to $N$.
\hal

\begin{prop}\label{su2}
Suppose $G = \GU_n(2)$ with $n \geq 9$ and $t \in \cR(\SU_n(2))$. Then $\sP_u(N)$ holds for 
$G$ and for every $N = 2^at^b$. 
\end{prop}

\pf
Consider an unbreakable $g \in G$, so 
$|\bfC_G(g)| \leq 2^{n+4} \cdot 3^2$ when $n \geq 10$ and $|\bfC_G(g)| \leq 2^{48}$ when $n = 9$
by Lemma \ref{LUL}.

\smallskip
(i) First, we use Lemma \ref{slu-det} to get a regular semisimple element $g_1$ of order $sm$, where 
$s \in \cR(G) \setminus \{t\}$, $m$ is a $3$-power, and $\det(g_1) = \det(g)$. If $n \geq 10$, we can
find a regular semisimple $g_2 \in \SU_n(2)$ of order $s$. In particular, $g_i$ is an 
$N'$-element and $|\bfC_G(g_i)| \leq 3(2^{n-1}+1)$ for $i = 1,2$. 
Fix $\xi \in \F_{q^2}^\times$ of order $q+1$. As in the proof of Proposition \ref{su3}, we
note that the elements $g_{1,2}$ have the property that $e(g_i,\xi^l) \leq 1$ for all $0 \leq l \leq q$,
with equality attained at most once. Hence, the estimate (\ref{weil-su1}) holds for $y = g_i$.

If $n = 9$, we choose $s':= 43$ and fix a regular semisimple $h \in \SU_7(2)$ of order
$43$. Also, we fix $h' \in \SU_2(2)$ of order $3$ and set $g_2:= \diag(h,h')$. 
In particular, $|\bfC_G(g_2)| = 9(2^7+1)$, and $e(g_2,\xi^l)$ equals $0$ for $l = 0$ and $1$ 
for $l = 1,2$. Direct computation shows that $|\zeta_{i,j}(g_2)| = 1$ for all $i,j$.
Thus, for $n \geq 9$ and $y \in \{g_1,g_2\}$,  
\begin{equation}\label{weil-su5}
  |\zeta_{i,j}(y)| \leq 1.
\end{equation}   

\smallskip
(ii) Choosing 
$$D = \left\{ \begin{array}{rl}(2^{n}+1)(2^{n-1}-1)(2^{n-2}-27)/81, & n \geq 10,\\
                                            2^{22} \cdot 7 \cdot (2^9+1), & n = 9, \end{array} \right.$$
by Lemma \ref{basic}(ii), for $n \geq 10$, 
\begin{equation}\label{su21}
  \left|\sum_{\chi \in \Irr(G),~\chi(1) \geq D}\frac{\chi(g_1)\chi(g_2)\oc(g)}{\chi(1)}\right|
    \leq \frac{3(2^{n-1}+1) \cdot 3 \cdot 2^{n/2+2}}{D} 
    < 0.37.
\end{equation}    
Now we estimate character ratios for $\chi \in \Irr(G)$ with $\chi(1) < D$ and $\chi(g_1)\chi(g_2) \neq 0$.
The latter condition implies that $\chi$ has positive $s$-defect and positive $s'$-defect. Applying 
\cite[Proposition 6.6]{ore} for $n \geq 10$, $\chi$ can be only one of the 
following:

$\bullet$ $3$ linear characters $\lambda_{i}$, $0 \leq i \leq 2$; 

$\bullet$ at most $6$ of the $9$ Weil characters $\zeta_{i,j}$ with $0 \leq i \leq 2$, and 

$\bullet$ (some of the) $27$ characters $D^\circ_\alpha\lambda_{i}$, $0 \leq i \leq 2$, 
$\alpha \in \Irr(S)$ with $S := \GU_2(2)$ (see \cite[Proposition 6.3]{ore} for the definition of 
$D^\circ_\alpha$).

This conclusion also holds for $n = 9$. (Indeed, 
for $n = 9$, using \cite{Lu} we can check that if $\sigma \in \Irr(\SU_9(2))$ has 
positive $s$-defect and positive $s'$-defect and $\sigma(1) < D$, then $\sigma$ is the 
restriction to $\SU_9(2)$ of one of the above characters of $\GU_9(2)$.)

Next, the inequality (\ref{weil-su3}) implies that $e(g,\xi^l) \leq n-5$ as $n \geq 9$, so
$$|\zeta_{i,j}(g)| \leq \frac{(q+1)q^{n-5}}{q+1} =  q^{n-5}.$$
It now follows from (\ref{weil-su5}) that
\begin{equation}\label{weil-su6}
  \sum_{0 \leq i,j \leq q}\frac{|\zeta_{i,j}(g_1)\zeta_{i,j}(g_2)\overline\zeta_{i,j}(g)|}{\zeta_{i,j}(1)} 
    \leq \frac{6 \cdot 2^{n-5}}{(2^n-2)/3} < 0.57.
\end{equation}    

\smallskip
(iii) Now we assume that $n \geq 10$. 
We already observed that $e(g_i,\xi^l) \leq 1$ for $0 \leq l \leq 2$, with equality attained at most once.
Thus $g_i$ satisfies the conclusion (i) of \cite[Lemma 6.7]{ore}. Hence it also satisfies the 
conclusion (ii) of \cite[Proposition 6.9]{ore}. Thus  
$$|D^\circ_\alpha(g_i)| \leq \left\{ \begin{array}{ll}
    2, & \alpha(1) = 1, \alpha \neq 1_S,\\
    3, & \alpha = 1_S,\\
    4, & \alpha(1) = 2. \end{array} \right.$$
Since $|\varphi(g)| \leq 3 \cdot 2^{n/2+2}$, 
$$\left|\sum_{\chi = D^\circ_\alpha \lambda_i}\frac{\chi(g_1)\chi(g_2)\oc(g)}{\chi(1)}\right|
    \leq 3^2 \cdot 2^{n/2+2} \cdot \left( \frac{5 \cdot 2^2 + 3^2}{(2^n-1)(2^{n-1}-4)/9}
    +  \frac{3 \cdot 4^2}{(2^n-2)(2^n-4)/9} \right)
      < 1.06.$$
Together with (\ref{weil-su6}) and (\ref{su21}), this implies that
$$\left|\sum_{\chi \in \Irr(G),~\chi(1) > 1}\frac{\chi(g_1)\chi(g_2)\oc(g)}{\chi(1)}\right|
    <  0.57+ 0.37 + 1.06 =  2.$$
Our choice of $g_1$ and $g_2$ ensures that 
$$\sum^{q}_{i=0}\lambda_i(g_1)\lambda_i(g_2)\overline\lambda_i(g) = 3.$$
Hence $g \in g_1^G \cdot g_2^G$ by Lemma \ref{basic}(i), so we are done since
$|g_1|$ and $|g_2|$ are both coprime to $N$.

\smallskip
(iv) Finally, we handle the case $n = 9$. Now $\chi = D^\circ_\alpha \lambda_i$ can have 
positive $s$-defect and positive $s'$-defect only when $t=19$, $s=17$, $\alpha = 1_S$. In this case,
$\varphi := D^\circ_{1_S}$ is the unipotent character of $G$ labeled by the partition 
$(n-2,2)$, of degree 
$$\varphi(1) = (2^9+1)(2^{8}-4)/9= 14364.$$ 
Let $\psi$ denote the unipotent  character of $G$ labeled by the partition $(n-2,1,1)$, of degree
$$\psi(1) = (2^9-2)(2^9+4)/9 = 29240.$$
Again, $\rho := 1_G + \varphi + \psi$ is the permutation character of the action of $G$ on the set of isotropic $1$-dimensional subspaces of the natural module $V$, see e.g. \cite[Table 2]{ST}.  
Recall we need to consider $\varphi$ only when $s=17$ (and $s' = 43$), so $\psi$ has $s$-defect $0$ and $s'$-defect $0$. In particular, 
$\psi(g_1) = \psi(g_2) = 0$. Therefore, 
$$\varphi(g_i) = \rho(g_i) - 1 = 0-1 = -1.$$
Recall that $e(g,\xi^l) \leq 4$ for all unbreakable $g \in \GU_9(q)$ and $0 \leq l \leq 2$. 
Arguing as in the proof of \cite[Proposition 6.9]{ore}, we obtain 
$$|\varphi(g)| = |D^\circ_{1_S}(g)| \leq 2^8+1 = 257.$$ 
It follows that
$$\left|\sum_{\chi = \varphi\lambda_i,~0 \leq i \leq 2}\frac{\chi(g_1)\chi(g_2)\oc(g)}{\chi(1)}\right|
    \leq \frac{3 \cdot 257}{14364}  < 0.06.$$
By Lemma \ref{basic}(ii), 
$$\left|\sum_{\chi \in \Irr(G),~\chi(1) \geq D}\frac{\chi(g_1)\chi(g_2)\oc(g)}{\chi(1)}\right|
    \leq \frac{(765 \cdot 1161)^{1/2}  \cdot 2^{24}}{2^{22} \cdot 7 \cdot 513} 
    < 1.05.$$       
In summary, 
$$\left|\sum_{\chi \in \Irr(G),~\chi(1) > 1}\frac{\chi(g_1)\chi(g_2)\oc(g)}{\chi(1)}\right|
    <  0.57+ 0.06 + 1.05 =  1.68,$$
so we are done again.
\hal

\begin{lem}\label{base-slu}
Let $q = p= 2,3$ and $\e =\pm$. 

{\rm (i)} Suppose that $3 \leq k \leq 4$ and $k \neq 3$ if $(q,\e) = (2,-)$. Then  
$\sP(N)$ holds for $\GL^\e_k(q)$ and for every $N = p^at^b$ with $t \neq p$ a prime and 
$t \nmid (q-\e)$. Also, Theorem \ref{main1} holds for $\PSL^\e_k(q)$.

{\rm (ii)} 
Let $\GL^\e_k(q)$ be one of the following groups:
\begin{enumerate}[\rm(a)]
\item $\GL_n(2)$ or $\GL_n(3)$, with $5 \leq n \leq 7$;
\item $\GU_n(2)$ with $5 \leq n \leq 8$;
\item $\GU_n(3)$ with $n = 5,6$.
\end{enumerate}
Then $\sP_u(N)$ holds for $\GL^\e_k(q)$ and for every $N = p^at^b$ with $t \in \cR(\SL^\e_k(q))$, 
%$\sP_u(N)$ holds for $\GL^\e_k(q)$ and for every $N = p^at^b$ with $t \in \cR(\SL^\e_k(q))$, 

\end{lem}

\pf
Direct calculations similar to those of Lemma \ref{alt-spor}.
\hal

\begin{cor}\label{slu-small}
Theorem \ref{main1} holds for $G = \PSL^\e_n(q)$ with $q = p^f = 2,3$, $\e = \pm$, 
$n \geq 3$, and $(n,q,\e) \neq (3,2,-)$.
\end{cor}

\pf
The case $n = 3,4$ follows from Lemma \ref{base-slu}(i). Suppose now that $n \geq 5$. 
Then we choose $n_0 = 4$ and apply Proposition \ref{ind-slu}. Note that condition (i) of
that proposition is verified by Lemma \ref{base-slu}(i), and (ii) holds by Propositions
\ref{sl2}, \ref{sl3}, \ref{su3}, \ref{su2}, and Lemma \ref{base-slu}(ii). Hence we are done by Proposition \ref{ind-slu}.
\hal

\section{Theorem \ref{main1} for symplectic and orthogonal groups}
\subsection{General inductive argument}
Recall $\cR(G)$ from \S2, and the notion of 
unbreakability for symplectic and orthogonal groups from Definition \ref{unbr-spo}. 

\begin{defn}\label{cond-spo}
{\em Given a prime power $q = p^f$, a finite symplectic or orthogonal group 
$G = \Cl(V) = \Cl_n(q)$, and an integer $N = p^at^b$ with $t > 2$ 
a prime. We say that $G$ satisfies 

{\rm (i)} the condition $\sP(N)$ if every 
$g \in G$ can be written as $g=x^Ny^N$ for some $x,y \in G$; and

{\rm (ii)} the condition $\sP_u(N)$ if every {\it unbreakable} 
$g \in G$ can be written as $g=x^Ny^N$ for some $x,y \in G$.}
\end{defn}

Our proof of Theorem \ref{main1} for symplectic and orthogonal groups relies on the following inductive argument: 

\begin{prop}\label{ind-clas}
Given a prime power $q = p^f$, an integer $n \geq 4$, let $V = \F_q^n$ be a 
finite symplectic or quadratic space, and let $G := \Cl(V) = \Cl_n(q)$ be perfect, with $\Cl = \Sp$ or 
$\Omega$. Suppose that there is an integer 
$n_0 \geq 4$ with the following properties:
\begin{enumerate}[\rm(i)]
\item If $1 \leq k \leq n_0$ and $\Cl_k(q)$ is perfect, then 
$\sP_u(N)$ holds for $\Cl_k(q)$ and for every 
$N = p^at^b$ with $t \neq 2, p$ any prime; and
\item For each $k$ with $n_0 < k \leq n$, $\sP_u(N)$ holds for $\Cl_k(q)$ and for every $N = p^at^b$ with $t \in \cR(\Cl_k(q))$.
\end{enumerate}
If $N = s^at^b$ for some primes $s,t$, then
the word map $(u,v) \mapsto u^Nv^N$ is surjective on 
$G/\bfZ(G)$.  
\end{prop}

\pf
By Corollary \ref{prime2}, we need to consider only the case $N = p^at^b$ with 
$t \in \cR(\Cl_n(q))$; in particular, $t > 2$. It suffices to show $\sP(N)$ holds for $G$. 
According to (ii), $\sP_u(N)$ holds for $G$. Consider a
breakable $g \in G$  and write it as $\diag(g_1, \ldots ,g_m)$ lying in the 
natural subgroup
$$\Cl(U_1) \times \ldots \times \Cl(U_i) \cong \Cl_{k_1}(q) \times \ldots \times \Cl_{k_m}(q)$$
that corresponds to an orthogonal decomposition
$V = U_1 \oplus \ldots \oplus U_m$. 
Here, $1 \leq k_i < n$, and for each $i$ either $\Cl_{k_i}(q)$ is perfect or $g_i = \pm 1_{U_i}$.
Relabeling the elements $g_i$ suitably, we may assume that there is some 
$m' \leq m$ such that $g_i$ is unbreakable if 
$1 \leq i \leq m'$ and $g_i = \pm 1_{U_i}$ if $i > m'$. 
Hence, according to (i), $\sP_u(N)$ holds for $\Cl_{k_i}(q)$ if $k_i \leq n_0$ and $i \leq m'$. 
Suppose $k_i > n_0$. Then $\sP_u(N)$ holds for $\Cl_{k_i}(q)$ if $t \in \cR(\Cl_{k_i}(q))$ by 
(ii). If $t \notin \cR(\Cl_{k_i}(q))$, then by Theorem \ref{prime1} every non-central element
of $G$ is a product of two $N'$-elements, so it is a product of two $N$th powers.
Furthermore, all central elements of $\Cl_{k_i}(q)$ are $N$th powers. Hence 
$\sP_u(N)$ holds for $\Cl_{k_i}(q)$ in this case as well. Thus 
for $i \leq m'$ we can write $g_i = x_i^Ny_i^N$ with $x_i,y_i \in \Cl(U_i)$. Setting
$$U := U_1 \oplus \ldots \oplus U_{m'}, ~W = U_{m'+1} \oplus \ldots \oplus U_{m},~
   h := \diag(g_{m'+1}, \ldots, g_m) \in \Iso(W),$$
(where $\Iso(W) = \Sp(W)$ if $\Cl = \Sp$ and $\Iso(W) = \GO(W)$ if $\Cl = \Omega$),   
we see that either $|h| = 1$, or $p$ and $N$ are odd and  $|h| = 2$. In particular,
$h = h^N$ in either case. Letting
$$x := \diag(x_1, \ldots ,x_{m'}) \in \Cl(U),~~y := \diag(y_1, \ldots, y_{m'}) \in \Cl(U)$$
we deduce that $g = x^Ny^Nh^N = x^N(yh)^N$. Also, $x,y \in G$, 
$g=\diag(g',h) \in G$ with $g':= \diag(g_1, \ldots ,g_{m'}) \in \Cl(U) \leq G$. 
It follows that $h \in G$, so $\sP(N)$ holds for $G$, as desired. 
\hal

\subsection{Induction base}
\begin{lem}\label{sp24}
Let $q = p^f$ and let $N = p^at^b$ with $t \neq 2,p$ any prime. Then $\sP(N)$ holds for 
$G = \SL_2(q)$ with $q \geq 4$, and for $\Sp_4(q)$ with $q \geq 3$.
% $\Sp_6(2)$, $\Omega_7(3)$, $\Omega^\pm_8(2)$, and $\Omega^\pm_8(3)$.
\end{lem}

\pf
(i) Consider the case $G = \SL_2(q)$. If $t \nmid (q-1)$, then we check 
that $X^G \cdot X^G = G$ for $X = x\bfZ(G)$ and $x \in G$ of order $q-1$. 
On the other hand, if $t \nmid (q+1)$, then $Y_1^G \cdot Y_2^G \supseteq G \setminus \{1\}$
for $Y_i = y^i\bfZ(G)$, $i = 1,2$, and $y \in G$ of order $q+1$. Since $N$ is odd,
we are done in both cases.  

\smallskip
(ii) Consider the case $G = \Sp_4(q)$ with $2|q$. The character table of $G$ is given in \cite{E}. Suppose 
that $t|(q^2+1)$. We fix a regular semisimple $x_1 \in G$ of order $q-1$ belonging to 
the class $B_1(1,2)$ and a regular semisimple $x_2 \in G$ of order $q+1$ belonging to 
the class $B_4(1,2)$, in the notation of \cite[Table IV-1]{E}. There are $3$ non-principal characters of $G$ that are nonzero at both $x_1$ and $x_2$: 
namely, 
$\theta_{1,2}$ of degree $q(q^2+1)/2$, and 
$\St$ of degree $q^4$. For every $1 \neq g \in G$, 
$$\sum_{1_G \neq \chi \in \Irr(G)}\frac{|\chi(x_1)\chi(x_2)\chi(g)|}{\chi(1)} 
    \leq 2 \cdot \frac{q(q+1)/2}{q(q^2+1)/2} + \frac{q}{q^4} < 1,$$
so we are done by Lemma \ref{basic}(i).      

Suppose now that $t \nmid (q^2+1)$. Then at least one of $x_1$ and $x_2$ has order coprime to $N$; denote it by $x$.  We also fix a regular semisimple $y \in G$ of order $q^2+1$ belonging to the 
class $B_5(1)$. 
There are at most $2$ non-principal characters of $G$ that are nonzero at both $x$ and $y$: namely, $\St$ and possibly 
a character $\theta$ of degree $\geq q(q-1)^2/2$. For every $1 \neq g \in G$, 
$$\sum_{1_G \neq \chi \in \Irr(G)}\frac{|\chi(x)\chi(y)\chi(g)|}{\chi(1)} 
    \leq \frac{q(q-1)/2}{q(q-1)^2/2} + \frac{q}{q^4} < 1,$$
so we are done by Lemma \ref{basic}(i).      

\smallskip
(iii) Assume that $G = \Sp_4(q)$ with $q \geq 7$ odd. 
The character table of $G$ is given in \cite{Sri}. 
If $t \nmid (q^2+1)$, then the statement follows from \cite[Theorem 7.3]{GM}. So we assume 
that $t|(q^2+1)$. Fix a regular semisimple $x_1 \in G$ of order $q^2-1$ belonging to 
the class $B_2(1)$ and a regular semisimple $x_2 \in G$ of order $(q^2-1)/2$ belonging to 
the class $B_5(1,1)$, in the notation of \cite{Sri}. There are $3$ non-principal characters of $G$ that are nonzero at both $x_1$ and $x_2$: namely, 
$\theta_{1,2}$ of degree $q(q^2+1)/2$, and 
$\St$ of degree $q^4$. For every $1 \neq g \in G$,
$$\sum_{1_G \neq \chi \in \Irr(G)}\frac{|\chi(x_1)\chi(x_2)\chi(g)|}{\chi(1)} 
    \leq 2 \cdot \frac{q(q+1)/2}{q(q^2+1)/2} + \frac{q}{q^4} < 1,$$
so we are done by Lemma \ref{basic}(i).      
%Finally, the cases of $\Sp_4(q)$ with $q = 3,4,5$, and of $\Sp_6(2)$, $\Omega_7(3)$,
%$\Omega^\pm_8(2)$, and $\Omega^\pm_8(3)$,
%can be checked directly using \cite{Atlas} or \cite{GAP}.
\hal

\begin{lem}\label{base-clas}
Let $G$ be one of the following groups: 
\begin{enumerate}[\rm(i)]
\item $\Sp_{2n}(2)$ with  $3 \leq n \leq 6$, $\Sp_{2n}(3)$ with $2 \leq n \leq 5$, and $\Sp_{2n}(4)$
with $n = 2,3$;
\item $\Omega_{2n+1}(3)$ with $3 \leq n \leq 5$;
\item $\Omega^\pm_{2n}(2)$ with  $4 \leq n \leq 6$,  $\Omega^\pm_{2n}(3)$ with $4 \leq n \leq 6$,
and $\Omega^\pm_8(4)$.
\end{enumerate}
Let $N = p^at^b$ where $p$ is the defining characteristic
of $G$ and $t \in \cR(G)$. Then $\sP(N)$ holds. 
\end{lem}

\pf
Direct calculations similar to those of Lemma \ref{alt-spor}.
\hal

\subsection{Induction step: Symplectic groups}
\begin{prop}\label{sp-odd-large}
Suppose $G = \Sp_{2n}(q)$ with $n \geq 3$, $q = p^f \geq 7$ odd, and $t \in \cR(G)$. 
Then $\sP_u(N)$ holds for $G$ and for every $N = p^at^b$. 
\end{prop}

\pf
Consider an unbreakable $g \in G$; in particular, 
$$|\bfC_G(g)| \leq \left\{ \begin{array}{ll} 2q^n, & 2|n,\\
   q^{2n-1}(q^2-1), & 2 \nmid n \end{array} \right.$$
by Lemma \ref{spunb}. Let $V = \F_q^{2n}$ denote the natural module for $G$.
Inside $\Sp_{2n-2}(q)$ we can find a regular semisimple element $x_{-}$ of order $s_{-}=\p(q,2n-2)$, and, if $2|n$, a regular semisimple element $x_+$ of order $s_{+}=\p(q,n-1)$. For 
$\nu = \pm$, we fix $y_\nu \in \Sp_2(q)$ of order $q-\nu$.

\smallskip
(a) Here we consider the case $2|n$, and set 
$$g_1 := \diag(x_{+},y_+),~~g_2 := \diag(x_{-},y_-)$$
so each $g_i$ is an $N'$-element and $|\bfC_G(g_i)| \leq (q^{n-1}+1)(q+1)$. 
We also choose 
$$D = \frac{(q^{n}-1)(q^{n}-q)}{2(q+1)}.$$
It follows that
\begin{equation}\label{sp-sum1}
  \sum_{\chi \in \Irr(G),~\chi(1) \geq D}\frac{|\chi(g_1)\chi(g_2)\oc(g)|}{\chi(1)}
    \leq \frac{(q^{n-1}+1)(q+1)(2q^n)^{1/2}}{D} 
    < 0.54.
\end{equation}    
By \cite[Theorem 5.2]{TZ1} the only non-principal irreducible character of $G$ of degree less than $D$ 
are the four irreducible Weil characters: $\eta_{1,2}$ of degree $(q^n-1)/2$ and
$\xi_{1,2}$ of degree $(q^n+1)/2$. The choice of $g_i$ implies that $\Ker(g_i \pm 1_V) = 0$.
Hence,  by \cite[Lemma 2.4]{GT2}, 
$$|\omega(g_i)|,~ |\omega(zg_i)| \leq 1,$$
where $\omega = \eta_1 + \xi_1$ is a reducible Weil character of $G$ and $z \in G$ is the central
involution. Note that
$$|\omega(g_i)| = |\eta_1(g_i)+\xi_1(g_i)|,~~|\omega(zg_i)| = |\eta_1(g_i)-\xi_1(g_i)|.$$
It follows that 
$$|\eta_1(g_i)| = \frac{|(\eta_1(g_i)+\xi_1(g_i))+(\eta_1(g_i)-\xi_1(g_i))|}{2} 
   \leq \frac{|\omega(g_i)|+|\omega(zg_i)|}{2} \leq 1.$$
Similarly, 
\begin{equation}\label{sp-weil1}
  |\eta_j(g_i)| \leq 1, ~~|\xi_j(g_i)| \leq 1, ~~ \forall i,j = 1,2.
\end{equation} 
It follows that 
$$\sum_{\chi \in \Irr(G),~1 < \chi(1) < D}\frac{|\chi(g_1)\chi(g_2)\oc(g)|}{\chi(1)}
    \leq \frac{4\cdot (2q^n)^{1/2}}{(q^n-1)/2} 
    < 0.24.$$
Together with (\ref{sp-sum1}), this implies that    
$$\sum_{\chi \in \Irr(G),~\chi(1) > 1}\frac{|\chi(g_1)\chi(g_2)\oc(g)|}{\chi(1)}
    < 0.54 + 0.24 = 0.78,$$  
whence $g \in g_1^G \cdot g_2^G$. Since both $g_1$ and $g_2$ are $N'$-elements, we are done. 

\smallskip
(b) Next we consider the case $n \geq 3$ odd. Here we choose 
$$D = \left\{ \begin{array}{ll}(q^{2n}-1)(q^{n-1}-q)/2(q^2-1),& n \geq 5,\\
          q^4(q^3-1)(q-1)/2, & n = 3, \end{array} \right.$$
so 
\begin{equation}\label{sp-sum2}
\sum_{\chi \in \Irr(G),~\chi(1) \geq D}\frac{|\chi(g_1)\chi(g_2)\oc(g)|}{\chi(1)}
    \leq \frac{(q^{n-1}+1)(q+1)(q^{2n-1}(q^2-1))^{1/2}}{D} 
    < 0.15.
\end{equation}    
Using \cite[Theorem 1.1]{Ng} for $n \geq 7$ and \cite{Lu} for $n=5$, we show that every 
non-principal irreducible character of $G$ of degree less than $D$ is one of the following: 

(b1) four irreducible Weil characters $\eta_{1,2}$, $\xi_{1,2}$ as above; 

(b2) four unipotent characters $\alpha_\nu$, $\beta_\nu$, $\nu = \pm$, of degree
$$\alpha_\nu(1) = \frac{(q^n-\nu)(q^n+\nu q)}{2(q-1)},~~
    \beta_\nu(1) = \frac{(q^n+\nu)(q^n+\nu q)}{2(q+1)};$$

(b3) two characters of degree $(q^{2n}-1)/2(q+1)$, two of degree 
$(q^{2n}-1)/2(q-1)$, $(q-1)/2$ of degree $(q^{2n}-1)/(q+1)$, and
$(q-3)/2$ of degree $(q^{2n}-1)/2(q-1)$.

If $n = 3$, then we check using \cite{Lu} that the characters $\chi \in \Irr(G)$ with $1 < \chi(1) < D$ 
and such that $\chi$ has positive $s$-defect and positive $s_{-}$-defect are described 
in (b1) and (b2). Thus, in all cases, in considering characters of $G$ of degree less than $D$ we can
restrict to the ones in (b1)--(b3).

\smallskip
Since $t \in \cR(G)$, there is an $\e = \pm$ such that $t|(q^n-\e)$. Now, we choose a regular 
semisimple element $g_1$ of order $s \in \cR(G) \setminus \{t\}$ and 
take $g_2 := \diag(x_{-},h_\e)$. In particular, $|\bfC_G(g_i)| \leq (q^{n-1}+1)(q+1)$.
Note that all characters in (b3) have $s$-defect $0$, so vanish 
at $g_1$. Next, $\beta_\e$ and $\alpha_{-\e}$ have $s$-defect $0$, whence
$$\beta_\e(g_1) = \alpha_{-\e}(g_1) = 0.$$
Likewise, $\beta_+$ and $\alpha_+$ have $s_-$-defect $0$, whence
$$\beta_+(g_2) = \alpha_+(g_2) = 0.$$

Consider the case $\e = -$. We have shown that $\chi(g_1)\chi(g_2) = 0$ for
$\chi = \alpha_+$, $\beta_+$, $\beta_-$, and 
$$\alpha_+(g_1)=\alpha_+(g_2) = 0.$$ 
On the other hand, $\rho := 1_G + \alpha_+ + \alpha_-$ is just the permutation character of the action
of $G$ on the set of $1$-spaces of $V$, cf.\ \cite[Table 2]{ST}. 
The choice of $g_i$ ensures that $\rho(g_i) = 0$, 
whence
$$\alpha_-(g_1) = \alpha_-(g_2) = -1.$$

Assume now that $\e = +$. We have shown that $\chi(g_1)\chi(g_2) = 0$ for
$\chi = \alpha_+$, $\beta_+$, $\alpha_-$, and 
$$\beta_+(g_1)=\beta_+(g_2) = 0.$$ 
On the other hand, as shown in \cite{T}, $\zeta := \beta_+ + \beta_-$ is just the restriction
to $G$ of the unipotent Weil character $\zeta_{0,0}$ of $\GU_{2n}(q)$ (as defined in the 
proof of Proposition \ref{su-large}) when we embed 
$$G = \Sp_{2n}(q) \hookrightarrow \SU_{2n}(q) \lhd \GU_{2n}(q).$$
The choice of $g_i$ ensures that $\zeta(g_i) = 0$, 
whence
$$\beta_-(g_1) = \beta_-(g_2) = -1.$$
The same arguments as in (a) show that (\ref{sp-weil1}) holds in this case as well. Observe 
that, for $\mu =\pm 1$, $U_{\mu} := \Ker(g-\mu \cdot 1_V)$ has dimension at most $n$, as 
otherwise it cannot be totally isotropic, so $g$ acts as the multiplication by $\mu$ on 
a $2$-dimensional non-degenerate subspace of $U$, contrary to the assumption that 
$g$ is unbreakable. Using \cite[Lemma 2.4]{GT2}, we see that 
$$|\omega(g)|,~|\omega(zg)| \leq q^{n/2},$$
so, arguing as in the above proof of (\ref{sp-weil1}), we obtain
$$|\eta_i(g)|, ~|\xi_i(g)| \leq q^{n/2}.$$
Certainly, $|\gamma(g)| \leq |\bfC_G(g)|^{1/2} \leq (q^{2n-1}(q^2-1))^{1/2}$ for 
$\gamma = \alpha_-$, $\beta_-$. 
In summary, 
$$\sum_{\chi \in \Irr(G),~1 < \chi(1) < D}\frac{|\chi(g_1)\chi(g_2)\oc(g)|}{\chi(1)}
    \leq \frac{4\cdot q^{n/2}}{(q^n-1)/2} + \frac{(q^{2n-1}(q^2-1))^{1/2}}{(q^n-1)(q^n-q)/2(q-1)}  
    < 0.53.$$
Together with (\ref{sp-sum2}), this implies that    
$$\sum_{\chi \in \Irr(G),~\chi(1) > 1}\frac{|\chi(g_1)\chi(g_2)\oc(g)|}{\chi(1)}
    < 0.15 + 0.53 = 0.68,$$  
whence $g \in g_1^G \cdot g_2^G$. Since both $g_1$ and $g_2$ are $N'$-elements, we are done. 
\hal

To handle the symplectic groups over $\F_3$, we need an explicit description of 
low-degree complex characters of $\Sp_{2n}(3)$.

\begin{lem}\label{sp3-lowdim}
Let $G = \Sp_{2n}(3)$ with $n \geq 6$ and let $D := (3^{2n}-1)(3^{n-1}-3)/16$. Then
$$\{ \chi \in \Irr(G) \mid 1 < \chi(1) < D\}$$
consists of the following $13$ characters:
\begin{enumerate}[\rm(i)]

\item four irreducible Weil characters $\eta$, $\bar\eta$ of degree $(3^n-1)/2$, $\xi$, $\bar\xi$ 
of degree $(3^n+1)/2$; 

\item four characters $\SQ(\xi)$, $\wedge^2(\eta)$, $\xi\bar\xi-1_G$, and 
$\eta\bar\eta-1_G$, of respective degree
$$\frac{(3^n+1)(3^n+3)}{8},~~\frac{(3^n-1)(3^n-3)}{8},~~\frac{(3^n-1)(3^n+3)}{4},~~
   \frac{(3^n+1)(3^n-3)}{4};$$

\item two characters $\SQ(\eta)$, $\wedge^2(\xi)$ of degree $(3^{2n}-1)/8$, and three 
characters $\xi\bar\eta$, $\bar\xi\eta$, $\xi\eta = \bar\xi\bar\eta$ of degree 
$(3^{2n}-1)/4$.
\end{enumerate}
Also, $\SQ(\eta) = \bar\wedge^2(\xi)$.
\end{lem}

\pf
Applying \cite[Theorem 1.1]{Ng}, we deduce that the degrees, and the multiplicity for each degree of 
non-principal irreducible character of $G$ of degree less than $D$ are as listed above. 
%Among them, the four Weil characters are listed in (i). 
The proof of \cite[Proposition 5.4]{MT} shows that 
the six characters $\SQ(\xi)$, $\wedge^2(\eta)$, $\xi\bar\xi-1_G$, $\eta\bar\eta-1_G$, $\SQ(\eta)$, 
and $\wedge^2(\xi)$ have the degrees listed in (ii) and (iii). It also shows that 
$\xi\bar\eta$ and $\bar\xi\eta$ are two distinct irreducible constituents (of a certain real character
$\tau$) of degree $(3^{2n}-1)/4$, so they are non-real. On the other hand, $\xi\eta$ is the 
unique irreducible constituent of degree $(3^{2n}-1)/4$ of a certain real character
$\sigma$, whence it must be real. We have therefore identified the three characters of 
degree $(3^{2n}-1)/4$. Finally,
$$[\SQ(\eta)+\wedge^2(\eta),\bar\SQ(\xi)+\bar\wedge^2(\xi)] = [\eta^2,\bar\xi^2]
    = [\xi\eta,\bar\xi\bar\eta]= 1,$$
so $\SQ(\eta) = \bar\wedge^2(\xi)$, since the involved characters are all irreducible, and only
$\SQ(\eta)$ and $\bar\wedge^2(\xi)$ have equal degree.   
\hal

\begin{prop}\label{sp3-large}
Suppose $G = \Sp_{2n}(3)$ with $n \geq 6$, and $t \in \cR(G)$. 
Then $\sP_u(N)$ holds for $G$ and for every $N = 3^at^b$. 
\end{prop}

\pf
(i) Consider an unbreakable $g \in G$; in particular, 
\begin{equation}\label{sp3-cent}
  |\bfC_G(g)| \leq 16 \cdot 3^{2n+2}
\end{equation}
by Lemma \ref{spunb}. Let $V = \F_3^{2n}$ denote the natural module for $G$.
Inside $\Sp_{2n-2}(3)$ we can find a regular semisimple element $x_{-}$ of order $s_{-}=\p(3,2n-2)$ and a regular semisimple element $x_+$ of order $s_{+}=\p(3,n-1)$. We fix 
$y \in \Sp_2(3)$ of order $4$. If $n$ is even, we set
$$g_1 := \diag(x_{+},y),~~g_2 := \diag(x_{-},y),$$
whereas for odd $n$, we choose a regular semisimple $g_1 \in G$ of order
$s \in \cR(G) \setminus \{t\}$ and set $g_2:= \diag(x_-,y)$. In particular,
$g_i$ is an $N'$-element and $|\bfC_G(g_i)| \leq 4 \cdot (3^{n-1}+1)$ for $i = 1,2$.
We also choose 
$$D = \frac{(3^{2n}-1)(3^{n-1}-3)}{16}.$$
Then the characters $\chi \in \Irr(G)$ with $1 < \chi(1)< D$ are described in Lemma \ref{sp3-lowdim}.

The choice of $g_i$ implies that $\Ker(g_i \pm 1_V) = 0$. Hence, as in the proof of Proposition
\ref{sp-odd-large}, 
\begin{equation}\label{sp3-weil1}
  |\chi(g_i)| \leq 1, ~~ \forall \chi \in \{\xi,\eta\}.
\end{equation}   
On the other hand, $\dim_{\F_3}\Ker(g \pm 1_V) \leq 4$ by Lemma \ref{blocks}. Arguing as in 
part (a) of the proof of Proposition \ref{sp-odd-large}, we obtain 
\begin{equation}\label{sp3-weil2}
|\chi(g)| \leq 3^2, ~~ \forall \chi \in \{\xi,\eta\}.
\end{equation}
It follows that 
\begin{equation}\label{sp3-sum1}
  \sum_{\chi \in \{\xi,\bar\xi,\eta,\bar\eta\}}\frac{|\chi(g_1)\chi(g_2)\oc(g)|}{\chi(1)}
    \leq \frac{4\cdot 3^2}{(3^n-1)/2} 
    < 0.099.
\end{equation}  
Let $\cX$ denote the set of nine characters listed in Lemma \ref{sp3-lowdim}(ii), (iii).  
Observe that $x_\nu$ has prime order $s_\nu$ for $\nu = \pm$. Hence 
$$\Ker(g_i^2 -1_V) = 0,~~\dim_{\F_3}\Ker(g_i^2+1_V) \leq 2.$$ 
This in turn implies that 
$$|\omega(g_i^2)| \leq 1,~~|\omega(zg_i^2)| \leq 3$$
for the reducible Weil character $\omega = \xi+\eta$ and the central involution $z \in G$. Arguing
as in part (a) of the proof of Proposition \ref{sp-odd-large}, we obtain
$$|\chi(g_i^2)| \leq (1+3)/2 = 2, ~~ \forall \chi \in \{\xi,\eta\}.$$
Together with (\ref{sp3-weil1}), this implies that 
\begin{equation}\label{sp3-weil3}
  |\chi(g_i)| \leq 3/2, ~~ \forall \chi \in \cX.
\end{equation}

\smallskip
(ii) Here we assume that $n \geq 7$.
If $2|n$, then the four characters listed in Lemma \ref{sp3-lowdim} have either $s_+$-defect $0$,
or $s_-$-defect $0$. If $2\nmid n$, then the five characters listed in Lemma \ref{sp3-lowdim} have 
$s$-defect $0$. Thus, at most five characters from $\cX$ can be nonzero at both $g_1$ and $g_2$. 
Also, $|\chi(g)| \leq 4 \cdot 3^{n+1}$ for all $\chi \in \Irr(G)$ by (\ref{sp3-cent}). Using 
(\ref{sp3-weil3}), we see that
$$\sum_{\chi \in \cX}\frac{|\chi(g_1)\chi(g_2)\oc(g)|}{\chi(1)}
    \leq \frac{5 \cdot (3/2)^2 \cdot 4 \cdot 3^{n+1}}{(3^n-1)(3^n-3)/8} 
    < 0.495.$$
On the other hand, 
$$\sum_{\chi \in \Irr(G),~\chi(1) \geq D}\frac{|\chi(g_1)\chi(g_2)\oc(g)|}{\chi(1)}
    \leq \frac{4(3^{n-1}+1) \cdot 4 \cdot 3^{n+1})}{D} 
    < 0.354.$$ 
Together with (\ref{sp3-sum1}), these estimates imply that    
$$\sum_{\chi \in \Irr(G),~\chi(1) > 1}\frac{|\chi(g_1)\chi(g_2)\oc(g)|}{\chi(1)}
    < 0.099 + 0.495 + 0.354 = 0.948,$$  
whence $g \in g_1^G \cdot g_2^G$. Since both $g_1$ and $g_2$ are $N'$-elements, we are done. 

\smallskip
(iii) We may now assume that $n=6$. In this case $\cR(G) = \{7, 13,73\}$ and 
$|g_1| = 44$, $|g_2| = 244$. Using \cite{Lu}, we check that $G$ has exactly $30$ irreducible 
characters $\chi$ that have both positive $11$-defect and positive $41$-defect: namely, $1_G$, 
four Weil characters, five characters from $\cX$ and listed in Lemma \ref{sp3-lowdim}(iii), four 
characters $\psi_{1,2,3,4}$ with two of each of the degrees
$$D = 15 \cdot (3^{12}-1), ~~D_1 := 15 \cdot (3^4+1) \cdot (3^8+3^4+1),$$
and $16$ more, of degree larger than $D_2 := 3^{19}$.  In particular,
\begin{equation}\label{sp3-sum2}
\sum_{\chi \in \Irr(G),~\chi(1) \geq D_2}\frac{|\chi(g_1)\chi(g_2)\oc(g)|}{\chi(1)}
    \leq \frac{4 \cdot (3^{n-1}+1) \cdot 4 \cdot 3^{n+1}}{D_2} 
    < 0.0074.
\end{equation}    

Next we strengthen the bound on $|\chi(g)|$ for $\chi \in \cX$. Consider 
$\lambda = \pm 1$ and write $g = uv=vu$, with $u$ unipotent and $v$ semisimple.
Let $\tilde V := V \otimes_{\F_3}\overline\F_3$. Note that if 
$w \in U_\lambda := \Ker(g^2-\lambda \cdot 1_{\tilde V})$, then $w$ belongs to 
$W_\mu := \Ker(v-\mu \cdot 1_{\tilde V})$ for some $\mu$ with $\mu^2 = \lambda$. Now $g^2$ 
acts on $W_\mu$ as $\lambda u'^2$, where $u' := u_{W_\mu}$ is unipotent. Next, observe that
$$\dim_{\overline\F_3}\Ker(u'^2-1_{W_\mu}) = \dim_{\overline\F_3}\Ker(u'-1_{W_\mu}).$$
It then follows from Lemma \ref{blocks} that 
$$\dim_{\overline\F_3}U_\lambda = \dim_{\overline\F_3}\Ker(g-\mu_0\cdot 1_{\tilde V}) + 
    \dim_{\overline\F_3}\Ker(g+\mu_0 \cdot 1_{\tilde V}) \leq 8,$$
where $\mu_0$ is a fixed square root $\mu_0$ of $\lambda$. In turn, this implies by
\cite[Lemma 2.4]{GT2} that
$$|\omega(g^2)|,~|\omega(zg^2)| \leq 3^4.$$
Arguing as in part (a) of the proof of Proposition \ref{sp-odd-large}, we obtain 
$$|\chi(g^2)| \leq 3^4, ~~ \forall \chi \in \{\xi,\eta\}.$$ 
Using this bound and (\ref{sp3-weil2}), we see that 
$$|\chi(g)| \leq 3^4, ~~ \forall \chi \in \cX.$$
Since only five characters from $\cX$ can be nonzero at both $g_1$ and $g_2$, this last estimate
together with (\ref{sp3-weil3}) yields 
\begin{equation}\label{sp3-sum3}
  \sum_{\chi \in \cX}\frac{|\chi(g_1)\chi(g_2)\oc(g)|}{\chi(1)}
    \leq \frac{5 \cdot (3/2)^2 \cdot 3^4}{(3^{2n}-1)/8} 
    < 0.0138.
\end{equation}

Finally, we estimate character ratios for the four characters $\psi_{1,2,3,4}$ of degree $D$ and $D_1$. 
Since $|g_2| = 4 \cdot 61$,
$\chi(g) = 0$ if and only if $\chi \in \Irr(G)$ has degree divisible by $61$. 
Using \cite{Lu}, we check that 
$$\Irr_{61'}(G) := \{ \chi \in \Irr(G) \mid 61 \nmid \chi(1) \}$$
consists of exactly $343$ characters. (Another way to check it is to observe that 
since $P \in \Syl_{61}(G)$ is cyclic, the {\it McKay conjecture} holds for $G$, i.e. 
$$|\Irr_{61'}(G)| = |\Irr_{61'}(\bfN_G(P))|.$$
Direct computation shows that 
$$\bfN_G(P) = (C_{244} \rtimes C_{10}) \times \Sp_2(3)$$
has exactly $343$ irreducible characters of degree coprime to $61$.) Certainly, 
$$\Irr_{61'}(G) \setminus \{\psi_{1,2,3,4}\}$$
is a union of some $\Gal(\overline\Q/\Q)$-orbits. Hence, by Lemma \ref{orbit}, 
$$\sum_{\chi \in \Irr(G) \setminus \{\psi_{1,2,3,4}\}}|\chi(g_2)|^2 
   =  \sum_{\chi \in \Irr_{61'}(G) \setminus \{\psi_{1,2,3,4}\}}|\chi(g_2)|^2 \geq 
   |\Irr_{61'}(G) \setminus \{\psi_{1,2,3,4}\}| = 343-4 = 339.$$
Since $\sum_{\chi \in \Irr(G)}|\chi(g_2)|^2 = |\bfC_G(g_2)| = 4 \cdot (3^5+1) = 976$, 
$$\sum^4_{j=1}|\psi_j(g_2)|^2 \leq 976-339 = 637.$$
Recall that 
$|\psi_j(g)| \leq 4 \cdot 3^7$ by (\ref{sp3-cent}) and 
$|\psi_j(g_1)|^2 \leq |\bfC_G(g_1)| = 4 \cdot (3^5-1) = 968$.
By the Cauchy-Schwarz inequality, 
$$\sum^4_{j=1}\frac{|\psi_j(g_1)\psi_j(g_2)\overline\psi_j(g)|}{\psi_j(1)} 
    \leq \frac{4 \cdot 3^7 \cdot 968^{1/2}}{D} \cdot (\sum^4_{j=1}|\psi_j(g_2)|^2)^{1/2}
    = \frac{4 \cdot 3^7 \cdot (968 \cdot 637)^{1/2}}{15 \cdot (3^{12}-1)} < 0.8618.
$$
Together with (\ref{sp3-sum1}), (\ref{sp3-sum2}), (\ref{sp3-sum3}), this implies that    
$$\sum_{\chi \in \Irr(G),~\chi(1) > 1}\frac{|\chi(g_1)\chi(g_2)\oc(g)|}{\chi(1)}
    < 0.099 + 0.0138 + 0.8618 +0.0074 = 0.982,$$  
whence $g \in g_1^G \cdot g_2^G$. Since both $g_1$ and $g_2$ are $N'$-elements, we are again
done. 
\hal

%\subsection{Induction step: Symplectic groups in even characteristic}
\begin{prop}\label{sp-even}
Suppose $G = \Sp_{2n}(q)$ with $n \geq 3$, $2 | q$, and $t \in \cR(G)$. 
Assume that $n \geq 4$ if $q = 4$, and $n \geq 7$ if $q=2$.
Then $\sP_u(N)$ holds for $G$ and for every $N = 2^at^b$. 
\end{prop}

\pf
Consider an unbreakable $g \in G$; in particular, 
$$|\bfC_G(g)| \leq B := \left\{ \begin{array}{ll} q^{2n}(q^2-1), & 2|n, q \geq 4\\
   2q^{2n}(q+1), & 2 \nmid n, q \geq 4 \\
   9 \cdot q^{2n+9}, & q = 2\end{array} \right.$$
by Lemma \ref{spunb}. Let $V = \overline\F_q^{2n}$ denote the natural module for $G$.
Inside $\Sp_{2n-2}(q)$ we can find a regular semisimple element $x_{-}$ of order $s_{-}=\p(q,2n-2)$, and, if $2|n$, a regular semisimple element $x_+$ of order $s_{+}=\p(q,n-1)$. 
We fix $y \in \Sp_2(q)$ of order $q+1$. Let $\cW$ denote the set of $q+3$ Weil characters 
$$\alpha_n, ~\beta_n, ~~\rho^1_n, ~\rho^2_n,~~ \zeta^i_n,~1 \leq i \leq q/2,~~\tau^j_n,~
    1 \leq j \leq q/2-1$$ 
(as described in \cite[Table 1]{GT2}). Assuming $n \geq 4$ and choosing
$$D := \frac{(q^{2n}-1)(q^{n-1}-1)(q^{n-1}-q^2)}{2(q^4-1)},$$
we see by \cite[Corollary 6.2]{GT2} that $\cW$ is precisely the set 
$\{ \chi \in \Irr(G) \mid 1 < \chi(1) < D\}$.

\smallskip
(i) Here we consider the case $2|n$, and set 
$$g_1 := \diag(x_{+},y),~~g_2 := \diag(x_{-},y)$$
so that each $g_i$ is an $N'$-element and $|\bfC_G(g_i)| \leq (q^{n-1}+1)(q+1)$. In particular,
\begin{equation}\label{sp2-sum1}
\sum_{\tiny{\begin{array}{l}\chi \in \Irr(G),\\ \chi(1) \geq D\end{array}}}
    \frac{|\chi(g_1)\chi(g_2)\oc(g)|}{\chi(1)}
    \leq \frac{(q^{n-1}+1)(q+1) \cdot B^{1/2}}{D} 
    < \left\{ \begin{array}{ll}0.8293, & n = 4\\
    0.1956, & n \geq 6. \end{array} \right.      
\end{equation}
If $\chi \in \{\alpha_n,\beta_n,\rho^1_n,\rho^2_n\}$ then $\chi$ has $s_\nu$-defect
$0$ for some $\nu = \pm$, so $\chi(g_1)\chi(g_2) = 0$.  
For $\gamma \in \overline\F_q^\times$, the choice of $g_i$ implies that 
$\dim_{\overline\F_q}\Ker(g_i - \gamma \cdot 1_V)$ equals $0$ if $\gamma^{q-1} = 1$, and 
is at most $1$ if $\gamma^{q+1} = 1$; 
in fact, it equals $1$ for exactly two primitive $(q+1)$th roots of 
unity in $\overline\F_q^\times$.
Hence, by formulae (1) and (4) of \cite{GT2}, 
$$|\tau^j_n(g_i)| = 0,~~|\zeta^j_n(g_i)| \leq b,$$
where $b := 2$ if $q \geq 4$ and $b := 1$ if $q = 2$.
For $n \geq 6$, it follows that 
$$\sum_{\chi \in \Irr(G),~1 < \chi(1) < D}\frac{|\chi(g_1)\chi(g_2)\oc(g)|}{\chi(1)}
    \leq \frac{q}{2} \cdot \frac{b^2 \cdot (q^{2n}(q^2-1))^{1/2}}{(q^{2n}-1)/(q+1)} 
    < 0.7956.$$
Suppose that $n = 4$ and $q \geq 4$. Observe that 
$\dim_{\overline\F_q}\Ker(g_i - \gamma \cdot 1_V) \leq 4$ 
for $\gamma \in \overline\F_q^\times$ with $\gamma^{q + 1} = 1$. (Indeed, this bound is
obvious if $\gamma \neq 1$. If $\gamma = 1$, it follows from the condition that
$g$ is unbreakable.) Hence, formula (4) of \cite{GT2} implies that  
$|\zeta^j_n(g)| \leq q^4$, so 
$$\sum_{\chi \in \Irr(G),~1 < \chi(1) < D}\frac{|\chi(g_1)\chi(g_2)\oc(g)|}{\chi(1)}
    \leq \frac{q}{2} \cdot \frac{b^2 \cdot q^{4}}{(q^{8}-1)/(q+1)} 
    < 0.1564.$$
Together with (\ref{sp2-sum1}), this implies that    
$$\sum_{\chi \in \Irr(G),~\chi(1) > 1}\frac{|\chi(g_1)\chi(g_2)\oc(g)|}{\chi(1)}
    < \left\{ \begin{array}{ll}0.8293 + 0.1564 = 0.9857, & n = 4\\
    0.1956 + 0.7956 = 0.9912, & n \geq 6 \end{array} \right.$$  
whence $g \in g_1^G \cdot g_2^G$. Since both $g_1$ and $g_2$ are $N'$-elements, we are done. 

\smallskip
(ii) From now on we assume $2 \nmid n$. By Proposition \ref{ratio} we may assume
that $n \geq 5$. We choose a regular 
semisimple element $g_1$ of order $s \in \cR(G) \setminus \{t\}$ and 
take $g_2 := \diag(x_{-},y)$. In particular, again $|\bfC_G(g_i)| \leq (q^{n-1}+1)(q+1)$.
Note that all characters $\zeta^j_n$ and $\tau^j_n$ have $s$-defect $0$, so vanish 
at $g_1$. Next, the choice of $g_i$ implies that, for $\gamma \in \overline\F_q^\times$, 
$\dim_{\overline\F_q}\Ker(g_i - \gamma \cdot 1_V)$ equals $0$ if $\gamma^{q-1} = 1$, and 
is at most $1$ if $\gamma^{q+1} = 1$; in fact, it equals $1$ for exactly two 
primitive $(q+1)$th roots of 
unity in $\overline\F_q^\times$. Using formulae (1), (3), (4), and (6)  of \cite{GT2}, we obtain  
$$(\rho^1_n+\rho^2_n)(g_i) = -1,~~(\alpha_n+\beta_n)(g_1) = 1,~~(\alpha_n+\beta_n)(g_2) = -1.$$
Furthermore, exactly one character among $\alpha_n$, $\beta_n$, and exactly one character 
among $\rho^1_n$, $\rho^2_n$, have $s$-defect zero. It follows that 
$$|\chi(g_1)| \leq 1,~\forall \chi \in \{\alpha_n,\beta_n,\rho^1_n,\rho^2_n\}.$$
Likewise, $\beta_n$ and $\rho^2_n$ have $s_-$-defect $0$, so
$$\beta_n(g_2) = \rho^2_n(g_2) = 0,~~|\alpha_n(g_2)| = |\rho^1_n(g_2)| = 1.$$
We also observe that $\alpha_n(g_1) = 0$ if $s|(q^n-1)$ and $\rho^1_n(g_1) = 0$ if $s|(q^n+1)$.
We have shown that, among the characters in $\cW$, exactly one character can be nonzero at both $g_1$ and $g_2$. Denoting this character by $\psi$, 
\begin{equation}\label{sp2-psi}
  |\psi(g)|/\psi(1) \leq 0.95, ~|\psi(g)| \leq B^{1/2},~~|\psi(g_i)| \leq 1.
\end{equation}
Here, the first bound follows from the main result of \cite{Glu}.

\smallskip
(iii) Assume in addition that $n \geq 9$ if $q = 2$. Now 
$$\sum_{\chi \in \Irr(G),~1 < \chi(1) < D}\frac{|\chi(g_1)\chi(g_2)\oc(g)|}{\chi(1)}
    \leq \frac{B^{1/2}}{(q^n-1)(q^n-q)/2(q+1)} 
    < 0.8003.$$
On the other hand,
$$\sum_{\chi \in \Irr(G),~\chi(1) \geq D}
    \frac{|\chi(g_1)\chi(g_2)\oc(g)|}{\chi(1)}
    \leq \frac{(q^{n-1}+1)(q+1) \cdot B^{1/2}}{D} 
    < 0.0478.$$  
It follows that      
$$\sum_{\chi \in \Irr(G),~\chi(1) > 1}\frac{|\chi(g_1)\chi(g_2)\oc(g)|}{\chi(1)}
    < 0.8003 + 0.0478 = 0.8481,$$  
whence $g \in g_1^G \cdot g_2^G$. Since both $g_1$ and $g_2$ are $N'$-elements, we are done. 

\smallskip
(iv) Now we consider the case $(n,q)= (7,2)$ and choose 
$$D_1 := \frac{q^{35}(q^7-1)(q^7-q)}{2(q+1)}.$$
Using \cite{Lu}, we check that there is only one character $\chi \in \Irr(G)$ with 
$1 < \chi(1) < D_1$ that has both positive $s$-defect and $s_-$-defect, namely the character 
$\psi$ described in (ii). Now using (\ref{sp2-psi}) 
$$\sum_{\chi \in \Irr(G),~1 < \chi(1) < D_1}\frac{|\chi(g_1)\chi(g_2)\oc(g)|}{\chi(1)}
    \leq \frac{|\psi(g)|}{\psi(1)} 
    < 0.95.$$
On the other hand,
$$\sum_{\chi \in \Irr(G),~\chi(1) \geq D_1}
    \frac{|\chi(g_1)\chi(g_2)\oc(g)|}{\chi(1)}
    \leq \frac{(q^{6}+1)(q+1) \cdot B^{1/2}}{D_1} 
    < 0.01.$$  
It follows that      
$$\sum_{\chi \in \Irr(G),~\chi(1) > 1}\frac{|\chi(g_1)\chi(g_2)\oc(g)|}{\chi(1)}
    < 0.95 + 0.01 = 0.96,$$
and we are done again.        
\hal

\subsection{Induction step: Orthogonal groups}
\begin{prop}\label{so-odd}
Suppose $G = \Omega_{2n+1}(q)$ with $n \geq 3$, $q = p^f$ odd, and $t \in \cR(G)$. 
Assume that $n \geq 6$ if $q=3$.
Then $\sP_u(N)$ holds for $G$ and for every $N = p^at^b$. 
\end{prop}

\pf
By Corollary \ref{main-real} and Proposition \ref{ratio}, we may assume that $q = 3$ and $n \geq 6$.
Let $V = \overline\F_q^{2n+1}$ denote the natural module for $G$, and let 
$$F_0 := \{ \gamma \in \overline\F_q^\times \mid \gamma^{q \pm 1} = 1\}.$$
Consider an unbreakable $g \in G$; in particular, 
$|\bfC_G(g)| \leq B := 2^4 \cdot q^{2n+3}$
by Lemma \ref{orthogunb}. Let $\cX$ denote the set of $q+4$ characters 
described in \cite[Proposition 5.7]{ore}: each is of the form 
$D^\circ_\alpha$ for $\alpha \in \Irr(S)$ and $S := \Sp_2(q)$. Choosing
$$D := q^{4n-8},$$
we see by \cite[Corollary 5.8]{ore} that $\cX$ is precisely the set 
$\{ \chi \in \Irr(G) \mid 1 < \chi(1) < D\}$. If $\gamma \in F_0$, then 
$$\dim_{\overline\F_q}\Ker(g-\gamma \cdot 1_V) \leq 4$$
by Lemma \ref{blocks}. Following the proof of \cite[Proposition 5.11]{ore}, one can show that 
\begin{equation}\label{so1-weil1}
   |D_\alpha(g)| \leq  q^4 \cdot \alpha(1).
\end{equation}
Now we choose $g_1 = g_2$ to be a regular semisimple
element of order $s \in \cR(G) \setminus \{t\}$, 
so that $g_i$ is an $N'$-element and $|\bfC_G(g_i)| \leq (q^{n-1}+1)(q+1)$. In particular,
\begin{equation}\label{so1-sum1}
\sum_{\tiny{\begin{array}{l}\chi \in \Irr(G),\\ \chi(1) \geq D\end{array}}}
    \frac{|\chi(g_1)\chi(g_2)\oc(g)|}{\chi(1)}
    \leq \frac{(q^{n-1}+1)(q+1) \cdot B^{1/2}}{D} 
    < 0.35. 
\end{equation}
The choice of $g_i$ implies that 
$$\dim_{\overline\F_q}\Ker(g-\gamma \cdot 1_V) \leq 1$$
for all $\gamma \in F_0$. Following the proof of \cite[Proposition 5.11]{ore}, one can show that 
\begin{equation}\label{so1-weil2}
   |D_\alpha(g)| \leq  q \cdot \alpha(1).
\end{equation}
In the notation of \cite[Table I]{ore}, if $\alpha \neq \xi_{1,2}$, then $D^\circ_\alpha = D_{\alpha}$.
In this case, it follows from (\ref{so1-weil1}) and (\ref{so1-weil2}) for $\chi = D^\circ_\alpha$ that
$$\frac{|\chi(g_1)\chi(g_2)\oc(g)|}{\chi(1)} \leq \frac{q^6 \cdot \alpha(1)^3}{\chi(1)} 
    < (1.1)\frac{q^6 \cdot \alpha(1)^2}{(q^{2n}-1)/(q^2-1)}.$$ 
In the case $\alpha = \xi_{1,2}$ (of degree $(q+1)/2$), for 
$\chi = D^\circ_\alpha = D_\alpha-1_G$, 
$$\frac{|\chi(g_1)\chi(g_2)\oc(g)|}{\chi(1)} \leq \frac{(q^4\alpha(1)+1)(q\alpha(1)+1)^2}{\chi(1)} 
    < (1.4)\frac{q^6 \cdot \alpha(1)^2}{(q^{2n}-1)/(q^2-1)}.$$ 
It follows that 
$$\begin{array}{ll}\sum_{\chi \in \Irr(G),~1 < \chi(1) < D}\dfrac{|\chi(g_1)\chi(g_2)\oc(g)|}{\chi(1)} &
    \leq (1.4) \dfrac{q^6}{(q^{2n}-1)/(q^2-1)} \cdot \sum_{\alpha \in \Irr(S)}\alpha(1)^2 \\ \\
    & = (1.4)\dfrac{q^6 \cdot q(q^2-1)}{(q^{2n}-1)/(q^2-1)}
    < 0.26.\end{array}$$
Together with (\ref{so1-sum1}), this implies that    
$$\sum_{\chi \in \Irr(G),~\chi(1) > 1}\frac{|\chi(g_1)\chi(g_2)\oc(g)|}{\chi(1)}
   < 0.35 + 0.26 = 0.61.$$  
whence $g \in g_1^G \cdot g_2^G$. Since both $g_1$ and $g_2$ are $N'$-elements, we are done. 
\hal

\begin{prop}\label{so2-even}
Suppose $G = \O^\e_{2n}(q)$ and $q = 2,4$, $\e = \pm$, and $t \in \cR(G)$. 
Assume that $n \geq 5$ if $q = 4$, and $n \geq 7$ if $q=2$.
Then $\sP_u(N)$ holds for $G$ and for every $N = 2^at^b$. 
\end{prop}

\pf
(i) Consider an unbreakable $g \in G$; in particular, 
$$|\bfC_G(g)| \leq B := \left\{ \begin{array}{rl} 3 \cdot q^{2n+6}, & q = 2\\
   25 \cdot q^{2n-2}, & q = 4\end{array} \right.$$
by Lemma \ref{orthogunb}. We also choose 
$$D  :=  \left\{ \begin{array}{ll}q^{4n-10}, & n \geq 6, (n,q) \neq (7,2),\\
     q^{4n-8}, & (n,q) = (7,2),\\
    q^3(q^3-1)(q^5-1)(q-1)^2/2, & n = 5, q=4.\end{array} \right.$$ 
Consider the prime $s \in \cR(G) \setminus \{t\}$. 
If $s|(q^{n-1}+1)$ with $q = 2$ and $\e = +$, then we choose $g_1 = \diag(x_1,y_1)$,
where $x_1 \in \O^-_{2n-2}(q)$ is regular semisimple of order $s$ and 
$y_1 \in \O^-_2(q)$ has order $q+1$. In all other cases, we choose a regular semisimple
$g_1 \in G$ of order $s$. 

 If $s|(q^{n-1}+1)$ and $(n,q,\e) = (7,2,+)$, then choose 
$g_2 := \diag(x_2,y_2)$,
where $x_2 \in \O^{-}_{2n-4}(q)$ is regular semisimple of order $s_{\e}=\p(q,2n-4)=11$, and 
$y_2 \in \O^-_4(q)$ of order $\p(q,4) = 5$. In all other cases,
let $g_2 := g_1$.

Our choices of $g_i$ imply that each
$g_i$ is an $N'$-element, and $|\bfC_G(g_i)| \leq (q+1)(q^{n-1}+1)$. It follows that
\begin{equation}\label{so2-sum1}
\sum_{\tiny{\begin{array}{l}\chi \in \Irr(G),\\ \chi(1) \geq D\end{array}}}
    \frac{|\chi(g_1)\chi(g_2)\oc(g)|}{\chi(1)}
    \leq \frac{(q^{n-1}+1)(q+1) \cdot B^{1/2}}{D} 
    < \left\{ \begin{array}{ll}0.10, & n \geq 5, q =4\\
    0.33, & n \geq 7, q =2 . \end{array} \right.      
\end{equation}
 
\smallskip    
(ii) Now we estimate character values for the characters in     
$$\cX:= \{ \chi \in \Irr(G) \mid 1 < \chi(1) < D\}.$$
By \cite[Theorem 1.3]{Ng}, when $n \geq 6$ and $(n,q) \neq (7,2)$ the set
$\cX$ consists of $q+1$ characters: 

$\bullet$ $\varphi$ of degree $(q^n-\e)(q^{n-1}+\e q)/(q^2-1)$, 

$\bullet$ $\psi$ of degree $(q^{2n}-q^2)/(q^2-1)$,

$\bullet$ $\zeta_i$ of degree $(q^n-\e)(q^{n-1}-\e)/(q+1)$ for $1 \leq i \leq q/2$,  and 

$\bullet$ $\sigma$ of degree $(q^n-\e)(q^{n-1}+\e)/(q-1)$ if $q=4$.\\
If $(n,q) = (7,2)$ and $\chi \in \cX$ has positive $s$-defect, then using \cite{Lu} we show that
$\chi$ must be one of these $q+1$ characters. Likewise, if 
$(n,q) = (5,4)$ and $\chi \in \cX$ has positive $s$-defect and positive $s_\e$-defect, then using \cite{Lu} we check that $\chi$ is again one of these characters. 

Let $V = \F_q^{2n}$ denote the natural module for $G$. Then 
\begin{equation}\label{so2-weil1}
  \rho_0 = 1_G + \varphi +\psi
\end{equation}
is the rank $3$ permutation character of the action of $G$ on singular $1$-spaces of $V$,
see \cite[Table 1]{ST}. It is shown in \cite{GMT} that 
\begin{equation}\label{so2-weil2}  
  \rho_0 = 1_G + \psi + \sigma + \sum^{q/2}_{i=1}\zeta_i
\end{equation}   
is the permutation character of the action of $G$ on non-singular $1$-spaces of $V$ 
(we use the convention that $\sigma = 0$ for $q=2$).
We can identify $G$ with its dual group $G^*$, cf.\ \cite{C}. Then the non-identity 
elements of the natural subgroup $\O^-_2(q)$ of $G$ break into $q/2$ conjugacy classes
with representatives $t_i$, $1 \leq i \leq q/2$, and 
$$\bfC_G(t_i) = \O^{-\e}_{2n-2}(q) \times \O^-_2(q).$$
All these semisimple elements have connected centralizer in the underlying algebraic group.
Hence, these classes yield $q/2$ semisimple characters in $\Irr(G)$, which can then be identified 
with $\zeta_i$, $1 \leq i \leq q/2$. If $q = 4$ then $\zeta_1$ and 
$\zeta_2$ are Galois conjugate and $\Q(\zeta_i) = \Q(\sqrt{5})$.
(Indeed, let $\omega$ denote a primitive $5$th root of unity in $\C$, so that 
$\Q(\omega+\omega^{-1}) = \Q(\sqrt{5})$. Let $\gamma:\omega \mapsto \omega^2$ be a generator of
$\Gal(\Q(\omega)/\Q$. Following the proof of \cite[Lemma 9.1]{NT}, one can show that
$\Q(\zeta_i) \subseteq \Q(\omega)$, and $\gamma$ sends $\zeta_1$ to $\zeta_2$. Moreover, since 
$s_i$ is real, $\Q(\zeta_i)$ is fixed by $\gamma^2:\omega \mapsto \omega^{-1}$. It follows that 
$\Q(\zeta_i) \subseteq \Q(\omega)^{\gamma^2} = \Q(\sqrt{5})$. As $\zeta_1$ and $\zeta_2$ are distinct
Galois conjugates, we conclude that $\Q(\zeta_i) = \Q(\sqrt{5}$.)
In particular, since the $g_j$ are chosen to be $5'$-elements, $\zeta_i(g_j) \in \Q$, so
\begin{equation}\label{so2-weil3}
  \zeta_1(g_i) = \zeta_2(g_i)
\end{equation}
when $q=4$.

\smallskip
(iii) Here we determine character values for the element $g_1$ of order $s$. 

Suppose that $s=\p(q,2n-2)$. Then $\psi$ has $s$-defect $0$, so $\psi(g_1) = 0$. Similarly,
$\sigma(g_1) = 0$ if $\e = +$ and $\zeta_i(g_1) = 0$ if $\e = -$. Next,
$$(\rho_0(g_1),\rho_1(g_1)) = \left\{\begin{array}{ll} (0,q+1), & \e = +,q=4,\\
     (0,0), & \e = +,q=2,\\ (2,q-1), & \e = -.
     \end{array}\right.$$
It follows by (\ref{so2-weil1})--(\ref{so2-weil3}) that 
$$\varphi(g_1) = \pm 1,\mbox{ and }\left\{
    \begin{array}{ll} \zeta_i(g_1) = 2, & \mbox{ if } \e = +, q= 4,\\ 
                             \zeta_i(g_1) = -1, & \mbox{ if } \e = +, q= 2,\\
                             \sigma(g_1) = q-2, & \mbox{ if }\e = -.\end{array} \right.$$

Suppose that either $s=\p(q,2n)$ and $\e = -$, or $s = \p(q,n)$ with $2 \nmid n$ and $\e = +$. Then 
$\varphi$, $\zeta_i$, and $\sigma$ all have $s$-defect $0$, so they all vanish at $g_1$. Also, 
$\rho_0(g_1) = 0$, so (\ref{so2-weil1}) implies that $\psi(g_1) = -1$. 

The remaining case is that $s=\p(q,n-1)$, $\e = +$, and $2|n$. Then $\psi$ and $\zeta_i$ have 
$s$-defect $0$, so they vanish at $g_1$. Also, $\rho_0(g_1) = 2$, so (\ref{so2-weil1}) implies that 
$\varphi(g_1) = 1$. Similarly, $\rho_1(g_1) = q-1$, so (\ref{so2-weil1}) implies that 
$\sigma(g_1) = q-2$. 

\smallskip
(iv) Suppose $n \geq 5$ and $q = 4$. The analysis in (iii) shows that there are at most 
$3$ characters $\chi \in \cX$ that can be nonzero at $g_1 = g_2$, in which case 
$|\chi(g_1)\chi(g_2)| \leq 4$. Also, one character in $\cX$ has degree 
$\geq d:= (q^n-1)(q^{n-1}-q)/(q^2-1)$ and all others have degree $\geq 3d$.
It follows that 
$$\sum_{\chi \in \Irr(G),~1 < \chi(1) < D}\frac{|\chi(g_1)\chi(g_2)\oc(g)|}{\chi(1)}
    \leq \frac{4 \cdot 5 \cdot q^{n-1}}{(q^{n}-1)(q^{n-1}-q)/(q^2-1)} \cdot \left(1+ \frac{2}{3}\right) 
    < 0.497.$$
Suppose $n \geq 7$ and $q = 2$. The analysis in (iii) shows that there are at most 
$2$ characters $\chi \in \cX$ that can be nonzero at $g_1$, in which case 
$|\chi(g_1)| \leq 1$. If $n \geq 8$, then 
$$\sum_{\chi \in \Irr(G),~1 < \chi(1) < D}\frac{|\chi(g_1)\chi(g_2)\oc(g)|}{\chi(1)}
    \leq \frac{2 \cdot 3^{1/2} \cdot q^{n+3}}{(q^{n}-1)(q^{n-1}-q)/(q^2-1)} 
    < 0.658.$$
If $(n,q) = (7,2)$, then the analysis in (iii) shows that the only case where
two characters $\chi \in \cX$ are nonzero at $g_1$ is when $\e = +$,
$s=\p(q,2n-2)$ and $\chi = \varphi$, $\zeta_1$. In this case,
$\varphi$ has $s_\e$-defect $0$, so it vanishes at $g_2$. Furthermore,
$$\rho_0(g_2) = \rho_1(g_2) = 0,$$
so (\ref{so2-weil1}) and (\ref{so2-weil2}) imply that 
$$\psi(g_2) = -1,~~\zeta_1(g_2) = 0,$$
so no character $\chi \in \cX$ can be nonzero at both $g_1$ and $g_2$. In all other cases,
only one $\chi \in \cX$ can be nonzero at $g_1=g_2$ and  
$|\chi(g_1)\chi(g_2)| \leq 1$. It follows that 
$$\sum_{\chi \in \Irr(G),~1 < \chi(1) < D}\frac{|\chi(g_1)\chi(g_2)\oc(g)|}{\chi(1)}
    \leq \frac{3^{1/2} \cdot q^{n+3}}{(q^{n}+1)(q^{n-1}-q)/(q^2-1)} 
    < 0.666.$$
Combining with (\ref{so2-sum1}), we are done in all cases.    
\hal

\begin{prop}\label{so3-even}
Suppose that $G = \O^\e_{2n}(q)$, where $n \geq 5$, $q = 3,5$, and $\e = \pm$. Assume 
that $t \in \cR(G)$, $2 \nmid n$ if $q=5$, and $n \geq 7$ if $q=3$.  Then $\sP_u(N)$ holds for $G$ and for every $N = q^at^b$. 
\end{prop}

\pf
(i) Consider an unbreakable $g \in G$; in particular, 
$$|\bfC_G(g)| \leq B = \left\{ \begin{array}{ll}2^6 \cdot q^{2n+4}, & q=3\\
    6^2 \cdot  q^{2n-2}, & q = 5 \end{array} \right.$$ 
by Lemma \ref{orthogunb}. We also choose 
$$D  :=  \left\{ \begin{array}{ll}q^{4n-10}, & (n,q) \neq (7,3), (5,5),\\
     q^{19}, & (n,q) = (7,3),
     \\ q^3(q^3-1)(q^5-1)(q-1)^2/2, & (n,q) = (5,5).
    \end{array} \right.$$ 
For $(n,q) \neq (5,5)$, we fix regular semisimple $g_1=g_2 \in G$ of order 
$s \in \cR(G) \setminus \{t\}$. 

Suppose now that $(n,q) = (5,5)$. First, we fix a regular semisimple 
$u_1 \in \O^{-\e}_6(5)$ of order $\ell := 7$ if $\e = +$ and $\ell:= 31$ if $\e = -$, 
and a regular semisimple $u_2 \in \O^-_4(5)$ of order $13$, and set
$g_1 = \diag(u_1,u_2)$.
If $t \nmid (q^5-\e)$, we fix a regular semisimple $g_2 \in G$ of order 
$s \in \cR(G) \setminus \{t\}$. 
Note that the central involution $z$ of $\SO^-_8(5)$ does {\it not} belong to
$\O^-_8(5)$. Also, a generator $v_2$ of $\SO^{-\e}_2(5)$ does not 
belong to $\O^{-\e}_2(5)$ and has two distinct eigenvalues $\nu$, $\nu^{-1}$ of 
order $q-\e$. Choosing a regular semisimple $v_1 \in \O^-_8(5)$ of order $s$,
we can now set $g_2 := \diag(zv_1,v_2)$ in the case $t|(q^5-\e)$.  

Our choice of $g_i$ implies that each
$g_i$ is an $N'$-element, and $|\bfC_G(g_i)| \leq (q+1)(q^{n-1}+1)$. It follows that
\begin{equation}\label{so3-sum1}
\sum_{\tiny{\begin{array}{l}\chi \in \Irr(G),\\ \chi(1) \geq D\end{array}}}
    \frac{|\chi(g_1)\chi(g_2)\oc(g)|}{\chi(1)}
    \leq \frac{(q+1)(q^{n-1}+1)\cdot B^{1/2}}{D} 
    < \left\{ \begin{array}{ll}0.14, & (n,q) \neq (7,3) \\ 0.40, & (n,q) = (7,3). \end{array} \right.   
\end{equation}
Also, $g_2$ is always $s$-singular. Furthermore, $g_1$ is $\ell$-singular 
when $(n,q) = (5,5)$.
 
\smallskip    
(ii) Now we estimate character values for the characters in     
$$\cX:= \{ \chi \in \Irr(G) \mid 1 < \chi(1) < D\}.$$
By \cite[Theorem 1.4]{Ng}, when $(n,q) \neq (7,3)$, $(5,5)$, the set $\cX$ consists of $q+4$ characters: 

$\bullet$ $\varphi = D_{1_S}-1_G$ of degree $(q^n-\e)(q^{n-1}+q\e)/(q^2-1)$, 

$\bullet$ $\psi = D_{\St}-1_G$ of degree $(q^{2n}-q^2)/(q^2-1)$,

$\bullet$ $D_{\xi_i}$ of degree $(q^n-\e)(q^{n-1}+\e)/2(q-1)$ for $1 \leq i \leq 2$, 

$\bullet$ $D_{\eta_i}$ of degree $(q^n-\e)(q^{n-1}-\e)/2(q+1)$ for $1 \leq i \leq 2$, 

$\bullet$ $D_{\theta_j}$ of degree $(q^n-\e)(q^{n-1}-\e)/(q+1)$ for $1 \leq j \leq (q-1)/2$, and 

$\bullet$ $D_{\chi_j}$ of degree $(q^n-\e)(q^{n-1}+\e)/(q-1)$ for $1 \leq j \leq (q-3)/2.$\\
The characters $D_\alpha$ of $G$ with $\alpha \in \Irr(S)$ and $S := \Sp_2(q)$ are constructed in \cite[Proposition 5.7]{ore}. If $(n,q) = (7,3)$ and $\chi \in \cX$ has positive $s$-defect, then using \cite{Lu} we show that $\chi$ must be one of these $q+4$ characters. If 
$(n,q) = (5,5)$ and $\chi \in \cX$ has positive $s$-defect and positive $\ell$-defect, 
then using \cite{Lu} we again show that $\chi$ must be one of these characters. 

Let $V = \overline\F_q^{2n}$ denote the natural module for $G$ and let 
$F_0 := \{ \lambda \in \overline\F_q^\times \mid \lambda^{q \pm 1}=1\}$. By Lemma \ref{blocks}, 
$$\dim_{\overline\F_q}\Ker(g - \lambda \cdot 1_V) \leq c$$
for all $\lambda \in F_0$, where $c := 4$ for $q = 3$ and
$c = 2$ for $n= 5$. Hence, arguing as in the proof
of \cite[Proposition 5.11]{ore}, we show that  
\begin{equation}\label{so3-weil1}
  |D_\alpha(g)| \leq q^c\cdot \alpha(1)
\end{equation}
for every $\alpha \in \Irr(S)$. On the other hand, by our choice of $g_i$,
$$\dim_{\overline\F_q}\Ker(g_i - \lambda \cdot 1_V) \leq e_i$$
for all $\lambda \in F_0$ and $i =1,2$, where $e_i := 2$ if $(n,q) \neq (5,5)$,
$e_1 := 0$ and $e_2 \leq 1$ if $(n,q) = (5,5)$.
Arguing as in the proof of \cite[Proposition 5.11]{ore}, we obtain 
\begin{equation}\label{so3-weil2}
  |D_\alpha(g_i)| \leq q^{e_i} \cdot \alpha(1)
\end{equation}
for every $\alpha \in \Irr(S)$. 

\smallskip
(iii) Recall that $D^\circ_\alpha = D_\alpha - k_\alpha \cdot 1_G$ where
$k_\alpha = 1$ if $\alpha = 1_S$ or $\St$ and $k_\alpha = 0$ otherwise cf.\ \cite[Table II]{ore}.
Suppose that $q=3$ and $n \geq 8$. Then $\alpha(1) \leq 3$ for all $\alpha \in \Irr(S)$. 
It now follows from (\ref{so3-weil1}) and (\ref{so3-weil2}) that
$$\sum_{\chi \in \Irr(G),~1 < \chi(1) < D}\frac{|\chi(g_1)\chi(g_2)\oc(g)|}{\chi(1)}
    \leq \frac{7 \cdot (3^3+1)^2 \cdot (3^5+1)}{(3^{n}-1)(3^{n-1}-3)/8} < 0.75.$$
Together with (\ref{so3-sum1}), this implies that $g \in g_1^G\cdot g_2^G$, so we are done in
this case.

Assume now that either $q = 5$ or $(n,q) =(7,3)$; in particular, either 
$s|(q^n-\e)$ or $s|(q^{n-1}+1)$. In the former case, 
all $\chi \in \cX$ but $\psi = D_\St-1_G$ have $s$-defect $0$, so vanish at $g_i$. Also, 
$\St(1) = q$, whence by (\ref{so3-weil1}) and (\ref{so3-weil2}) 
$$\sum_{\chi \in \cX}\frac{|\chi(g_1)\chi(g_2)\oc(g)|}{\chi(1)}
    \leq \frac{(q^{e_1+1}+1)(q^{e_2+1}+1)(q^{c+1}+1)}{(q^{2n}-q^2)/(q^2-1)} < 0.33.$$
In the latter case, the only $\chi \in \cX$ that have positive $s$-defect are 
$\varphi = D_{1_S}-1_G$, and $k = (q+1)/2$ or $(q+3)/2$ characters $D_{\alpha_i}$, $1 \leq i \leq k$
with 
$$\sum^k_{i=1}\alpha_i(1)^2 \leq (q+1)^2(q-2)/2.$$
Moreover, $\varphi(1) \geq d_1 := (q^n-1)(q^{n-1}-q)/(q^2-1)$ and 
$D_{\alpha_i}(1) \geq \alpha_i(1)d$. In this case, using (\ref{so3-weil1}) and (\ref{so3-weil2}),
we obtain 
$$\sum_{\chi \in \cX}\frac{|\chi(g_1)\chi(g_2)\oc(g)|}{\chi(1)}
    \leq \frac{(q^{e_1}+1)(q^{e_2}+1)(q^c+1)}{d} + \sum^k_{i=1} 
    \frac{q^{e_1+e_2}\alpha_i(1)^2 \cdot q^c\alpha_i(1)}{\alpha_i(1)d}< 0.44.$$    
In either case, together with (\ref{so3-sum1}), this implies that 
$g \in g_1^G \cdot g_2^G$, so we are again done.   
\hal

\subsection{Completion of the proof of Theorem \ref{main1} for classical groups}
\begin{prop}\label{clas-main}
Theorem \ref{main1} holds for all finite non-abelian simple symplectic or orthogonal groups. 
\end{prop}

\pf
Let $G = \Cl_n(q)$ be such that $G/\bfZ(G)$ is simple non-abelian and 
$q = p^f$. By Corollary \ref{prime2}(i), we need to prove the surjectivity of the 
word map $(x,y) \mapsto x^Ny^N$  only in the case $N = p^at^b$ with $t \in \cR(G)$. In particular,
$t \neq 2,p$. 

First we consider the case $G = \Sp_{2m}(q)$. By Lemma \ref{sp24}, we may assume that
$m \geq 3$.  We are also done by Corollary \ref{main-real} if $q \equiv 1 \bmod 4$.  For the remaining cases, we take $n_0 = 4$ if $q \geq 7$, $n_0=6$ if $q = 4$, and 
$n_0 = 6$ if $q=2$, and set $n = 2m$. Note that 
condition (i) of Proposition \ref{ind-clas} holds by Lemmas \ref{sp24} and \ref{base-clas}. 
Next, condition (ii) of Proposition \ref{ind-clas} holds by Propositions \ref{sp-odd-large},
\ref{sp3-large}, and \ref{sp-even}. Hence we are done by Proposition \ref{ind-clas}.

Next assume that $G = \Omega^\pm_{2m}(q)$ with $m \geq 3$ and $2|q$.  Then we are done by Proposition \ref{ratio} if $q \geq 8$ and $m \geq 4$. Since 
\begin{equation}\label{iso1}
  \begin{array}{l} 
     \O_3(q) \cong \PSL_2(q),~\Omega^+_4(q) \cong \SL_2(q) \circ \SL_2(q), 
          ~\Omega^-_4(q) \cong \PSL_2(q^2),\\
     \O_5(q) \cong \PSp_4(q),~     
   ~\Omega^+_6(q) \cong \SL_4(q)/Z,~\Omega^-_6(q) \cong \SU_4(q)/Z \end{array}
\end{equation}    
(for all $q$ and for a suitable central $2$-subgroup $Z$), cf.\ \cite[Proposition 2.9.1]{KL}, we are done in 
the case $m = 2,3$ by the results of \S4.  In the remaining cases of $q=2,4$ and $n=2m \geq 8$, 
we take $n_0 = 8$ for $q=4$ and $n_0=12$ for $q=2$. Note that condition (i) of Proposition \ref{ind-clas} holds by Lemmas \ref{sp24} and \ref{base-clas} 
for $8 \leq k \leq n_0$, and by the isomorphisms in (\ref{iso1}) for $k = 4,6$.
Next, condition (ii) of Proposition \ref{ind-clas} holds by Proposition 
\ref{so2-even}. Hence we are done by Proposition \ref{ind-clas}.

Finally, let $G = \Omega^\pm_{n}(q)$ with $n \geq 7$ and $q$ odd.  Then we take $n_0 = 6$
if $q > 3$ and $n_0=12$ if $q = 3$. Note that 
condition (i) of Proposition \ref{ind-clas} holds for $1 < k \leq 6$ by the isomorphisms in 
(\ref{iso1}) and Lemma \ref{sp24}, and for $7 \leq k \leq n_0$ by Lemma \ref{base-clas}. 
Next, condition (ii) of Proposition \ref{ind-clas} holds by Proposition \ref{so-odd} when $2\nmid k$, 
by Proposition \ref{ratio} if $2|k$, $q \geq 5$, and $(k,q) \neq (10,5)$, $(14,5)$, 
and by Proposition \ref{so3-even} if $2|k$, and $q=3$ or $(k,q) = (10,5)$, $(14,5)$. 
Hence we are done by Proposition \ref{ind-clas}.
\hal

\section{Theorem \ref{main1} for exceptional groups}

\begin{lem}\label{sz}
Theorem \ref{main1} holds for the Suzuki groups $\tw2 B_2(q^2)$ with $q^2 \geq 8$ and 
the Ree groups $\tw2 G_2(q^2)$ with $q^2 \geq 27$.
\end{lem}

\pf
Let $S$ be of these groups. Note that $|S|$ is divisible
by at least four different odd primes. Hence we can find a prime  divisor $\ell > 2$ of $|S|$ that is coprime
to both $q^2$ and $N$, and a semisimple $x \in S$ of order $\ell$. By \cite[Theorem 7.1]{GM}, 
$x^S \cdot x^S \supseteq S \setminus \{1\}$, whence the claim follows.
\hal

\begin{lem}\label{exc-base}
Theorem \ref{main1} holds for the following: 
$\tw2 F_4(2)'$; $G_2(q)$ with $q=3,4$; $\tw3 D_4(q)$ with $q = 2,4$; $F_4(2)$; $E_6(2)$; $\tw2 E_6(2)$.
\end{lem}

\pf
The cases $\tw2F_4(2)'$, $G_2(3)$, $G_2(4)$ were checked directly 
using their character tables. 
For the remainder, by Corollary  \ref{prime2}(i), it suffices to prove Theorem \ref{main1} for $N = p^at^b$, where $p$ is the defining characteristic 
and $t \in \cR(G) = \{r,s\}$, which is $\{13\}$, $\{241\}$, $\{13,17\}$, $\{73,17\}$, $\{19,17\}$, respectively. This was done by direct calculations similar to those of Lemma \ref{alt-spor}.
\hal

In what follows, let $\Phi'_{24} := q^4 +q^3\sqrt{2}+q^2+q\sqrt{2}+1$.

\begin{lemma}\label{exc-norm} 
Let $S$ be one of $G_2(q)$, $\tw3 D_4(q)$, $\tw2 F_4(q)$, $F_4(q)$, $E_6^\e(q)$, $E_7(q)$, $E_8(q)$
where $q=p^f$. Define the primes $r,s$ as follows: 
\[
\begin{array}{ccccc}
\hline 
S & r & s  & |\bfN_S(T_r):T_r| &  |\bfN_S(T_s):T_s| \\
\hline
G_2(q) & \p(p,3f) & & 6 & \\
\tw3 D_4(q) & \p(p,12f) & & 4 & \\
\tw2 F_4(q^2) & \p(2,24f)|\Phi'_{24} & & 12 & \\
F_4(q) & \p(p,12f) & \p(p,8f)  & 12 & 8 \\
E_6(q) & \p(p,9f) & \p(p,8f)  & 9 & 8 \\
\tw2 E_6(q) & \p(p,18f) & \p(p,8f)  & 9 & 8 \\
E_7(q) & \p(p,18f) & \p(p,7f) &   18 & 14 \\
E_8(q) & \p(p,24f) & \p(p,20f) & 24 & 20 \\
\hline
\end{array}
\]
For $t \in \{r,s\}$ let $x_t \in \cX_t$, where $\cX_t$ is the set of $t$-singular elements in $S$. 

{\rm (i)} $\bfC_S(x_t) = T_t$, where $T_t$ is a uniquely determined maximal torus of $S$.

{\rm (ii)} $|\bfN_S(T_t):T_t|$ is as in the table.

{\rm (iii)} $|\cX_t| < |S|/|N_S(T_t):T_t|$.

{\rm (iv)} If $S \not\cong G_2(q)$, $\tw3 D_4(q)$, then $|\cX_t| <|S|/8$.
\end{lemma}

\pf We know that $x_t$ lies in some maximal torus $T_t$ of $S$. The orders of maximal tori are given by \cite{car}. Inspection shows that for each $t$ there is a unique possible order $|T_t|$ divisible by $t$, as follows, where in most cases we give also the label of $T_t$ in \cite{car} (and  $d = (3,q-\e)$ and $e = (2,q-1)$):

\[
\begin{array}{lll}
\hline
S & |T_r|, \hbox{ label} & |T_s|, \hbox{ label} \\
\hline
G_2(q) & q^2+q+1 & \\
\tw3 D_4(q) & q^4-q^2+1 & \\
\tw2 F_4(q^2) & \Phi'_{24} & \\
F_4(q) & q^4-q^2+1, \;F_4 & q^4+1,\;B_4 \\
E_6^\e(q) & (q^6+\e q^3+1)/d,\;E_6(a_1)& (q^4+1)(q^2-1)/d,\; D_5 \\
E_7(q) & (q^6-q^3+1)(q+1)/e,\; E_7 & (q^7-1)/e,\;A_6 \\
E_8(q) &  q^8-q^4+1,\;E_8(a_1) & q^8-q^6+q^4-q^2+1,\;E_8(a_2) \\
\hline
\end{array}
\]
Write $S = (\cG^F)'$, where $\cG$ is the corresponding adjoint algebraic group and $F$ a Frobenius endomorphism of $\cG$. By \cite[II,4.4]{SS}, $\bfC_{\cG}(x_t)$ is connected. Then 
$\bfC_{\cG}(x_t) = \cD\cZ$ where $\cD$ is semisimple and $\cZ$ is a torus. If $\cD\ne 1$ then $\cD^F$ contains a subsystem $\SL_2(q)$ or $\SU_3(q)$ subgroup $D$, so $x_t \in \bfC_S(D)$. However $\bfC_S(D)$ does not have order divisible by $t$. Hence $\cD=1$ and $\bfC_{\cG}(x_t)$ is a maximal torus, whence $\bfC_S(x_t) = T_t$, proving (i).

Part (ii) follows from the tables in \cite[pp.\ 312--315]{car}.

By (i), every element of $\cX_t$ lies in a unique conjugate of $T_t$, and the number of these conjugates is 
$|S:\bfN_S(T_t)|$; also, $1 \notin \cX_t$. This gives (iii), and (iv) follows immediately. 
\hal

\begin{prop}\label{exc}
Theorem \ref{main1} holds for the simple exceptional group $S = G/\bfZ(G)$, where $G$ is one of the following groups:

\begin{enumerate}[\rm(i)]
\item $G_2(q)$, $q \geq 5$;
\item $\tw3 D_4(q)$, $q \neq 2,4$; 
\item $\tw2 F_4(q^2)$, $q^2 \geq 8$;
\item $F_4(q)$, $q \geq 5$;
\item $E_6(q)_\SC$ or $\tw2 E_6(q)_\SC$, $q \geq 3$; 
\item $E_7(q)_\SC$ or $E_8(q)$.
\end{enumerate}

\end{prop}

\pf 
By Corollary  \ref{prime2}(i), it suffices to prove Theorem \ref{main1} in the case $N = p^at^b$
with $p|q$ and $t \in \cR(G) = \{r,s\}$. 

%(a) 
First we consider the case $S = G_2(q)$ with $q \geq 5$; in particular, 
$t = \p(p,3f)$ (with $q = p^f$ as usual). Note that  
$|\cX_p|/|S| \leq 2/(q-1)-1/(q-1)^2 < 0.31$ for $q \geq 7$ by \cite[Theorem 3.1]{GL}, and 
$|\cX_p|/|S| \leq 1-0.68 = 0.32$ for $q = 5$ by \cite{Lu2}. Lemma \ref{exc-norm} implies that
$$\frac{|\cX_t|}{|S|}+\frac{|\cX_p|}{|S|} < \frac{1}{6} + 0.32 
    < \frac{1}{2},$$
so we are done by Corollary \ref{prime2}(ii).

We can argue similarly in other cases. In the case $S = \tw3 D_4(q)$ with $q \geq 5$, by 
Lemma \ref{exc-norm} and \cite[Theorem 3.1]{GL}, 
$$\frac{|\cX_t|}{|S|}+\frac{|\cX_p|}{|S|} < \frac{1}{4} + \frac{1}{q-1}  \leq \frac{1}{2},$$
so we are done. Also, note that the odd $q$ case is covered by Corollary \ref{main-real}.

Suppose $S = \tw2 F_4(q^2)$ with $q^2 \geq 8$. By \cite[Theorem 7.3]{GM}, 
$S \setminus \{1\} \subseteq x^S \cdot x^S$ for a regular semisimple 
$x \in S$ of order $\Phi'_{24}$. It remains therefore to consider the case $t|\Phi'_{24}$. By Lemma 
\ref{exc-norm} and \cite[Theorem 3.1]{GL}, 
$$\frac{|\cX_t|}{|S|}+\frac{|\cX_p|}{|S|} < \frac{1}{12} + \frac{2}{q^2-1} - \frac{1}{(q^2-1)^2}  
    < \frac{1}{2}.$$

Next we consider the case $S = F_4(q)$ with $q \geq 5$. Note that  
$|\cX_p|/|S| \leq 2/(q-1)-1/(q-1)^2 < 0.3056$ for $q \geq 7$ by \cite[Theorem 3.1]{GL}, and $|\cX_p|/|S| \leq 1-0.6619 = 0.3381$ for $q = 5$ by \cite{Lu2}. It follows by Lemma \ref{exc-norm} that 
$$\frac{|\cX_t|}{|S|}+\frac{|\cX_p|}{|S|} \leq \frac{1}{8} + 0.3381 
    < \frac{1}{2}.$$
%(Note that we are done in the case of $F_4(4)$ if $|\bfN_S(T)|/|T| \geq 1/11$.)    
    
For cases (v) and (vi), we note that  
$|\cX_p|/|G| \leq 1/(q-1) \leq 1/3$ for $q \geq 4$ by \cite[Theorem 3.1]{GL},
and $|\cX_p|/|G| \leq 1-0.6627 = 0.3373$ for $q = 3$ by \cite{Lu2}. It follows by Lemma \ref{exc-norm}
that 
$$\frac{|\cX_t|}{|G|}+\frac{|\cX_p|}{|G|} \leq \frac{1}{8} + 0.3373
    < \frac{1}{2},$$
so we are done.    

%\smallskip
If $G = E_8(q)$, then $G$ has two maximal tori $T_{1,2}$ of
order $q^8-1$ and $\Phi_{15}$, and $t$ is coprime to both $|T_{1,2}|$. According to 
\cite[Theorem 10.1]{LuM}, $T_i$ contains a regular semisimple element $s_i$ for $i =1,2$, such that 
$\Irr(G)$ contains exactly two irreducible characters $\chi$ with $\chi(s_1)\chi(s_2) \neq 0$, namely $1_G$ and $\St$. Since $|\St(s_i)| = 1$, it follows that 
$$\sum_{\chi \in \Irr(G)}\frac{|\chi(s_1)\chi(s_2)\chi(g)|}{\chi(1)} = 1 + \frac{|\St(g)|}{\St(1)} > 0,$$
so $g \in s_1^G \cdot s_2^G$ for all $1 \neq g \in G$, and we are done.

Finally, let $G = E_7(2)$, so that $t \in \{19, 127\}$. Consider $s_1 \in G$ of order $73$ and 
$s_2 \in G$ of order $43$. Using \cite{Lu}, we check that the only $\chi \in \Irr(G)$ that has
positive $73$-defect and positive $43$-defect are $1_G$ and $\St$. Hence 
$s_1^G \cdot s_2^G = G \setminus \{1\}$ and we are done as above.    
\hal

\begin{lem}\label{f4a}
{\rm (i)} Let $G = F_4(4)$ and let $x$ be a non-semisimple element of $G$ such that 
$|\bfC_G(x)|> 3\cdot 4^{19}$. Then there is a quasisimple classical subgroup $S$
in characteristic $2$ of $G$ such that 
$|\bfZ(S)|$ is a $3$-power and $x \in S$.

{\rm (ii)} Let $G = F_4(3)$ and let $x$ be a non-semisimple element of $G$ such that 
$|\bfC_G(x)|> 3^{19}$. Then there is a quasisimple classical subgroup $S$ 
in characteristic $3$ of $G$ such that 
$|\bfZ(S)|$ is a $2$-power and $x \in S$.
\end{lem}

\pf (i) Suppose first that $x$ is unipotent. Following \cite[Table 22.2.4]{LSei}, 
the bound on $|\bfC_G(x)|$ forces $x$ to be in one of the following unipotent classes:
\[
A_1,\,\tilde A_1,\, (\tilde A_1)_2,\, A_1\tilde A_1,\, A_2\,(2 \hbox{ classes}),\,\tilde A_2\,(2 \hbox{ classes}),\,B_2\,(2 \hbox{ classes}).
\]
In the first two cases $x$ lies in a subgroup $\SL_2(4)$. The third class $ (\tilde A_1)_2$ has representative 
$x = u_{1232}(1)u_{2342}(1)$ (see \cite[Table 16.2 and (18.1)]{LSei}). This is centralized by the long root groups $U_{\pm 0100}$, and these generate $A \cong \SL_2(4)$. Then $x \in \bfC_G(A) = \Sp_6(4)$. 
The class $A_1\tilde A_1$ has a representative in a subgroup $A_1(4)\tilde A_1(4)$, which is contained in a subgroup $\Sp_8(4)$.
Representatives of the four classes with labels $A_2,\tilde A_2$ lie in subgroups $\SL_3(4)$ or 
$\SU_3(4)$. Finally, representatives of the classes with label $B_2$ lie in a subgroup $\Sp_4(4)$.

Now suppose $x$ is non-unipotent, with Jordan decomposition $x = su$, where $s\ne 1$ is semisimple and $u$ unipotent. 
As $x$ is assumed non-semisimple, $u\ne 1$. Then $\bfC_G(s)$ is a subsystem subgroup of order greater than 
$3\cdot 4^{19}$, and the only possibility is that $\bfC_G(s) = C_3 \times \Sp_6(4)$.
But then $u \in \Sp_6(4)$ has centralizer of order greater than $4^{19}$, which is impossible for a nontrivial unipotent element of $\Sp_6(4)$.

\vspace{2mm}
(ii) This is similar to (i). Suppose $x$ is unipotent. Then $x$ lies in one of the classes
\[
A_1,\,\tilde A_1\,(2 \hbox{ classes}),\, A_1\tilde A_1,\, A_2\,(2 \hbox{ classes}),\,\tilde A_2.
\]
For the $A_1\tilde A_1$ class, as above $x$ lies in a subgroup $\Spin_9(3)$. Each of the other class representatives lies in a subgroup $\SL_3(3)$ or $\SU_3(3)$. 

Now suppose $x$ is non-unipotent, so $x = su$ with semisimple and unipotent parts $s,u \ne 1$. Then $\bfC_G(s)$ is a subsystem subgroup of type $B_4$, $A_1C_3$, $T_1C_3$ or $T_1B_3$, where $T_1$ denotes a 1-dimensional torus. The last two cases are not possible, as in (i). In the first case, $x \in \bfC_G(s) = \Spin_9(3)$. So assume finally that $\bfC_G(s)$ is of type $A_1C_3$, and let $u = u_1u_2$ with $u_1 \in \SL_2(3)$, $u_2\in \Sp_6(3)$. If $u_2 \ne 1$ then 
$$|\bfC_G(x)| \le |\SL_2(3)|\cdot |\bfC_{\Sp_6(3)}(u_2)| < 3^{19}.$$ 
Hence $u_2=1$ and $x = su \in \SL_2(3) < \SL_3(3)$. This completes the proof. \hal

\begin{lem}\label{f4b}
Theorem \ref{main1} holds for the simple exceptional groups $G = F_4(q)$ with $q = 3$, $4$.
\end{lem}

\pf 
By Corollary  \ref{prime2}(i), it suffices to prove Theorem \ref{main1} in the case $N = q^at^b$
with $t \in \cR(G) = \{r,s\}$.

Suppose that $t \nmid \Phi_8$. Then $G$ has two maximal tori $T_{1,2}$ of orders 
$(q^2-1)(q^2+q+1)$ and $q^4+1$, which are coprime to $N$. It is shown in \cite[Theorem 10.1]{LuM}
that $T_i$ contains a regular semisimple element $s_i$ for $i =1,2$, such that $\Irr(G)$ contains
exactly two irreducible characters $\chi$ with $\chi(s_1)\chi(s_2) \neq 0$, namely $1_G$ and 
$\St$. It follows that $G \setminus \{1\} = s_1^G \cdot s_2^G$, so we are done as in the proof
of Proposition \ref{exc}.

Now consider the case where $t|\Phi_8$. Choose regular semisimple 
$s_1 = s_2 \in G$ of (prime) order $s=\Phi_{12}$. Using \cite{Lu}, we check that if 
$\chi \in \Irr(G)$, $1 < \chi(1) < q^{18}$, and $\chi(s_1)\chi(s_2) \neq 0$ (in particular,
$\chi$ has positive $s$-defect), then $\chi = \chi_{1,2}$ with
$$\chi_1(1) = q\Phi_1^2\Phi_3^2\Phi_8,~~\chi_2(1) = q\Phi_2^2\Phi_6^2\Phi_8.$$ 
It suffices to show that every nontrivial $g \in G$ belongs to $s_1^G \cdot s_2^G$. This is indeed
the case if $g$ is semisimple by \cite{Gow}, so we assume $g$ is non-semisimple. Moreover, if $|\bfC_G(g)| > B$, where $B := 3^{19}$ for $q=3$ and 
$B := 3 \cdot 4^{19}$ for $q=4$, then by Lemma \ref{f4a} we can embed 
$G$ in a quasisimple classical subgroup $S$ in characteristic $q$ with
$|\bfZ(S)|$ coprime to $N$, in which case we are done by applying Theorem
\ref{main1} to $S/\bfZ(S)$. So we may assume that $|\bfC_G(g)| \leq B$. 
Next observe that $\chi_i$ is rational-valued (as it is the unique character in $\Irr(G)$ of its degree),
and $\chi_i(1) \equiv \pm 1 \bmod s$. It follows that $\chi_i(s_1) \in \Z$ and 
$\chi_i(s_1) \equiv \pm 1 \bmod s$. Since $|\chi_i(s_1)| \leq |\bfC_G(s_1)|^{1/2} = s^{1/2}$, 
we conclude that $\chi_i(s_1) = \pm 1$. It follows that
$$\sum^2_{i=1}\frac{|\chi_i(s_1)\chi_i(s_2)\chi_i(g)|}{\chi_i(1)} \leq 
  \frac{B^{1/2}}{\chi_1(1)} + \frac{B^{1/2}}{\chi_2(1)} < 0.87.$$
On the other hand, since $|\bfC_G(s_i)| = \Phi_{12}$, 
$$\sum_{\chi \in \Irr(G),~\chi(1) \geq q^{18}}
  \frac{|\chi(s_1)\chi(s_2)\chi(g)|}{\chi(1)} \leq 
  \frac{B^{1/2}\Phi_{12}}{q^{18}} < \frac{1}{q^4} \leq \frac{1}{81}.$$
It follows that $g \in s_1^G \cdot s_2^G$, as stated.
\hal

In summary, we have proved the following.

\begin{cor}\label{exc-main}
Theorem \ref{main1} holds for all finite non-abelian simple exceptional groups of Lie type. 
\end{cor}

\noindent
{\bf Proof of Theorem \ref{main1}.}
The case of simple groups of Lie type is completed by Proposition \ref{clas-main} for classical groups
and Corollary \ref{exc-main} for exceptional groups. Alternating and sporadic groups are handled 
by Lemma \ref{alt-spor} and Proposition \ref{alt1}.
\hal

\section{Odd power word maps}
\subsection{Preliminaries}
\begin{lem}\label{schur}
Let $S$ be a finite non-abelian simple group. To prove Theorem \ref{main2}
for all quasisimple groups $G$ with $G/\bfZ(G) \cong S$, it suffices to prove
it for the $2'$-universal cover $H$ of $S$, that is, $H/\bfZ(H) \cong S$ and
$|\bfZ(H)|$ is the $2'$-part of the order of the Schur multiplier of $S$.   
\end{lem}

\pf
It suffices to prove Theorem \ref{main2} for the universal cover $L$ of 
$S$. By assumption, Theorem \ref{main2} holds for $H = L/Z$, where 
$Z \leq \bfZ(L)$ is a $2$-group. Thus every $g \in L$ can be written in the form
$g = xyzt$, where $x,y,z$ are $2$-elements of $L$ and $t \in Z$. It follows
that $g = xy(zt)$ is a product of three $2$-elements in $L$.
\hal

\begin{lem}\label{alt-odd}
Theorem \ref{main2} holds for all quasisimple covers of alternating 
groups $S = \A_n$ with $n \geq 5$. Moreover, every element of 
$S$ is a product of two $2$-elements.
\end{lem}

\pf
The cases $S = \A_6, \A_7$ are checked directly using \cite{Atlas}.
%Hence, by Lemma \ref{lie2} we may assume $n \neq 6,7$. 
By Lemma \ref{schur}, it suffices to prove Theorem \ref{main2} for $G = \A_n$.

\smallskip
(i) First we show that if $g = (1,2,\ldots,m)$ is an $m$-cycle with 
$m = 2k+1 \geq 5$, then $g = x_1y_1 = x_2y_2$, where $x_i,y_i \in \SSS_m$
have order $2$ or $4$, and moreover $x_1,y_1 \in \A_m$, 
$x_2,y_2 \in \SSS_m \smallsetminus \A_m$. Indeed, $g$ is inverted
by the involution
$$x := (1,2k+1)(2,2k) \ldots (k-1,k+3)(k,k+2).$$
Setting $y := xg$, we get $y^2 = xgxg = g^{-1}g = 1$, so $g = xy$. 
Next, we set 
$$x' := (1,2k+1)(2,2k) \ldots (k-1,k+3),~~y' := x' g.$$ 
A computation establishes
that $|x'| = 2$, $|y'| = 4$, and $g = x' y'$.
%$$\begin{array}{ll}(y')^2 = (x'gx')g & 
%   = (2k+1,2k, \ldots,k+3,k,k+1,k+2,k-1,k-2, \ldots,1)(1,2,\ldots ,2k+1)\\
%  & = (k,k+2)(k + 1,k+3), \end{array}$$  
%whence $|y'| = 4$. 
Since exactly one of $x,x'$ belongs to $\A_m$ and $g \in \A_m$, the claims follow.

\smallskip
(ii) Now we show that every $g \in \A_n$ is a product of two $2$-elements. 
Indeed, if $g$ is real in $\A_n$ 
then the statement follows from Lemma \ref{real1}. 
Since $g$ is always real in $\SSS_n$, we may assume that $g$ is not real in 
$\A_n$, so it is not centralized by any {\it odd} permutation in $\SSS_n$. Thus 
$g = g_1g_2 \ldots g_s$ is a product of $s \geq 1$ disjoint cycles,
where $g_i$ is an $n_i$-cycle, $3 \leq n_1 < n_2 < \ldots < n_s$, and
$n_i$ is odd for all $i$. We may assume that 
$$\SSS_n \geq X_1 \times X_2 \times \ldots \times X_s,$$
where $X_i \cong \SSS_{n_i}$ and $g_i \in X_i$. 

Suppose $n_1 \geq 5$. Then, according to (i) we can write $g_i = x_iy_i$
where $x_i,y_i \in [X_i,X_i] \cong \A_{n_i}$ are $2$-elements. Hence
$g = xy$ with $x := x_1x_2 \ldots x_s$ and $y := y_1y_2 \ldots y_s$, as desired.

Assume now that $n_1 = 3$.  Since $g$ is not real in $\A_n$ and $n \geq 5$, 
we observe that $s \geq 2$. Again by (i), for $i \geq 2$ we can write 
$g_i = x_iy_i$, where $x_i,y_i \in X_i$ are $2$-elements; moreover,
$x_i,y_i \in [X_i,X_i]$ if $i \geq 3$ and $x_{2},y_{2} \in X_{2} \smallsetminus [X_{2},X_{2}]$.
We may assume that $g_1 = (1,2,3)$ and write $g_1 = x_1y_1$ with
$x_1 = (1,3)$, $y_1 = (1,2)$. Now setting $x:= x_1x_2 \ldots x_s$ and 
$y := y_1y_2 \ldots y_s$, again $g = xy$ is a product of 
two $2$-elements in $\A_n$.
\hal

\begin{lem}\label{lie2}
Let $S$ be a non-abelian simple group of Lie type in characteristic $2$.
Theorem \ref{main2} holds for all quasisimple covers of $S$. 
\end{lem}

\pf
The case $S = \tw2 F_4(2)'$ is checked directly using \cite{Atlas};
and $S = \A_6$ follows from Lemma \ref{alt-odd}.
Suppose now that $S \not\cong \A_6$, $\tw2 F_4(2)'$.
Then there is a quasisimple Lie-type group $H$ of simply connected
type such that $H$ is a $2'$-universal cover of $S$. According to 
\cite[Corollary, p.\ 3661]{EG}, every non-central element of $H$ is a 
product of two $2$-elements. For $g \in \bfZ(H)$, consider a non-central 
$2$-element $t$ of $H$. Again $gt^{-1} = xy$ for some $2$-elements
$x,y$ of $H$, so $g = xyt$ is a product of three $2$-elements. Hence
we are done by Lemma \ref{schur}.  
\hal

\begin{lem}\label{odd-base}
{\rm (i)} Theorem \ref{main2} holds for the quasisimple group $G$ if $G/\bfZ(G)$ is one of the following 
simple groups: a sporadic group,  
$\PSU_4(3)$, $\PSp_6(3)$, $\O_7(3)$, $\PSp_8(3)$.

{\rm (ii)} Suppose that $G = \GU_n(3)$ with $3 \leq n \leq 6$. Each $g \in G$ can be written
as $g = xyz$, where $x, y, z$ are $2$-elements of $G$ and $\det(x) = \det(y) = 1$.
\end{lem}

\pf
These statements were established using direct calculations similar to 
those of Lemma \ref{alt-spor}.
\hal

\subsection{Regular $2$-elements in classical groups in odd characteristic}
We show that finite classical groups in odd characteristic 
admit regular $2$-elements with prescribed determinant or spinor norm.

We begin with the general linear and unitary groups.

\begin{lem}\label{reg-slu}
Let $G = \GL^\e_n(q)$ with $n \geq 1$, $\e = \pm 1$,  $q$ an odd prime power and let
$\mu_{q-\e} := \{ \lambda \in \overline{\F}_q^\times \mid \lambda^{q-\e} = 1\}$. 
For every 
$2$-element $\delta$ of $\mu_{q-\e}$, there exists a regular $2$-element 
$s = s_n(\delta) $ of $ G$, 
such that $\det(s) = \delta$ and $s$ has at most two eigenvalues
$\beta$ that belong to $\mu_{q-\e}$ (and each such eigenvalue appears with
multiplicity one).
\end{lem}

\pf
(i) First we consider the special case $n = 2^m \geq 2$ and construct a regular $2$-element
$s_m$ of $ G$. Fix $\g \in  \overline{\F}_q^\times$ with $|\g| = (q^{2^m}-1)_2 \geq 8$. Using 
the embeddings
$$\GL_1(q^{2^m}) \hookrightarrow \GL_{2^{m-1}}(q^2) \hookrightarrow \GL^\e_{2^m}(q) = G,$$
we can find $s_m \in G$ which is conjugate over $\overline\F_q$ to 
$$\diag(\g,\g^{q\e},\g^{(q\e)^2}, \ldots,\g^{(q\e)^{n-1}}).$$
It is straightforward to check that all eigenvalues of $s_m$ appear with multiplicity one and 
have order $(q^{2^m}-1)_2$; in particular, $s_m$ is regular. 

\smallskip
(ii) If $n = 1$, then we set $s_1(\delta) = \delta$. Suppose $n = 2$. If $\delta \neq 1$, then we 
choose $s_2(\delta) := \diag(1,\delta)$. If $\delta = 1$, then we can choose $\a = \pm 1$ such that
$q \equiv \a \bmod 4$ and take 
$$s_2(1) \in C_{q-\a} \hookrightarrow \SL^\e_2(q) < G$$
with $|s_2(1)| = 4$.    
Note that $|s_n(\delta)| < (q^2-1)_2$ for all $\delta \in \mu_{q-\e}$ and $n = 1,2$.

Consider the case $n \geq 3$ odd and write 
$$n = 2^{m_1} + 2^{m_2} + \ldots + 2^{m_t} + 1$$
with $m_1 > m_2 > \ldots > m_t \geq 1$. Setting 
$$s := \diag(s_{m_1},s_{m_2}, \ldots,s_{m_t},\a) \in \GL^\e_{2^{m_1}}(q) \times \ldots \times \GL^\e_{2^{m_t}}(q) 
    \times \GL^\e_1(q) < G,$$
with $\a := \delta/\prod^t_{i=1}\det(s_{m_i})$, we deduce that 
$\det(s) = \delta$ and all eigenvalues of $s$ appear with multiplicity one, as required.

\smallskip
(iii) We may now assume that 
$$n = 2^{m_1} + 2^{m_2} + \ldots + 2^{m_t}$$
with $m_1 > m_2 > \ldots > m_t \geq 1$. 

Suppose first that $m_t = 1$. We choose 
$$s := \diag(s_{m_1},s_{m_2}, \ldots,s_{m_{t-1}},s_2(\a)) \in 
    \GL^\e_{2^{m_1}}(q) \times \ldots \times \GL^\e_{2^{m_{t-1}}}(q) \times \GL^\e_2(q) < G,$$
with $\a := \delta/\prod^t_{i=1}\det(s_{m_i})$, so that $\det(s) = \delta$. The construction of
$s$ ensures that all eigenvalues of $s$ appear with multiplicity one, so $s$ is regular.

If $a := m_t \geq 2$, then we rewrite 
$$n = 2^{a_1} + 2^{a_2} + \ldots + 2^{a_{t-1}} + 2^{a_t} + 2^{a_{t+1}} + \ldots + 2^{a_k},$$
where $a_i = m_i$ for $1 \leq i \leq t-1$, $k = t+a-1$, and 
$(a_{t},a_{t+1}, \ldots ,a_k) = (a-1,a-2, \ldots,2,1,1)$. Now we can choose 
$$s := \diag(s_{a_1},s_{a_2}, \ldots,s_{a_{k-1}},s_2(\a)) \in 
   \GL^\e_{2^{a_1}}(q) \times \ldots \times \GL^\e_{2^{a_{k-1}}}(q) \times \GL^\e_2(q) < G,$$
with $\a := \delta/\prod^{k-1}_{i=1}\det(s_{a_i})$. Again $\det(s) = \delta$, and all eigenvalues of $s$ 
appear with multiplicity one, as desired.

The last condition on $s$ can be checked easily in all cases.
\hal

\begin{lem}\label{reg-sp}
Let $G = \Sp_{2n}(q)$ with $n \geq 1$ and $q$ an odd prime power. 
There exists a regular $2$-element 
$s$  of $G$ (and neither $1$ nor $-1$ is an eigenvalue of $s$).
\end{lem}

\pf
First we consider the special case $n = 2^m \geq 2$.  We fix $\g \in  \overline{\F}_q^\times$ with $|\g| = (q^{2^m}-1)_2 \geq 8$
and use the element $s_m$ constructed in part (i) of the proof of Lemma \ref{reg-slu} via  
the embeddings
$$\GL_1(q^{2^m}) \hookrightarrow \GL_{2^{m}}(q) \hookrightarrow \Sp_{2n}(q) = G.$$
Note that $s_m$ is conjugate over $\overline\F_q$ to 
$$\diag(\g,\g^{q},\g^{q^2}, \ldots,\g^{q^{n-1}},\g^{-1},\g^{-q},\g^{-q^2}, \ldots,\g^{-q^{n-1}}).$$
In particular, all eigenvalues of $s_m$ appear with multiplicity one and 
have order $(q^{2^m}-1)_2$, whence $s_m$ is regular. 

Consider the general case
$$n = 2^{m_1} + 2^{m_2} + \ldots + 2^{m_t}$$
with $m_1 > m_2 > \ldots > m_t \geq 0$ and $t \geq 1$. If $m_t \geq 1$, set
$$s := \diag(s_{m_1},s_{m_2}, \ldots,s_{m_t}) \in \Sp_{2^{m_1}}(q) \times \ldots \times Sp_{2^{m_t}}(q) \leq G.$$ 
If $m_t = 0$, then we can choose
$$s:= \diag(s_{m_1},s_{m_2}, \ldots,s_{m_{t-1}},s_2(1)) \in \Sp_{2^{m_1}}(q) \times \ldots \times Sp_{2^{m_{t-1}}}(q) \times \Sp_2(q) < G,$$
where $s_2(1)$ is constructed in part (ii) of the proof of Lemma \ref{reg-slu}. It is easy to check that $s$ has 
the desired properties.
\hal

Recall that the {\it spinor norm} $\theta(g)$ of $g \in \SO^\e_{n}(q)$ is defined in 
\cite[pp.\ 29--30]{KL}.
 
\begin{lem}\label{reg-so}
Let $G = \SO^\e_n(q)$ with $n \geq 2$, $\e = \pm 1$,  $q$ an odd prime power. 
For $\delta = \pm 1$, there exists a regular $2$-element 
$s = s^\e_n(\delta)$ of $G$, such that $\theta(s) = \delta$; moreover, every 
$\beta \in \F_{q^2}^\times$ can appear as an eigenvalue of $s$ with 
multiplicity at most two, and multiplicity two can occur only when $\beta = \pm 1$.
\end{lem}

\pf
(i) First we consider the special case $n = 2^{m+1} \geq 4$.  We fix $\g \in  \overline{\F}_q^\times$ with 
$|\g| = (q^{2^m}-1)_2 \geq 8$
and use the element $s_m$ constructed in part (i) of the proof of Lemma \ref{reg-slu} via  
the embeddings
$$\GL_1(q^{2^m}) \hookrightarrow \GL_{2^{m}}(q) \hookrightarrow \SO^+_{n}(q)$$
Note that $s_m$ is conjugate over $\overline\F_q$ to 
$$\diag(\g,\g^{q},\g^{q^2}, \ldots,\g^{q^{2^m-1}},\g^{-1},\g^{-q},\g^{-q^2}, \ldots,\g^{-q^{2^m-1}}).$$
In particular, all eigenvalues of $s_m$ appear with multiplicity one and 
have order $(q^{2^m}-1)_2$, whence $s_m$ is regular. As an element of 
$\GL_{2^m}(q)$, $s_m$ has determinant
$$\nu : = \g^{1+q+q^2+ \ldots + q^{2^m-1}} = \g^{(q^{2^m}-1)/(q-1)}.$$
It follows that $\nu^{(q-1)/2} = \g^{(q^{2^m}-1)/2} =-1$, so $\theta(s_m) = -1$ by \cite[Lemma 2.7.2]{KL}.

\smallskip
(ii) Suppose that $n = 2$. We take $s^\e_2(-1) \in \SO^\e_2(q) \cong C_{q-\a}$ of order $(q-\e)_2$,
and $s^\e_2(1) = I_2$.
%if $4 \nmid (q-\e)$, and $s^\e_2(1) = -I_2$ if $4|(q-\e)$.

Next suppose that $n = 4$ and choose $\a = \pm 1$ such that $q \equiv \a \bmod 4$. We also fix 
$s_0 \in \SO^\a_2(q)$ of order $(q-\a)_2 \geq 4$ so that $\theta(s_0) = -1$ (note that 
we can take $s_0 = s^\a_2(-1)$).
Since $\SO^\e_4(q) > \SO^\a_2(q) \times \SO^{\e\a}_2(q)$, we can choose
$$s^+_4(1) = \diag(-I_2,I_2),~~s^-_4(1) = \diag(s_0,-I_2),~~s^+_4(-1) = \diag(s_0,-I_2),~~s^-_4(-1) = \diag(s_0,I_2).$$
Note that $|s^\e_n(\delta)| < (q^2-1)_2$ for all $\delta = \pm 1$ and $n = 2,4$. Also, we need later
the fact that $s^\e_4(-\e)$ does not have $1$ as an eigenvalue. 

\smallskip
(iii) Suppose that $6 \leq n \equiv 2 \bmod 4$. We write  
$$n = 2^{m_1+1} + 2^{m_2+1} + \ldots + 2^{m_t+1} + 2$$
with $m_1 > m_2 > \ldots > m_t \geq 1$, and choose 
$$s := \diag(s_{m_1},s_{m_2}, \ldots,s_{m_t},s^\e_2(\a)) \in \SO^+_{2^{m_1+1}}(q) \times \ldots \times 
    \SO^+_{2^{m_t+1}}(q) \times \SO^\e_{2}(q) < G$$
with $\a := (-1)^t\delta$, so that $\theta(s) = \delta$. 

Consider the case $n \equiv 0 \bmod 4$ and write 
$$n = 2^{m_1+1} + 2^{m_2+1} + \ldots + 2^{m_t+1}$$
with $m_1 > m_2 > \ldots > m_t \geq 1$. We can rewrite 
$$n = 2^{a_1+1} + 2^{a_2+1} + \ldots + 2^{a_{t-1}+1} + 2^{a_t+1} + 2^{a_{t+1}+1} + \ldots + 2^{a_k+1},$$
where $a_i = m_i$ for $1 \leq i \leq t-1$, $k = t+m_t-1$, and 
$$(a_{t},a_{t+1}, \ldots ,a_k) = \left\{ \begin{array}{ll}(m_t-1,m_t-2, \ldots,2,1,1), & m_t \geq 2,\\
    (m_t), & m_t = 1. \end{array} \right.$$ 
Now we can choose 
$$s := \diag(s_{a_1},s_{a_2}, \ldots,s_{a_{k-1}},s^\e_4(\a)) \in \SO^+_{2^{a_1+1}}(q) \times \ldots \times 
    \SO^+_{2^{a_{k-1}+1}}(q) \times \SO^\e_{4}(q) < G$$
with $\a := (-1)^{k-1}\delta$, so that $\theta(s) = \delta$. 

\smallskip
(iv) From now on, we may assume 
$$n = 2^{m_1+1} + 2^{m_2+1} + \ldots + 2^{m_t+1} + 1$$
with $m_1 > m_2 > \ldots > m_t \geq 0$ and $t \geq 1$. Again choose $\a = \pm 1$ such that 
$4|(q-\a)$. 

If $m_t = 0$, then we choose 
$$s:= \left\{ \begin{array}{ll} 
    \diag(s_{m_1}, s_{m_2}, \ldots ,s_{m_{t-1}},s_{m_t},1), & \delta = (-1)^t,\\
    \diag(s_{m_1}, s_{m_2}, \ldots ,s_{m_{t-1}},-I_2,1), & \delta = (-1)^{t-1},\end{array} \right.$$
and note that $s \in \SO^+_{2^{m_1+1}}(q) \times \ldots \times \SO^+_{2^{m_{t-1}+1}}(q) \times \SO^\a_2(q) \times \SO_1(q) < G$.    
%Similarly, if $m_t \geq 1$ and $\delta = (-1)^t$, we can also choose 
%$$s =  \diag(s_{m_1}, s_{m_2}, \ldots ,s_{m_{t-1}},s_{m_t},1) \in G.$$

Finally, suppose that $m_t \geq 1$.  We rewrite 
$$n = 2^{a_1+1} + 2^{a_2+1} + \ldots + 2^{a_{t-1}+1} + 2^{a_t+1} + 2^{a_{t+1}+1} + \ldots + 2^{a_k+1}+1,$$
where $a_i = m_i$ for $1 \leq i \leq t-1$, $k = t+m_t-1$, and 
$$(a_{t},a_{t+1}, \ldots ,a_k) = \left\{ \begin{array}{ll}(m_t-1,m_t-2, \ldots,2,1,1), & m_t \geq 2,\\
    (m_t), & m_t = 1. \end{array} \right.$$ 
Next, we set 
$$s:= \diag(s_{a_1}, s_{a_2}, \ldots ,s_{a_{k-1}},s^\beta_4(-\beta),1)$$
which belongs to  
$$\SO^+_{2^{a_1+1}}(q) \times \ldots \times \SO^+_{2^{a_{k-1}+1}}(q) \times \SO^\beta_4(q) \times \SO_1(q) < G,$$
where $\beta = (-1)^{k}\delta$.        

In all cases, one can verify that $\theta(s) = \delta$, and $s$ has the desired properties. 
\hal
 
\subsection{Proof of Theorem \ref{main2} for classical groups in odd characteristics}
First we deal with special linear and unitary groups in dimensions $3$ and $4$. 

The {\sf CHEVIE} project~\cite{CHEVIE} provides \emph{generic character tables}
for the groups $\SL_3(q)$ and $\SU_3(q)$; these are symbolic parametrized
descriptions of the character tables of all of these groups. 
To establish Lemma \ref{main2-slu34}, it suffices to prove that 
\[ c_{x,y,z} =  \frac{|x^G| \cdot |y^G|}{|G|} 
\sum_{\chi \in \Irr(G)} \frac{\chi(x)\chi(y)\chi(z^{-1})}{\chi(1)} > 0, \]
for all $x,y,z \in G$.  While, in principle there is a function which computes 
$c_{x,y,z}$ from the generic tables, 
its application is often difficult because the
result may depend in a complicated way on the parameters for a conjugacy
class.  We thank Frank L\"ubeck
for providing us with the following alternative proof of this result.

%In principle, there is also a function which computes 
%in general the application of this function can be difficult because the
%result may depend in a complicated way on the parameters for a conjugacy
%class.  

%Also, the generic character tables for $\SL_4(q)$ and $\SU_4(q)$ are not
%explicitly known (however, for our purpose, one could use character tables of
%$\GL_4(q)$ and $\GU_4(q)$ instead).

\begin{lemma}\label{main2-slu34}
Let $q$ be a power of an odd prime and let $G$ be one of the groups
$\SL_n(q)$ or $\SU_n(q)$ with $n \in \{3,4\}$. Every element of $G$ is a
product of three $2$-elements in $G$.
\end{lemma}

\pf
We choose conjugacy classes carefully, such that only
very few character values from the generic character tables are needed (and
these are also available for $n=4$).

We first consider $G = \SL_3(q)$ or $G = \SU_3(q)$. Let $c \in
\F_{q^2}^\times$ have order $(q^2-1)_2$. %, the highest $2$-power dividing $q^2-1$. 
Since $q-1$ and $q+1$ are even, $c \notin \F_q$,  $c \neq
c^q$ and $c \neq c^{-q}$. 

Let $x$ be a regular semisimple element with
eigenvalues $\{c,c^q,c^{-q-1}\}$ (in case $\SL$), or $\{c, c^{-q}, c^{q-1}\}$
(in case $\SU$). The centralizer of $x$ in $G$ is a maximal torus of order 
$q^2-1$. Let $y$ be a regular semisimple element of the maximal torus of
order $q^2 \pm q +1$. 

By inspecting the generic character tables for $\SL_3$ and $\SU_3$ in {\sf CHEVIE}, we
notice that there are only two irreducible characters which both have a 
non-zero value on the conjugacy classes of $x$ and $y$ (the trivial
character and the Steinberg character of degree $q^3$). This can be
explained in terms of Deligne-Lusztig theory and Lusztig's Jordan 
decomposition of characters,
see~\cite[13.16]{DM}: The only semisimple element of the dual group of $G$
whose centralizer contains maximal tori of types of the centralizers of $x$ 
and of $y$ is the trivial element. From information about the values of
Deligne-Lusztig characters, it follows that only unipotent characters can be
non-zero on both $x$ and $y$. Which unipotent characters have this property
can be read from the character table of the Weyl group of $G$, isomorphic to
the symmetric group on $3$ points, because up to sign this describes the
values of unipotent characters on regular semisimple elements.

Now let $z \in G$. %The number $c_{x,y,z}$ is easy to evaluate:
Observe that 
\[ c_{x,y,z} = \frac{|x^G|\cdot |y^G|}{|G|} (1 - \frac{\St(z^{-1})}{q^3}). \]
Hence $c_{x,y,z} > 0$ for every non-central element $z$. 
The case $z = x$ shows that $y$ is the product of two $2$-power elements,
so every non-central $z$ is a product of three $2$-power elements.

For some $q$ there are non-trivial $z$ in the center of $G$. To show that 
such $z$ can be written as product of three $2$-power elements, we have a
closer look at the generic character table to establish that $c_{x,x,xz} > 0$.
We can compute readily a sufficient lower bound for this number: 
for example, in $\SL_3(q)$ there are $q-2$ irreducible
characters of degree $q^2+q+1$ whose value on $x$ and $xz$ are some root
of unity; for a lower bound we can substitute the corresponding terms in
$c_{x,x,xz}$ by $-(q-2)/(q^2+q+1)$.

Now we turn to the case $G = \SL_4(q)$ and $G = \SU_4(q)$. In this case the
center of $G$ has order $2$ or $4$, so there is nothing to show for center
elements. All groups of type $\SL_n(q)$ contain pairs of regular semisimple
elements such that only two characters are non-zero on both elements. But
for $n=4$ there are no such pairs containing $2$-power elements, therefore 
we need a slightly more complicated argument than before. 

Let $c \in \F_{q^2}^\times$ have order $(q^2-1)_2$. Now $c \neq
c^{-1}$ and $c \neq c^{\pm q}$, and $G$ contains a regular $2$-power element
$x$ with eigenvalues $\{c,c^q,c^{-1},c^{-q}\}$; its centralizer in $G$ is a
maximal torus of order $(q^2-1)(q \pm 1)$.
We choose as $y$ a regular element of a cyclic maximal torus of
order $q^3 \pm 1$. With the same arguments as sketched in the
$\SL_3/\SU_3$-case, we find that only unipotent characters can have non-zero
value on both $x$ and $y$. 
%We do not know the generic character tables for $\SL_4/\SU_4$, but 
The unipotent characters of $G$ are obtained
by restricting the unipotent characters of $\GL_4(q)$ or $\GU_4(q)$, 
respectively. These
are available in {\sf CHEVIE}, and their values are all given by evaluating
polynomials over the integers at $q$. 
%(no complicated dependency on parameters for conjugacy classes). 

There are three unipotent
characters with non-zero value on $x$ and $y$, and we can compute the
precise values of $c_{x,x,y}$ and $c_{x,y,z}$ for every non-central $z \in G$.
For all resulting  polynomials, it is easy to see that they evaluate to a
positive number for all prime powers $q$. This shows that $y$ is a product
of two $2$-power elements, so every non-central element is a product of
three $2$-power elements.
\hal

\begin{prop}\label{main2-sl}
Theorem \ref{main2} holds for all quasisimple covers of $S = \PSL_n(q)$, if $n \geq 5$ and $2 \nmid q$. 
\end{prop}

\pf
(i) By Lemma \ref{schur}, it suffices to prove Theorem \ref{main2} for $G = \SL_n(q)$. Let $s = s_n(1) \in G$ be as constructed
in Lemma \ref{reg-slu}. It suffices to show that every $g \in G$ is a product of three conjugates of $s$, which is equivalent to
\begin{equation}\label{3class1}
  \sum_{\chi \in \Irr(G)}\frac{\chi(s)^3\bar{\chi}(g)}{\chi(1)^2} \neq 0.
\end{equation}
As $|\chi(g)/\chi(1)| \leq 1$, it suffices to prove 
\begin{equation}\label{3class2}
 \sum_{1_G \neq \chi \in \Irr(G)}\frac{|\chi(s)|^3}{\chi(1)} < 1.
\end{equation}
Set 
$$D := \left\{ \begin{array}{ll}(q^n-1)(q^{n-1}-q^2)/(q-1)(q^2-1), & (n,q) \neq (6,3),\\
            (q^5-1)(q^3-1), & (n,q) = (6,3). \end{array} \right.$$
By \cite[Theorem 3.1]{TZ1}, every character $\chi \in \Irr(G)$ of degree less than $D$ is either $1_G$ or 
one of $q-1$ irreducible Weil characters $\tau_{i}$, $0 \leq i \leq q-2$. 

\smallskip
(ii) Consider the case $n \geq 6$. The construction of 
$s$ in Lemma \ref{reg-slu} shows that 
$$|\bfC_G(s)| \leq \left\{ \begin{array}{ll}(q^n-1)/(q-1), & (n,q) \neq (6,3),\\
            (q^4-1)(q^2-1)/(q-1), & (n,q) = (6,3). \end{array} \right.$$      
Hence 
$$\sum_{\chi \in \Irr(G),~\chi(1) \geq D}\frac{|\chi(s)|^3}{\chi(1)} <  
    \frac{|\bfC_G(s)|^{1/2}}{D}\cdot\sum_{\chi \in \Irr(G)}|\chi(s)|^2 =  \frac{|\bfC_G(s)|^{3/2}}{D}   
    < 0.9099.$$            
Next we estimate $|\tau_i(s)|$. Recall that $1_G+\tau_0$ is just the permutation character of $G$ acting on 
the set of $1$-spaces of $\F_q^n$. In the notation of the proof of Proposition \ref{sl-large}, 
by Lemma \ref{reg-slu}, $e(g,\delta^l) \leq 1$ for all $0 \leq l \leq q-2$ and
equality can be attained at most twice. It follows that $s$ fixes at most two $1$-spaces, i.e.\ 
$0 \leq \tau_0(s)  +1 \leq 2$, so $|\tau_0(s)| \leq 1$. 
Arguing as in part (i) of the proof of Proposition \ref{sl-large}, for $1 \leq i \leq q-2$ we obtain
$$|\tau_i(s)| \leq \frac{q+q + 1 \cdot (q-3)}{q-1} = 3.$$        
Hence
$$\sum_{\chi \in \Irr(G),~1 < \chi(1) < D}\frac{|\chi(s)|^3}{\chi(1)} = \sum^{q-2}_{i=0}\frac{|\tau_i(s)|^3}{\tau_i(1)}
    \leq \frac{1+(q-2) \cdot 3^3}{(q^n-q)/(q-1)} < 0.0772.$$
Thus
$$\sum_{1_G \neq \chi \in \Irr(G)}\frac{|\chi(s)|^3}{\chi(1)} < 0.9099 + 0.0772 =  0.9871,$$   
so we are done by (\ref{3class2}).     

\smallskip
(iii) Assume now that $n = 5$. The construction of  $s$ in Lemma \ref{reg-slu} implies that
$|\bfC_G(s)| \leq q^4-1$; furthermore,  $e(g,\delta^l) \leq 1$ for all $0 \leq l \leq q-2$ and
equality can be attained at most once.  It follows by (\ref{weil-sl1}) that 
$|\tau_i(s)| \leq 1$ for all $i$. Arguing as in (ii), we obtain 
$$\sum_{\chi \in \Irr(G),~\chi(1) \geq D}\frac{|\chi(s)|^3}{\chi(1)} <  \frac{(q^4-1)^{1.5}}{q^2(q^5-1)/(q-1)},$$
$$\sum_{\chi \in \Irr(G), 1 < \chi(1) < D}\frac{|\chi(s)|^3}{\chi(1)} = \sum^{q-2}_{i=0}\frac{|\tau_i(s)|^3}{\tau_i(1)} 
    \leq \frac{q-1}{(q^5-q)/(q-1)}.$$
Thus    
$$\sum_{1_G \neq \chi \in \Irr(G)}\frac{|\chi(s)|^3}{\chi(1)} <  \frac{(q^4-1)^{1.5}}{q^2(q^5-1)/(q-1)} + \frac{q-1}{(q^5-q)/(q-1)}
    < \frac{q^4+q-1}{(q^5-q)/(q-1)} < 1,$$
so we are done again.  
\hal 

\begin{prop}\label{main2-su-large}
Theorem \ref{main2} holds for all quasisimple covers of $S = \PSU_n(q)$, 
if $n \geq 5$ and $q \geq 5$ is odd, or 
if $(n,q) = (5,3)$. 
\end{prop}

\pf
(i) By Lemma \ref{schur}, it suffices to prove Theorem \ref{main2} for $G = \SU_n(q)$. Let $s = s_n(1) \in G$ be as constructed
in Lemma \ref{reg-slu}. It suffices to show that every $g \in G$ is a product of three conjugates of $s$. Hence, it suffices
to prove (\ref{3class2}). Set 
$$D := \frac{(q^n-1)(q^{n-1}-q^2)}{(q-1)(q^2-1)}.$$
By \cite[Theorem 4.1]{TZ1}, every character $\chi \in \Irr(G)$ of degree less than $D$ is either $1_G$ or 
one of $q+1$ irreducible Weil characters $\zeta_{i}$, $0 \leq i \leq q$. 

Consider the case $n \geq 6$. The construction of $s$ in Lemma \ref{reg-slu} shows that 
$$|\bfC_G(s)| \leq \left\{ \begin{array}{ll}(q+1)^{n-1}, & n \geq 8,\\
            (q^4-1)(q+1)^2, & n = 7,\\
            (q^4-1)(q+1), & n=6. \end{array} \right.$$      
Hence, as in the proof of Lemma \ref{main2-sl}, 
$$\sum_{\chi \in \Irr(G),~\chi(1) \geq D}\frac{|\chi(s)|^3}{\chi(1)} < \frac{|\bfC_G(s)|^{3/2}}{D}   
    < 0.6992.$$            
Next we estimate $|\zeta_i(s)|$. In the notation of the proof of 
Proposition \ref{su-large}, by 
Lemma \ref{reg-slu},  $e(g,\xi^l) \leq 1$ for all $0 \leq l \leq q$ and
equality can be attained at most twice. 
Arguing as in part (i) of the proof of Proposition \ref{su-large}, we obtain
$$|\zeta_i(s)| \leq \frac{q+q + 1 \cdot (q-1)}{q+1} = \frac{3q-1}{q+1}.$$        
Hence
$$ \sum_{\chi \in \Irr(G),~1 < \chi(1) < D}\frac{|\chi(s)|^3}{\chi(1)} = \sum^{q}_{i=0}\frac{|\zeta_i(s)|^3}{\zeta_i(1)}
    \leq \frac{(q+1)((3q-1)/(q+1))^3}{(q^n-q)/(q+1)} < 0.1467.$$
Thus
$$\sum_{1_G \neq \chi \in \Irr(G)}\frac{|\chi(s)|^3}{\chi(1)} < 0.6992 + 0.1467 =  0.8459,$$   
so we are done by (\ref{3class2}).     

\smallskip
(ii) Assume now that $n = 5$. The construction of  $s$ in Lemma \ref{reg-slu} implies that
$|\bfC_G(s)| \leq q^4-1$; furthermore,  $e(g,\xi^l) \leq 1$ for all $0 \leq l \leq q$ and
equality can be attained at most once.  It follows by (\ref{weil-su1}) that 
$|\zeta_i(s)| \leq 1$ for all $i$. Set 
$$D = (q-1)(q^2+1)(q^5+1)/(q-1).$$
Using \cite{Lu}, we check that if $\chi \in \Irr(G)$ satisfies $1 < \chi(1) < D$ then 
$\chi$ is either one of $q+1$ Weil characters $\zeta_i$, $0 \leq i \leq q$, or one of
$q+1$ characters $\a_i$, $0 \leq i \leq q$, where
$$\a_0(1) = q^2(q^5+1)/(q+1),~~\a_i(1) = (q^2+1)(q^5+1)/(q+1),~1 \leq i \leq q.$$
Inspecting the character table of $\GU_5(q)$ as given in \cite{Noz2}, we observe that
each $\a_i$ extends to
$\GU_5(q)$ and $|\a_i(s)| \leq 1$. Hence, 
$$\sum_{\chi \in \Irr(G), 1 < \chi(1) < D}\frac{|\chi(s)|^3}{\chi(1)} = 
    \sum^{q}_{i=0}\frac{|\zeta_i(s)|^3}{\zeta_i(1)} + \sum^{q}_{i=0}\frac{|\zeta_i(s)|^3}{\zeta_i(1)} 
    \leq \frac{2(q+1)}{(q^5-q)/(q+1)} < 0.1334.$$
On the other hand,
$$\sum_{\chi \in \Irr(G),~\chi(1) \geq D}\frac{|\chi(s)|^3}{\chi(1)} <  \frac{(q^4-1)^{1.5}}{(q-1)(q^2+1)(q^5+1)/(q-1)} < 0.5866.$$
Thus    
$$\sum_{1_G \neq \chi \in \Irr(G)}\frac{|\chi(s)|^3}{\chi(1)} < 0.1334 +0.5866 = 0.72,$$
so we are done again.  
\hal

For $\PSU_n(3)$, respectively $\PSp_{2n}(q)$, we again employ the notion of breakable elements 
as defined in Definition \ref{br-slu}(iii), respectively Definition \ref{unbr-spo}.

\begin{prop}\label{main2-su3}
Theorem \ref{main2} holds for all quasisimple covers of $S = \PSU_n(3)$ if $n \geq 5$.
\end{prop}

\pf
By Lemma \ref{schur}, it suffices to prove Theorem \ref{main2} for $L := \SU_n(3)$.  Consider 
the following statements for $G := \GU_n(3)$:
$$\sQ(n):\begin{array}{l}\mbox{Every }g \in G\mbox{ can be written as }g = xyz,\\
    \mbox{where }x,y,z \in G\mbox{ are 2-elements and }\det(x) = \det(y) = 1,\end{array}$$
$$\sQ_u(n):\begin{array}{l}\mbox{Every unbreakable }g \in G\mbox{ can be written as }g = xyz,\\
    \mbox{where }x,y,z \in G\mbox{ are 2-elements and }\det(x) = \det(y) = 1,\end{array}$$    
By Lemma \ref{odd-base}(ii),   $\sQ(n)$ holds for $3 \leq n \leq 6$.   
It is straightforward to check that Theorem \ref{main2} holds for $L$ with $n \geq 7$ once we 
show that $\sQ_u(n)$ holds. 

We now prove $\sQ_u(n)$ for $n \geq 7$. Consider an unbreakable 
$g \in G$. Lemma \ref{LUL3} implies that $|\bfC_G(g)| \leq 3^{n+2} \cdot 2^4$. 
Let $s_1 =s_2 := s_n(1)$ and $s_3 := s_n(\det(g))$, where $s_n(\delta)$ is constructed in
Lemma \ref{reg-slu}; in particular, $|\bfC_G(s_i)| \leq 4^n$. Choosing
$$D := (3^n-1)(3^{n-1}-9)/32,$$
by the Cauchy-Schwarz inequality, 
$$\sum_{\chi \in \Irr(G),~\chi(1) \geq D}\frac{|\chi(s_1)\chi(s_2)\chi(s_3)\bar\chi(g)|}{\chi(1)^2}
    < \frac{(4^n)^{3/2}(3^{n+2} \cdot 2^4)^{1/2}}{((3^n-1)(3^{n-1}-9)/32)^2} \leq 0.4866.$$
By \cite[Proposition 6.6]{ore}, the characters $\chi \in \Irr(G)$ of degree less than $D$ consist of
$4$ linear characters and $4^2$ Weil characters $\zeta_{i,j}$, $0 \leq i,j \leq 3$. 
Arguing as in part (i) of the proof of Proposition \ref{su-large}, we obtain
$$|\zeta_{i,j}(s_k)| \leq \frac{q+q + 1 \cdot (q-1)}{q+1} = \frac{3q-1}{q+1} = 2$$
for $q= 3$. Together with (\ref{weil-su31}), this implies that
$$\sum_{{\tiny \begin{array}{l}\chi \in \Irr(G),\\1 < \chi(1) < D\end{array}}}
    \frac{|\prod^3_{k=1}\chi(s_k)\cdot\bar\chi(g)|}{\chi(1)^2} = 
    \sum_{{\tiny \begin{array}{l}\chi=\zeta_{i,j}\\0 \leq i,j \leq 3\end{array}}}
    \frac{|\prod^3_{k=1}\chi(s_k)\cdot\bar\chi(g)|}{\chi(1)^2}
    \leq \frac{4^2 \cdot 2^3 \cdot 3^{n-4}}{((3^n-3)/4)^2} < 0.0117.$$  
Since   
$$\sum_{\chi \in \Irr(G),~\chi(1) =1}\frac{\chi(s_1)\chi(s_2)\chi(s_3)\bar\chi(g)}{\chi(1)^2}  = 4,$$
we conclude that 
$$\sum_{\chi \in \Irr(G)}\frac{\chi(s_1)\chi(s_2)\chi(s_3)\bar\chi(g)}{\chi(1)^2} \neq 0,$$
i.e.\ $g \in (s_1)^G\cdot(s_2)^G\cdot(s_3)^G$, as stated.
\hal
 
\begin{prop}\label{main2-sp}
Theorem \ref{main2} holds for all quasisimple covers of $S = \PSp_{2n}(q)$ if $n \geq 1$,
$2 \nmid q$, and $(n,q) \neq (1,3)$.
\end{prop}

\pf
(i) Consider the case $n = 1$. The cases $\PSp_2(5) \cong \Sp_2(4)$,
$\PSp_2(7) \cong \SL_3(2)$,  and $\PSp_2(9) \cong \A_6$ are covered by Lemma \ref{lie2}, so we 
may assume $q \geq 11$.  By Lemma \ref{schur}, it suffices to prove Theorem \ref{main2} for $L := \Sp_2(q)$.
Using the character table of $L$ as given in \cite{DM}, it is straightforward to check that 
$g \in s^L \cdot s^L \cdot s^L$ for all $g \in L$ if $|s| = 4$.   

From now on we may assume $n \geq 2$. Hence by Lemma \ref{schur}, it suffices to prove Theorem \ref{main2} 
for $L := \Sp_{2n}(q)$. If $q \equiv 1 \bmod 4$, then $L$ is real by \cite[Theorem 1.2]{TZ2}, whence we are done
by Lemma \ref{real1}. Also, the case $\PSp_4(3) \cong \SU_4(2)$ is covered by Lemma \ref{lie2}. 
Note that Theorem \ref{main2} holds for $\Sp_6(3)$ and $\Sp_8(3)$ by Lemma \ref{odd-base}(i). 
So we may assume $q \equiv 3 \bmod 4$ and $(n,q) \neq (2,3)$, $(3,3)$, $(4,3)$. 

\smallskip
(ii) It suffices to prove that every unbreakable $g \in L$ is a product of three $2$-elements of $L$.
By Lemma \ref{spunb}, 
$$|\bfC_L(g)| \leq B := \left\{ \begin{array}{ll}2q^n, & 2|n,~q \geq 5,\\
    48 \cdot 3^{2n+1}, & 2|n,~q = 3,\\
    q^{2n-1}(q^2-1), & 2 \nmid q. \end{array} \right.$$ 
Let $s$ be as constructed in Lemma \ref{reg-sp}; in particular, 
$$|\bfC_L(s)| \leq C := \left\{ \begin{array}{ll}q^2-1, & n=2,\\
    (q^2-1)(q+1), & n = 3,\\
    (q+1)^n, & n \geq 4. \end{array} \right.$$ 
Choosing
$$D := (q^n-1)(q^n-q)/2(q+1),$$
by the Cauchy-Schwarz inequality, 
$$\sum_{\chi \in \Irr(L),~\chi(1) \geq D}\frac{|\chi(s)^3\cdot\bar\chi(g)|}{\chi(1)^2}
    < \frac{C^{3/2}\cdot B^{1/2}}{D^2} \leq 0.5255.$$
By \cite[Theorem 5.2]{TZ1}, the characters $\chi \in \Irr(L)$ of degree less than $D$ consist of $1_L$ 
and four Weil characters: $\eta_{1,2}$ of degree $(q^n - 1)/2$ and $\xi_{1,2}$ of degree $(q^n-1)/2$. 
Recall by Lemma \ref{reg-sp} that neither $1$ nor $-1$ is an eigenvalue of $s$. Hence,
(\ref{sp-weil1}) holds for $s$.  Since $|\chi(g)| \leq \chi(1)$, 
$$\sum_{\chi \in \Irr(L),~1 < \chi(1) < D}
    \frac{|\chi(s)^3\cdot\bar\chi(g)|}{\chi(1)^2} \leq 
    \sum_{\chi \in \Irr(L),~1 < \chi(1) < D}\frac{|\chi(s)^3|}{\chi(1)} \leq \frac{4}{(q^n-1)/2} < 0.1668.$$  
Thus
$$\sum_{1_L \neq \chi \in \Irr(L)}\frac{|\chi(s)^3\cdot\bar\chi(g)|}{\chi(1)^2} < 0.5255 + 0.1668 = 0.6923,$$
so $g \in s^L\cdot s^L \cdot s^L$, as stated.
\hal
 
\begin{prop}\label{main2-so}
Theorem \ref{main2} holds for all quasisimple covers of $S = {\rm P}\Omega^\e_{m}(q)$ if $m \geq 7$ and
$2 \nmid q$.
\end{prop}

\pf
By Lemma \ref{real1} and \cite[Theorem 1.2]{TZ2}, we may assume that $m \neq 8,9$ and 
$q \equiv 3 \bmod 4$ if $m = 7$. 
Note that Theorem \ref{main2} holds for $\Omega_7(3)$ by Lemma \ref{odd-base}.
By Lemma \ref{schur}, it suffices to prove Theorem \ref{main2} for $L := \Omega^\e_m(q)$.
Let $s = s^\e_m(1) \in L$ be as constructed in Lemma \ref{reg-so}.

\smallskip
(i) First we consider the case $m = 2n$; in particular, $n \geq 5$. The construction of $s$ implies that
$$|\bfC_L(s)| \leq C := \left\{ \begin{array}{ll}(q^4-1)(q+1), & n=5,\\    
   (q+1)^n, & n \geq 6. \end{array} \right.$$ 
Choosing
$$D :=  \left\{ \begin{array}{ll}q^{4n-10}, & (n,\e) \neq (5,-),\\   
    (q-1)(q^2+1)(q^3-1)(q^4+1), & (n,\e) = (5,-), \end{array} \right.$$
by the Cauchy-Schwarz inequality, 
$$\sum_{\chi \in \Irr(L),~\chi(1) \geq D}\frac{|\chi(s)^3|}{\chi(1)}
    < \frac{C^{3/2}}{D} \leq 0.135.$$
By \cite[Propositions 5.3, 5.7]{ore}, the characters $\chi \in \Irr(L)$ of degree less than $D$ consist of $1_L$ 
and $q+4$ characters $\DC_\a$, $\a \in \Irr(X)$, where $X := \Sp_2(q)$. 
By Lemma \ref{reg-so}, each $\beta \in \F_{q^2}^\times$ can 
appear as an eigenvalue of $s$ of multiplicity at most $2$.  Arguing as in the proof of \cite[Proposition 5.11]{ore}, we obtain
that $|D_\a(s)| \leq q^2\a(1)$. Recalling that $\DC_\a$ equals $D_\a$ if $\a \neq 1_X$, $\St_X$ and 
$D_\a - 1_L$ otherwise, cf.\ \cite[Table II]{ore}, for $n \geq 6$ 
$$\sum_{{\tiny \begin{array}{l}\chi \in \Irr(L),\\1 < \chi(1) < D \end{array}}}\frac{|\chi(s)^3|}{\chi(1)} \leq 
    \sum_{\a = 1_X,\St_X}\frac{(q^2\a(1)+1)^3}{\DC_\a(1)} +
    \sum_{{\tiny \begin{array}{l}\a \in \Irr(X),\\ \a \neq 1_X,\St_X\end{array}}}\frac{(q^2\a(1))^3}{\DC_\a(1)} < 0.849.$$  
    
Consider the case $n = 5$. If $4|(q-\e)$, then each $\beta \in \F_{q^2}^\times$ can 
appear as an eigenvalue of $s$ of multiplicity at most $1$, so arguing as above 
$|D_\a(s)| \leq q\a(1)$.  Suppose that $4\nmid(q-\e)$. In this case, the only eigenvalue $\beta$ of $s$
that belongs to $\F_{q^2}^\times$ is $-1$ and its multiplicity is $2$. In the notation of the 
proof of \cite[Proposition 5.11]{ore}, for every $x \in X$
$$|\o(xs)| \leq q^{\dim \Ker(xs-I_{2m})/2} = q^{\dim \Ker(x+I_2)}.$$
When $x $ runs over $X$, $\dim \Ker(x+I_2)$ is $2$ only for $x = -I_2$, 
it is $1$ for $q^2-1$ elements, and it is $0$ for the rest. Hence,
$$|D_\a(s)| \leq \frac{1}{|X|}\sum_{x \in X}|\o(xs)\overline{\a(x)}| 
    \leq \frac{\a(1)}{|X|}(q^2 + q \cdot (q^2-1) + 1 \cdot(q(q^2-1)-q^2)) = 2\a(1).$$    
We have shown that $|D_\a(s)| \leq q\a(1)$. Hence,     
$$\sum_{{\tiny \begin{array}{l}\chi \in \Irr(L),\\1 < \chi(1) < D \end{array}}}\frac{|\chi(s)^3|}{\chi(1)} \leq 
    \sum_{\a = 1_X,\St_X}\frac{(q\a(1)+1)^3}{\DC_\a(1)} +
    \sum_{{\tiny \begin{array}{l}\a \in \Irr(X),\\ \a \neq 1_X,\St_X\end{array}}}\frac{(q\a(1))^3}{\DC_\a(1)} < 0.329.$$      
Thus in all cases
$$\sum_{1_L \neq \chi \in \Irr(L)}\frac{|\chi(s)^3|}{\chi(1)} < 1,$$
so $g \in s^L\cdot s^L \cdot s^L$ by (\ref{3class2}), as stated.

\smallskip
(ii) Now we consider the case $m = 2n+1 \geq 11$. Again 
$|\bfC_L(s)| \leq (q+1)^n$.  Set $D :=  q^{4n-8}$.
By the Cauchy-Schwarz inequality 
$$\sum_{\chi \in \Irr(L),~\chi(1) \geq D}\frac{|\chi(s)^3|}{\chi(1)}
    < \frac{C^{3/2}}{D} \leq 0.062.$$
By \cite[Corollary 5.8]{ore}, the characters $\chi \in \Irr(L)$ of degree less than $D$ consist of $1_L$ 
and $q+4$ characters $\DC_\a$, $\a \in \Irr(X)$. By Lemma \ref{reg-so}, each $\beta \in \F_{q^2}^\times$ can 
appear as an eigenvalue of $s$ of multiplicity at most $e$, where we can choose $e = 2$ for $n \geq 6$ and
$e=1$ for $n = 5$.  Arguing as in the proof of \cite[Proposition 5.11]{ore}, we obtain
that $|D_\a(s)| \leq q^e\a(1)$. Recalling that $\DC_\a$ equals $D_\a$ if $\a \neq \xi_{1,2}$
(the two Weil characters of degree $(q+1)/2$ of $X$) and 
$D_\a - 1_L$ otherwise, cf.\ \cite[Table I]{ore}, 
$$\sum_{\chi \in \Irr(L),~1 < \chi(1) < D}\frac{|\chi(s)^3|}{\chi(1)} \leq 
    \sum_{\a = \xi_{1,2}}\frac{(q^2\a(1)+1)^3}{\DC_\a(1)} +
    \sum_{\a \in \Irr(X),~ \a \neq \xi_{1,2}}\frac{(q^2\a(1))^3}{\DC_\a(1)} < 0.281,$$
so we are done by (\ref{3class2}).      

\smallskip
(iii) Finally, we consider the case $m=7$, so $q \geq 7$. Theorem 
\ref{main2} holds for 
$$\O_3(q) \cong \PSL_2(q),~\O^+_4(q) \cong \SL_2(q) \circ \SL_2(q),~\O^-_4(q) \cong \PSL_2(q^2), ~\O_5(q) \cong \PSp_4(q)$$ 
by Proposition \ref{main2-sp}, and for
$\Spin^+_6(q) \cong \SL_4(q),~\Spin^-_6(q) \cong \SU_4(q)$
by Lemma \ref{main2-slu34}.  Hence, if $g \in L = \O_7(q)$ is breakable in the sense of Definition \ref{unbr-spo}, then $g$ is a 
product of three $2$-elements of $L$. If $g \in L$ is unbreakable then 
$|\bfC_L(g)| \leq q^4(q+1)^2$ by Lemma \ref{orthogunb}.  Also, $\chi(1) \geq q^4+q^2+1$ for 
all $1_L \neq \chi \in \Irr(L)$ by \cite[Theorem 1.1]{TZ1}. As $|\bfC_L(s)| \leq (q+1)^3$, by the 
Cauchy-Schwarz inequality, 
$$\sum_{1_L \neq \chi \in \Irr(L)}\frac{|\chi(s)^3\cdot\bar\chi(g)|}{\chi(1)^2} \leq 
    \frac{(q+1)^{4.5} \cdot q^2(q+1)}{(q^4+q^2+1)^2} < 0.757,$$
so we are done as well.    
\hal
 
\subsection{Proof of Theorem \ref{main2} for exceptional groups in odd characteristics}
Our goal is to prove the following result, which, together with the results of \S\S7.1 and 7.3, 
completes the proof of Theorem \ref{main2}.

\begin{thm}\label{2eltexcep}
Let $G$ be a quasisimple group such that $G/\bfZ(G)$ is an exceptional simple group of Lie type in odd characteristic. Every element of $G$ is a product of three $2$-elements.
\end{thm}

The proof consists of a series of lemmas. The first is immediate from \cite{lub}.

\begin{lemma}\label{chardeg}
Let $G$ be as in Theorem $\ref{2eltexcep}$, and let $\chi$ be a nontrivial irreducible character of $G$. Then $\chi(1)\ge N$, where $N$ is as in Table $\ref{chartab}$.
\end{lemma}

\begin{table}[htb]
\[
\begin{array}{|l|l|l|}
\hline
G & N & C \\
\hline
E_8(q) & q(q^6+1)(q^{10}+1)(q^{12}+1) & q^8-1 \\
E_7(q) & q(q^{14}-1)(q^6+1)/(q^4-1) & (q+1)^2q^7 \\
E_6^\e(q)\,(\e = \pm) & q(q^4+1)(q^6+\e q^3+1) & (q^4-1)(q-\e)^2,\,q\equiv 
\e \bmod 4 \\
       &   &  (q-\e)q^7,\,q\equiv -\e \bmod 4 \\
F_4(q) & q^8+q^4+1  & (q+1)^3q^3\\
G_2(q)\,(q>3) & q^3-1 & q^2-1 \\
\tw3D_4(q) & q(q^4-q^2+1) & (q^3-1)(q+1) \\
\tw2G_2(q)\,(q>3) & q^2-q+1 & q+1 \\
\hline
\end{array}
\]
\caption{Bounds for character degrees and centralizers}\label{chartab}
\end{table}

\begin{lemma}\label{2eltcent}
If $G$ is as in Theorem $\ref{2eltexcep}$, then $G$ has a $2$-element $s$ such that $|\bfC_G(s)|\le C$, where $C$ is as in Table $\ref{chartab}$.
\end{lemma}

\pf For the most part we construct the element $s$ within a suitable product of classical groups inside $G$, using the methods of \S7.2. 

For $G = E_8(q)$ we work in a subsystem subgroup $A$ of type $A_8$. This has shape $d.L_9(q).e$, where $e = (3,q-1)$ and $d = (9,q-1)/e$ (see for example \cite[Table 5.1]{LSS}); the derived subgroup is a quotient of $\SL_9(q)$ by a central subgroup $Z$. We shall define $s$ in $\SL_9(q)$, and identify it with its image modulo $Z$. Choose $\g \in \F_{q^8}$ of order $(q^8-1)_2$, and define $s_8 \in \GL_1(q^8) \le \GL_8(q)$ to be conjugate over $\bar \F_q$ to $\hbox{diag}(\g,\g^q,\g^{q^2},\ldots ,\g^{q^7})$. Let $s = \hbox{diag}(s_8,\a) \in \SL_9(q)$, where $\a^{-1} = \hbox{det}(s_8)$. Then $|\bfC_A(s)| = q^8-1$. 
Now, by \cite[11.2]{LSei}, 
\[
\cL(E_8)\downarrow A_8 = \cL(A_8) + V_{A_8}(\l_3) + V_{A_8}(\l_6).
\]
Here $ V_{A_8}(\l_3) \cong \wedge^3(V_9)$, the wedge-cube of the natural module for $\SL_9(q)$, and $ V_{A_8}(\l_6) $ is the dual of this. Since $\g^{q^i+q^j+q^k}$ cannot equal 1 for distinct $i,j,k$ between 0 and 7, and also $\g^{q^i+q^j}$ cannot lie in $\F_q$, the element $s$ has no nonzero fixed points in $ \wedge^3(V_9)$, so $\dim \bfC_{\cL(E_8)}(s) = 8$. Hence $\bfC_G(s)$ is a maximal torus, so $\bfC_G(s)=\bfC_A(s)$ of order $q^8-1$.

Next consider $G = E_7(q)$. We shall work in the simply connected version of $G$; the element $s$ we construct works equally well for the adjoint version. Let $A$ be a subsystem subgroup of type $A_2^\e A_5^\e$ ($\e = \pm 1$), where 
$q\equiv -\e \bmod 4$. This has the subgroup 
$\SL_3^\e(q) \circ \SL_6^\e(q)$ of index $(3,q-\e)$. Let $\g \in \F_{q^4}$ have order $(q^4-1)_2$, and define
$\a = \g^{(q^2+1)(\e q+1)}$, $\b = \g^{2(\e q+1)}$. Now define $s_1 \in \SL_3^\e(q)$, $s_2\in \SL_6^\e(q)$ so that they are conjugate over $\bar \F_q$ to 
\[
\hbox{diag}(\g^{-2},\g^{-2\e q},\b) \in \SL_3,\;\;  \hbox{diag}(\g,\g^{\e q},\g^{q^2},\g^{\e q^3},1,\a) \in \SL_6,
\]
respectively. Let $s = s_1s_2 \in A$. Then $|\bfC_A(s)| = (q^4-1)(q^2-1)(q-\e)$. From \cite[11.8]{LSei},
\[
\cL(E_7)\downarrow A_2A_5 = \cL(A_2A_5) + (V_{A_2}(\l_1)\otimes V_{A_5}(\l_2))+(V_{A_2}(\l_2)\otimes V_{A_5}(\l_4)).
\]
Here $V_{A_2}(\l_1)\otimes V_{A_5}(\l_2) \cong V_3 \otimes \wedge^2(V_6)$, where $V_3,V_6$ are the natural modules for $\SL_3$ and $\SL_6$. One checks that $s$ has fixed space of dimension 1 on this module (coming from the product of the eigenvalues $\b,1,\a$). Hence $\dim \bfC_{\cL(E_7)}(s) = 9$, and 
so over $\bar \F_q$ we deduce that $\bfC_{E_7}(s) = A_1T_6$, where $T_6$ denotes a torus of rank 6. It follows that $|\bfC_G(s)| = |A_1(q)|\cdot |T_6(q)|$. As $\bfC_G(s)$ contains $\bfC_A(s)$, of the order given above, $|\bfC_G(s)| \le |A_1(q)|(q^4-1)(q+1)^2 < (q+1)^2q^7$.

If $G = E_6^\e(q)$, then we work in a subsystem subgroup $A$ of type $A_1A_5$ containing $\SL_2(q) \circ \SL_6^\e(q)$. Again let $\g \in \F_{q^4}$ have order $(q^4-1)_2$ and define $s_2 \in \SL_6^\e(q)$ as the previous paragraph. Define $s_1 \in \SL_2(q)$ to be conjugate over $\bar \F_q$ to $\hbox{diag}(\g^{2(q+\e)}, \g^{-2(q+\e)})$ if $q \equiv \e \bmod 4$, and to $I_2$ otherwise. Set $s = s_1s_2$. Then $|\bfC_A(s)|$ is equal to $(q^4-1)(q-\e)^2$ if $q \equiv \e \bmod 4$, and to 
$|A_1(q)|(q^4-1)(q-\e)$ otherwise. 
By \cite[11.10]{LSei},
\[
\cL(E_6)\downarrow A_1A_5 = \cL(A_1A_5) + (V_{A_1}(1)\otimes V_{A_5}(\l_3)),
\]
and the second summand is $V_2\otimes \wedge^3(V_6)$, where $V_2,V_6$ are the natural modules for $A_1,A_5$. We check that $s$ has no nonzero fixed points on this tensor product, and it follows that $\bfC_{E_6}(s) = \bfC_{A_1A_5}(s)$; hence 
$\bfC_G(s) = \bfC_A(s)$, giving the result.

Now let $G = F_4(q)$. Here we construct our element $s$ in a subsystem subgroup $A = B_4(q) \cong {\rm \Spin}_9(q)$. It is convenient to define it in the quotient $\Omega_9(q)$ and take a preimage. We follow the proof of Lemma 
\ref{reg-so}. Let $\g \in \F_{q^2}$ have order $(q^2-1)_2$, and define $s_4 \in \GL_1(q^2) \le \GL_2(q) \le SO_4^+(q)$ to be conjugate over $\bar \F_q$ to ${\rm diag}(\g,\g^q,\g^{-1},\g^{-q})$. Then $s_4$ has spinor norm $-1$. Let $q\equiv \e \bmod 4$ with $\e = \pm 1$, and define $s_2 \in SO_2^\e(q)$ to be conjugate to ${\rm diag}(\g^{q+\e},\g^{-(q+\e)})$. Then $s_2$ also has spinor norm $-1$, so 
$t_1:={\rm diag}(s_4,s_2) \in \Omega_6^\e(q)$. Finally, let $t_2:= {\rm diag}(-1,-1,1) \in \Omega_3(q)$ and define $s \in A$ to be the preimage of ${\rm diag}(t_1,t_2) \in \Omega_9(q)$. Then $|\bfC_A(s)| = (q^2-1)(q-\e)^2$. Now
\[
\cL(F_4) \downarrow B_4 = \cL(B_4) \oplus V_{B_4}(\l_4).
\]
The second summand is the spin module for $B_4(q)$, which restricts to the preimage of  $\O_6^\e(q) \times \O_3(q)$ as $(V_4\otimes V_2)\oplus (V_4^*\otimes V_2^*)$, where each summand is a tensor product of natural modules for the isomorphic group $\SL_4^\e(q) \times \SL_2(q)$. Elements of $\SL_4^\e(q)$, $\SL_2(q)$ inducing $t_1$, $t_2$ are $x_1:= {\rm diag}(1,\g,\g^{\e q},\g^{-\e q-1})$, $x_2:={\rm diag}(\g^{(q-\e)/2}, \g^{-(q-\e)/2})$, respectively. The tensor product of $x_1$ and $x_2$ has fixed point space of dimension at most 1, and it follows that $\dim \bfC_{\cL(F_4)}(s) =4$ or 6. If it is 4 then $\bfC_G(s) = \bfC_A(s)$, while if it is 6, then $\bfC_{F_4}(s) = A_1T_3$, whence $|\bfC_G(s)| \le |A_1(q)|(q+1)^3$, as in the conclusion. 

For $G = G_2(q)$ or $\tw3D_4(q)$, we pick our element $s$ in a subgroup $A = \SL_3(q)$: let $\g \in \F_{q^2}$ have order $(q^2-1)_2$ and take $s$ to be conjugate over $\bar \F_q$ to ${\rm diag}(\g,\g^q,\a)$ where $\a = \g^{-(q+1)}$. Now 
$\cL(G_2) \downarrow A_2 = \cL(A_2)+V_3+V_3^*$ and $\cL(D_4)\downarrow A_2 = \cL(A_2)+V_3^3+(V_3^*)^3+V_1^2$, where $V_3$ is the natural 3-dimensional module and $V_1$ is trivial. It follows that $\bfC_{\cL(G_2)}(s)$ and $\bfC_{\cL(D_4)}(s)$ have dimensions 2 and 4 respectively, so $\bfC_G(s)$ is a maximal torus, as in the conclusion.

Finally, for $G=\tw2 G_2(q)$, an element $s$ of order 4 has centralizer of order $q+1$ (see \cite{ward}). This completes the proof.
 \hal

\begin{lemma}\label{first}
Theorem $\ref{2eltexcep}$ holds for $E_8(q)$, $E_7(q)$, $G_2(q)$, $\tw2G_2(q)$, and also for $E_6^\e(q)$ 
with $q \equiv \e \bmod 4$.
\end{lemma}

\pf Let $G$ be one of these groups, and let $s$ be the 2-element of $G$
produced in Lemma \ref{2eltcent}. As in the proof of Proposition \ref{main2-sl}, 
it is sufficient to establish that for every $g\in G$, 
\begin{equation}\label{eqn}
\sum_{\chi \in \Irr(G)} \frac{\chi(s)^3\bar \chi(g)}{\chi(1)^2} \ne 0,
\end{equation}
and to prove this it suffices to show
\[
\sum_{1\ne \chi \in \Irr(G)} \frac{|\chi(s)|^3}{\chi(1)} < 1.
\]
Lemma \ref{chardeg} implies that $\chi(1)\ge N$ for all nontrivial irreducible characters $\chi$, where $N$ is as in Table \ref{chartab}. Hence 
\[
\sum_{1\ne \chi \in Irr(G)} \frac{|\chi(s)|^3}{\chi(1)} < \frac{|\bfC_G(s)|^{1/2}}{N}\sum_{\chi \in Irr(G)} |\chi(s)|^2
= \frac{|\bfC_G(s)|^{3/2}}{N} \le \frac{C^{3/2}}{N},
\]
where $C$ is as in Table \ref{chartab}. One checks that $C^{3/2}/N < 1$ for the groups in the hypothesis, so the lemma follows. \hal

\begin{lemma}\label{second}
Theorem $\ref{2eltexcep}$ holds for $E_6^\e(q)$ with $q \equiv -\e \bmod 4$, $F_4(q)$ and 
$\tw3D_4(q)$.
\end{lemma}

\pf Let $G$ be one of these groups, let $s$ be the 2-element of $G$ from 
Lemma \ref{2eltcent}, and let $g \in G$. As in the previous proof, 
\[
\sum_{1\ne \chi \in \Irr(G)} \frac{|\chi(s)|^3 |\chi(g)|}{\chi(1)^2} < \frac{|\bfC_G(s)|^{3/2}|\bfC_G(g)|^{1/2}}{N^2} \le 
\frac{C^{3/2}|\bfC_G(g)|^{1/2}}{N^2},
\]
where $C,N$ are as in Table \ref{chartab}. The result is proved if the above sum is less than 1, 
so we may assume that 
\begin{equation}\label{centeqn}
|\bfC_G(g)| \ge \frac{N^4}{C^3}.
\end{equation}
Our strategy is to show that an element $g$ satisfying this bound must lie in a subgroup of $G$ that is a commuting product of quasisimple classical groups. (A similar strategy was carried out in Section 7 of \cite{ore}.) The conclusion then follows immediately from the results in Section 7.3, where Theorem \ref{2eltexcep} is established for classical groups. 

Consider $G = F_4(q)$. Here (\ref{centeqn}) gives
\begin{equation}\label{f4cent}
|\bfC_G(g)| \ge \frac{(q^8+q^4+1)^4}{(q+1)^9q^9}.
\end{equation}
Assume first that $g$ is a unipotent element. The classes and centralizers of unipotent elements in $G$ are given in \cite[Table 22.2.4]{LSei}, and every centralizer satisfying the above bound has even order. Hence there is an involution $t$ such that $g \in \bfC_G(t)$. Now $\bfC_G(t)$ is either a quasisimple group $B_4(q)$, or a group of the form $(\SL_2(q)\circ \Sp_6(q)).2$, with the unipotent element $g$ lying in the subgroup $\SL_2(q)\circ \Sp_6(q)$. Hence $g$ is in a product of quasisimple classical groups, except possibly in the case where $q=3$ and $g \in \bfC_G(t) = (\SL_2(3)\circ \Sp_6(3)).2$. In the latter case, a computation shows that every element of $\bfC_G(t)$ is a product of three 2-elements. 

Now assume $g$ is not unipotent; say $g = xu$ has semisimple part $x\ne 1$ and unipotent part $u \in \bfC_G(x)$. Now $\bfC_G(x)$ is a subsystem subgroup of $G$, and the bound (\ref{f4cent}) forces this  to have a normal subgroup $D = B_4(q)$, $D_4^\e(q)$, $B_3(q)$, $C_3(q)$, $A_3^\e(q)$, $B_2(q)$ or $A_2^\e(q)$. Then $x \in \bfC_G(D)$, and the unipotent elements of 
$\bfN_G(D)$ generate a subgroup of $D\bfC_G(D)$, which is contained in a subsystem subgroup $S:=B_4(q)$, $A_1(q)C_3(q)$ or $A_2^\e(q)A_2^\e(q)$. Hence $g = xu \in S$. Observe that $S$ is a product of quasisimple classical groups, except for
$A_1(q)C_3(q)$ when $q=3$; however, we already noted that every element of this subgroup is 
a product of three 2-elements in its normalizer. This completes the proof for $G = F_4(q)$.

The proof for $G = E_6^\e(q)$ is similar. If $g$ is unipotent then the bound (\ref{centeqn}) and \cite[Table 22.2.3]{LSei} imply that $\bfC_G(g)$ has even order, so $g \in \bfC_G(t)$ for some involution $t$. This centralizer is either $(q-\e)\circ D_5^\e(q)$ or 
$(\SL_2(q) \circ \SL_6^\e(q)).2$. Hence the unipotent element $g$ lies in $D_5^\e(q)$ or $\SL_2(q) \circ \SL_6^\e(q)$, and this is a product of quasisimple groups, apart from the latter when $q=3$, in which case a computation shows that every element of $\bfC_G(t)$ is a product of three 2-elements. When $g$ is not unipotent, the bound (\ref{centeqn}) is actually stronger than the bound used in the proof 
of \cite[Theorem 7.1]{ore} for non-unipotent elements of $E_6^\e(q)$,
% in \cite[p.\ 1004]{ore}, 
and this proof shows  that such elements lie in a product of quasisimple classical subgroups. 
Alternatively, an argument 
similar to that for $F_4(q)$ gives the result in this case.

Finally, let $G = \tw3D_4(q)$. The unipotent classes and centralizers can be found in \cite{spalt}, and the unipotent case is handled exactly as for $F_4(q)$. For $g=xu$ non-unipotent as above, (\ref{centeqn}) implies that $\bfC_G(x)$ has a normal subgroup $D = A_1(q^3)$ or $A_2^\e(q)$. In the first case we argue as before that $g=xu$ lies in $D\bfC_G(D) = A_1(q^3)\circ A_1(q)$. In the second case $\bfC_G(s) = ((q^2+\e q+1)\circ D).(3,q-\e)$, and we can assume that $u\ne 1$ (otherwise $g=x$ is real, and the result follows from Lemma \ref{real1}. The group generated by the unipotent elements of $\bfC_G(s)$ is just $D$, so $u \in D$. But the centralizer of a nontrivial unipotent element of $D = A_2^\e(q)$ has order at most $(q+1)q^3$ (see \cite[Chapter 3]{LSei}), so this gives $|\bfC_G(g)| \le (q^2+\e q+1)(q+1)q^3$, which contradicts (\ref{centeqn}).  \hal

\section{Asymptotic surjectivity: Proofs of Theorems \ref{main3} and \ref{main4}}
\begin{lem}\label{center}
Let $k, Q \geq 2$ be integers. There is an integer $D = D(k,Q)$ depending on
$k$ and $Q$ such that, for every integer $N$ with $\O(N) \leq k$ and for every $q < Q$, 
every central element of $G \in \{\SL_m(q),\SU_m(q),\Sp_{2m}(q),\O^+_{2m}(q)\}$
is an $N$th power in $G$ whenever $m|D$. 
\end{lem}

\pf
We define $D = 2(Q!)^{k+1}$ in the case $G = \SL$ or $\SU$, and $D = 2^{k+1}$ in 
the case $G = \Sp$ or $\O^+$. It suffices to prove the claim for nontrivial 
$z \in \bfZ(G)$. 

Consider the case $G = \SL_m(q)$ or $\SU_m(q)$, and set $\e = +$, respectively $\e = -$. 
Since $2|m$, 
$$\GL^\e_m(q) > \GL_{m/2}(q^2) \geq T:= C_{q^{m}-1}.$$
Furthermore, $T_1:= T \cap G$ has index dividing $q-\e$ and 
contains $\bfZ(G)$; in particular, $z \in T_1$. 
%Now, for any prime $p$ dividing $|z|$, 
If $p$ is a prime dividing $|z|$, then  $p|(q-\e)$, whence $p|(q^2-1)$  and 
$p \leq q+1 \leq Q$. Thus 
$$\left( \frac{q^m-1}{q^2-1} \right)_p =  \left( \frac{m}{2} \right)_p \geq ((q-\e)_p)^{k+1},$$
so
$$\left( \frac{|T_1|}{|z|} \right)_p \geq ((q-\e)_p)^k \geq p^k.$$
Write $N = N_1N_2$, where all prime divisors of $N_1$ divide $|z|$ and $\gcd(N_2,|z|) = 1$. 
Since $\O(N) \leq k$, we have shown that $N_1$ divides $|T_1|/|z|$. As $T_1$ is cyclic,
we can find $t \in T_1$ such that all prime divisors of $|t|$ divide $|z|$ and 
$t^{N_1} = z$. Since $\gcd(N_2,|t|) = 1$, $t = h^{N_2}$ for some $h \in T_1$. It follows that
$z = h^N$, as desired. 

% In the case $G = 
If $G$ is 
$\Sp_{2m}(q)$ or $\O^+_{2m}(q)$, then $|z| = 2$ and $q$ is odd. 
We can use the same argument as above, taking $T_1$ to be a cyclic maximal torus of 
order $q^m-1$ in $\Sp_{2m}(q)$, respectively $\SO^+_{2m}(q)$. 
\hal

Let $q$ be a prime power, let $n \geq 13$ be an integer, and let $\e = \pm$. 
If $\e = +$, then we use $\ps(q^n-\e)$ to denote a primitive prime divisor $\p(q,n)$ if 
$2 \nmid n$, and $\p(q,n)\p(q,n/2)$ if $2|n$.
If $\e = -$, then we use $\ps(q^n-\e)$ to denote a primitive prime divisor $\p(q,2n)$. 
These primitive prime divisors exist by \cite{Zs}. 

\begin{lem}\label{pair}
Let $q$ be a prime power, let $n \geq m \geq 13$ be integers, and 
let $\alpha, \beta = \pm$.
Suppose that $\gcd(\ps(q^n-\alpha),\ps(q^m-\beta)) > 1$. Then either $(n,\alpha) = (m,\beta)$, or 
$\alpha = +$ and $n \in \{2m,4m\}$.
\end{lem}

\pf
If $n = m$, then $\gcd(\ps(q^n-\alpha),\ps(q^m-\beta)) > 1$ certainly implies $\alpha = \beta$. Suppose 
$n > m$. If $\alpha = -$, then $\ps(q^n-\alpha) = \p(q,2n)$ 
does not divide $\prod^{2n-1}_{i=1}(q^i-1)$,  so it cannot be non-coprime to
$\ps(q^m-\beta)$. So $\alpha = +$, and $\gcd(\ps(q^n-1),\ps(q^m-\beta)) > 1$ implies
that $n = 2m$ or $n = 4m$.
\hal

Now we prove an analogue of \cite[Proposition 3.4.1]{LarShT} for groups of type $A$ and $C$:

\begin{prop}\label{C-bounds}
Fix $a\ge 1$, and let $n > 2a + 2$ be an integer. 
Let $s$ and $t$ be regular semisimple elements
of $G := \Sp_{2n}(q)$ belonging to maximal tori $T_{1}$ and $T_{2}$ of type
$T_{n-a,a}^{\e_1,\e_2}$ and $T_{a+1,n-a-1}^{\e_3,\e_4}$
respectively, where $\e_i =  \pm$ and
$\e_1\e_2 = -\e_3\e_4$.
The number of distinct irreducible characters of $G$
which vanish neither on $s$ nor on $t$ is bounded, independent of $n$, $q$,
and the choices of $s$ and $t$.
Likewise, the absolute values of these characters on $s$ and $t$ are bounded
independent of $n$, $q$, and the choices of $s$ and $t$.
\end{prop}

\pf
(i) First we show that the maximal tori $T_1$ and $T_2$ 
are {\it weakly orthogonal} in the sense of \cite[Definition 2.2.1]{LarShT} whenever 
$\e_1\e_2 = -\e_3\e_4$. We follow the proof of 
\cite[Proposition 2.6.1]{LarShT}. The dual group $G^{*}$ is 
$\SO(V) \cong \SO_{2n+1}(q)$, where
$V = \F_{q}^{2n+1}$ is endowed with a suitable quadratic form $Q$. Consider the tori dual to
$T_1$ and $T_2$, and assume $g$ is an element
belonging to both of them. We need to show that $g =1$.
We consider the spectrum $S$ of the semisimple element $g$ on $V$ as a
multiset. Then $S$ can be represented
as the joins of multisets $X \sqcup Y \sqcup \{1\}$ and $Z \sqcup T \sqcup \{1\}$, where
$$\begin{array}{l}
  X := \{x,x^{q}, \ldots,x^{q^{n-a-1}},x^{-1},x^{-q}, \ldots,
             x^{-q^{n-a-1}}\},\\
  Y := \{y,y^{q}, \ldots,y^{q^{a-1}},y^{-1},y^{-q},
             \ldots, y^{-q^{a-1}}\},\\
  Z := \{z,z^{q}, \ldots,z^{q^{n-a-2}},z^{-1},z^{-q}, \ldots,
             z^{-q^{a}}\},\\
  T := \{t,t^{q}, \ldots,t^{q^{a}},t^{-1},t^{-q},
             \ldots, t^{-q^{a}}\},\end{array}$$
for some $x,y,z,t \in \bar{\F}_{q}^{\times}$. Furthermore,
$$x^{q^{n-a}-\e_1} = y^{q^{a}-\e_2} = z^{q^{n-a-1}-\e_3} = t^{q^{a+1}-\e_4} = 1.$$
Let $A$ be a multiset of elements of $\bar{\F}_{q}$, where $1 \in A$, 
the multiplicity of each element of $A$ is $2n+1$, and with the property that
if $u \in A$ then $u^{q},u^{-1} \in A$. We claim that
if $|A \cap S| > 1$ then $A \supseteq S$. Indeed, since the multiplicity of every $u \in S$ is at most
$2n+1$, if $A \cap (X \sqcup \{1\}) > 1$ then $A \supseteq X$, and if
$|A \cap (X \sqcup \{1\})|, |A \cap (Y \sqcup \{1\})|  > 1$ then $A \supseteq S$; and similarly for
$Y$, $Z$, $T$. Now if $|A \cap S| > 1$ but $A \not\supseteq S$,
then $S = X \sqcup Y \sqcup \{1\}$ implies that 
$|A \cap S| \in \{2a+1, 2(n-a)+1\}$. But $S = Z \sqcup T \sqcup \{1\}$ also, so 
$|A \cap S| \in \{2a+3,2(n-a)-1\}$, which is a contradiction as
$n \geq 2a+3$.

Applying the claim to the multiset $A$ consisting of those $u \in \bar{\F}_{q}$ such that 
$u^{q^{n-a}-\e_1} = 1$, each with multiplicity $2n+1$, and
noting that $A \supseteq X \sqcup \{1\}$, we deduce that $u^{q^{n-a}-\e_1} = 1$ for
all $u \in S$. Arguing similarly, we obtain 
$$u^{q^{n-a}-\e_1} = u^{q^{a}-\e_2} = u^{q^{n-a-1}-\e_3} = u^{q^{a+1}-\e_4} = 1$$
for all $u \in S$.

Consider $u \in S$. 
Suppose for instance that $\e_3  \neq \e_1$.  In particular,
$$u^{q^{n-a-1}+\e_1} = u^{q^{n-a}-\e_1} = 1,$$
whence $u^{q+1} = 1$. The condition $\e_1\e_2 = -\e_3\e_4$ now implies
that $\e_2 = \e_4$, so $|u|$ divides $\gcd(q^{a+1}-\e_2,q^a-\e_2)|(q-1)$. It follows
that $u^2 = 1$ for all $u \in S$. The same argument applies to the case $\e_3 = \e_1$.
We have shown that $u^2 = 1$ for all $u \in S$. Now if $1$ has multiplicity at least $2$ in 
$S$, then applying the claim to the multiset $A'$ consisting only of $1$ with multiplicity 
$2n+1$, we see that $g = 1_V$ as stated. It remains to consider the case
$g = \diag(-1,-1, \ldots,-1,1)$. Now $\Ker(g+1_V)$ is a quadratic subspace of $V$
of type $\e_1\e_2$ and also of type $\e_3\e_4$, a contradiction.

\smallskip
(ii) Now we proceed exactly as in the proof of \cite[Proposition 3.4.1]{LarShT}, using 
the main result of \cite{Lusztig} which holds for both types $B_n$ and $C_n$. Also note
that the proof of \cite[Proposition 3.4.1]{LarShT} uses only the weak orthogonality of
the two tori  $T_1$ and $T_2$ but not the signs $\e_i$ in their definitions. 
\hal

\begin{prop}\label{A-bounds}
Fix $a\ge 1$, $\e = \pm$, and let $n$ be an integer greater than $2a+2$.
Let $s$ and $t$ be regular semisimple elements
of $G := \SL^\e_{n}(q)$ belonging to maximal tori $T_{1}$ and $T_{2}$ of type
$T_{n-a,a}$ and $T_{a+1,n-a-1}$.
The number of distinct irreducible characters of $G$
which vanish neither on $s$ nor on $t$ is bounded, independent of $n$, $q$,
and the choices of $s$ and $t$.
Likewise, the absolute values of these characters on $s$ and $t$ are bounded
independent of $n$, $q$, and the choices of $s$ and $t$.
\end{prop}

\pf
(i) Again, we show that the maximal tori $T_1$ and $T_2$ are weakly orthogonal. Here, the dual group $G^{*}$ is $\PGL^\e(V) \cong \PGL^\e_{n}(q)$, where
$V = \F_q^n$ for $\e = +$ and $V = \F_{q^2}^n$ for $\e = -$. Consider the 
complete inverse images $T_{n-a,a}$ and $T_{n-a-1,a+1}$ of the tori dual to
$T_1$ and $T_2$ in $H := \GL^\e(V)$, and assume $g$ is an element
belonging to both of them. We need to show that $g \in \bfZ(H)$.
The multiset $S$ of eigenvalues of the semisimple element $g$ on $V$ can be represented
as the joins of multisets $X \sqcup Y \sqcup \{1\}$ and $Z \sqcup T \sqcup \{1\}$, where
$$\begin{array}{l}
  X := \{x,x^{q\e}, \ldots,x^{(q\e)^{n-a-1}}\},~
  Y := \{y,y^{q\e}, \ldots,y^{(q\e)^{a-1}}\},\\
  Z := \{z,z^{q\e}, \ldots,z^{(q\e)^{n-a-2}}\},~
  T := \{t,t^{q\e}, \ldots,t^{(q\e)^{a}}\},\end{array}$$
for some $x,y,z,t \in \bar{\F}_{q}^{\times}$; furthermore,
$$x^{(q\e)^{n-a}-1} = y^{(q\e)^{a}-1} = z^{(q\e)^{n-a-1}-1} = t^{(q\e)^{a+1}-1} = 1.$$
Let $A$ be a multiset of elements of $\bar{\F}_{q}$, where 
the multiplicity of each element of $A$ is $n$, and with the property that
if $u \in A$ then $u^{q\e} \in A$. We claim that
if $A \cap S \neq \emptyset$ then $A \supseteq S$. Indeed, since the multiplicity of every $u \in S$ is at most $n$, if $A \cap X \neq \emptyset$ then $A \supseteq X$, and if
$A \cap X, A \cap Y \neq \emptyset$ then $A \supseteq S$; and similarly for
$Y$, $Z$, $T$. Now if $A \cap S \neq \emptyset$ but $A \not\supseteq S$,
then $S = X \sqcup Y$ implies that 
$|A \cap S| \in \{a, n-a\}$. But $S = Z \sqcup T$ as well, so 
$|A \cap S| \in \{a+1,n-a-1\}$, which is a contradiction as
$n \geq 2a+3$.

Applying the claim to the multiset $A$ consisting of those $u \in \bar{\F}_{q}$ such that 
$u^{(q\e)^{n-a}-1} = 1$, each with multiplicity $n$, and
noting that $A \supseteq X$, we see that $u^{(q\e)^{n-a}-1} = 1$ for
all $u \in S$. Arguing similarly, we see that
$u^{(q\e)^{n-a-1}-1} = 1$, so $u^{q\e-1} = 1$ for all $u \in S$. Now applying the claim to the multiset $A'$ consisting of only $x$ but with multiplicity $n$, and
noting that $A \supseteq X$, we conclude that $A =S$ and $g = x \cdot 1_V$, as stated.

\smallskip
(ii) Now we proceed as in the proof of \cite[Proposition 3.1.5]{LarShT}. Assume that
$\chi \in \Irr(G)$ and $\chi(s)\chi(t) \neq 0$. By (i) and \cite[Proposition 2.2.2]{LarShT}, 
$\chi = \chi_{\uni,\alpha}$ is a unipotent character of $G$ labeled by a partition $\alpha \vdash n$.
If $\chi_\alpha \in \Irr(\SSS_n)$ corresponds to $\alpha$, then 
$$\chi_\alpha(s_1) = \chi(s) \neq 0,~~\chi_\alpha(t_1) = \chi(t) \neq 0,$$
where $s_1 \in \SSS_n$ has cycle type $(n-a,a)$ and $t_1 \in \SSS_n$ has cycle type
$(n-a-1,a+1)$. Arguing as in the proof of \cite[Corollary 3.1.3]{LarShT}, one can show 
that there are at most $4a+6$ possibilities for $\alpha$, and 
$|\chi_\alpha(s_1)|$, $|\chi_\alpha(t_1)| \leq 4$.
\hal

\begin{prop}\label{tori}
For every positive integer $k$, there are positive integers $A = A(k)$, $B_1 = B_1(k)$, and 
$B_2= B_2(k)$, each depending on $k$, with the following property. For every $n \geq A$ and 
for every prime power $q$, a group 
$G \in \{ \SL_n(q), \SU_n(q), \Sp_n(q), \Spin^\pm_n(q)\}$ contains $k+1$ pairs $(s_i,t_i)$ of regular semisimple elements, $1 \leq i \leq k+1$, such that: 
\begin{enumerate}[\rm(a)]
\item If $i \neq j$, then $\gcd(|s_i| \cdot |t_i|,|s_j| \cdot |t_j|) = 1$;

\item For each $i$, there are at most $B_1$ irreducible characters of $G$ that vanish neither on 
$s_i$ nor on $t_i$. The absolute values of these characters at $s_i$ and $t_i$ are 
at most $B_2$. 
\end{enumerate} 
\end{prop}

\pf
(i) First we consider the case $G = \Spin^\e_{2n}(q)$ with $n \geq 10k+65$. For 
{\it odd} $a_i = 2i+11$, $1 \leq i \leq k+1$, there 
are regular semisimple elements $s_i$, $t_i$ of $G $ belonging to maximal 
tori $T^1_{i}$ and $T^2_{i}$ of type
$T_{n-a_i,a_i}^{\e,+}$ (of order $(q^{n-a_i}-\e)(q^{a_i}-1)$) and 
$T_{n-a_i-1,a_i+1}^{-\e,-}$ (of order $(q^{n-a_i-1}+\e)(q^{a_i+1}+1)$) 
respectively. In fact, we can choose 
$$|s_i| = \ps(q^{n-a_i}-\e) \cdot \ps(q^{a_i}-1),~~|t_i| = \ps(q^{n-a_i-1}+\e) \cdot \ps(q^{a_i+1}+1).$$
By \cite[Proposition 3.3.1]{LarShT} the number of distinct irreducible characters of $G$
that vanish neither on $s_i$ nor on $t_i$ is bounded by some integer $B_1(k)$, dependent on $k$ but independent of $n$, $q$. Likewise, the absolute values of these characters on $s_i$ and $t_i$ are bounded by some integer $B_2(k)$, dependent on $k$ but independent of $n$, $q$.

It remains to check the condition (a). Let $1 \leq i < j \leq k+1$. By the choice of 
$n$, $n/5 \geq a_j+1 \geq a_i+3 \geq 16$. It follows that
$$2(n-a_j-1)   > n-a_i > n-a_i-1 > \max(n-a_j,4(a_j+1)).$$ 
Hence, by Lemma \ref{pair}, each of $\ps(q^{n-a_i}-\e)$ and $\ps(q^{n-a_i-1}+\e)$
is coprime to $\ps(q^{n-a_j}-\e)\cdot\ps(q^{n-a_j-1}+\e)\cdot\ps(q^{a_j}-1)\cdot\ps(q^{a_j+1}+1)$.
Similarly, as $n-a_j-1 > 4(a_i+1)$, each of $\ps(q^{a_i}-1)$ and $\ps(q^{a_i+1}+1)$
is coprime to $\ps(q^{n-a_j}-\e)\cdot\ps(q^{n-a_j-1}+\e)$.
Finally, since $a_j$ and $a_i$ are distinct odd integers, Lemma \ref{pair} also yields that 
$\ps(q^{a_i}-1)\cdot\ps(q^{a_i+1}+1)$ is coprime to $\ps(q^{a_j}-1)\cdot \ps(q^{a_j+1}+1),$
and we are done.

\smallskip
(ii) Suppose $G = \Spin_{2n+1}(q)$ with $n \geq 10k+65$. For 
{\it odd} $a_i = 2i+11$, $1 \leq i \leq k+1$, there 
are regular semisimple elements $s_i$, $t_i$ of $G $ belonging to 
maximal tori $T^1_{i}$ and $T^2_{i}$ of type
$T_{n-a_i,a_i}^{+,+}$ (of order $(q^{n-a_i}-1)(q^{a_i}-1)$) and 
$T_{n-a_i-1,a_i+1}^{-,-}$ (of order $(q^{n-a_i-1}+1)(q^{a_i+1}+1)$) 
respectively. In fact, we can choose 
$$|s_i| = \ps(q^{n-a_i}-1) \cdot \ps(q^{a_i}-1),~~|t_i| = \ps(q^{n-a_i-1}+1) \cdot \ps(q^{a_i+1}+1).$$
By \cite[Proposition 3.4.1]{LarShT} the number of distinct irreducible characters of $G$
that vanish neither on $s_i$ nor on $t_i$ is bounded by some integer $B_1(k)$, dependent on $k$ but independent of $n$, $q$. Likewise, the absolute values of these characters on $s_i$ and $t_i$ are bounded by some integer $B_2(k)$, dependent on $k$ but independent of $n$, $q$.
Finally, condition (a) is satisfied as shown in (i).

\smallskip
(iii) Consider the case $G = \Sp_{2n}(q)$ with $n \geq 10k+65$. For 
{\it odd} $a_i = 2i+11$, $1 \leq i \leq k+1$, there 
are regular semisimple elements $s_i$, $t_i$ of $G $ belonging to 
maximal tori $T^1_{i}$ and $T^2_{i}$ of type
$T_{n-a_i,a_i}^{+,+}$ (of order $(q^{n-a_i}-1)(q^{a_i}-1)$) and 
$T_{n-a_i-1,a_i+1}^{+,-}$ (of order $(q^{n-a_i-1}-1)(q^{a_i+1}+1)$) 
respectively. In fact, we can choose 
$$|s_i| = \ps(q^{n-a_i}-1) \cdot \ps(q^{a_i}-1),~~|t_i| = \ps(q^{n-a_i-1}-1) \cdot \ps(q^{a_i+1}+1).$$
Now we can finish as in (ii) but using Proposition \ref{C-bounds}.

\smallskip
(iv) Consider the case $G = \SL^\e_n(q)$ with $n \geq 4k+17$. For 
$a_i = 2i+5$, $1 \leq i \leq k+1$, there 
are regular semisimple elements $s_i$, $t_i$ of $G $ belonging to 
maximal tori $T^1_{i}$ and $T^2_{i}$ of type
$T_{n-a_i,a_i}$ (of order $(q^{n-a_i}-\e^{n-a_i})(q^{a_i}-\e^{a_i})$) and 
$T_{n-a_i-1,a_i+1}$ (of order $(q^{n-a_i-1}-\e^{n-a_i-1})(q^{a_i+1}-\e^{a_i+1})$) 
respectively. Next, observe that for every $m \geq 7$, there is a prime
$\p(-q,m)$ that divides $(-q)^m-1$ but does not divide $\prod^{m-1}_{i=1}((-q)^i-1$;
namely, we can take $\p(-q,m) = \p(q,2m)$ if $2 \nmid m$, $\p(-q,m) = \p(q,m)$ if 
$4|m$, and $\p(-q,m) = \p(q,m/2)$ if $4|(m-2)$. In particular, 
if $m \geq m' \geq 7$ and $\p(q\e,m) = \p(q\e,m')$, then $m = m'$. Now we can choose 
$$|s_i| = \p(q\e,n-a_i) \cdot \p(q\e,a_i),~~
    |t_i| = \p(q\e,n-a_i-1) \cdot \p(q\e,a_i+1).$$
Condition (b) follows from Proposition \ref{A-bounds}. By the choice of
$n$, $n/2 \geq a_j \geq a_i+2 \geq 9$ if $1 \leq i < j \leq k+1$. It follows that
$$n-a_i-1 > n-a_j > n-a_j-1 > a_j+1 > a_j > a_i+1,$$
so condition (a) is satisfied.
\hal

\noindent
{\bf Proof of Theorems \ref{main3} and \ref{main4}.}
Let $k$ be a positive integer. For Theorem \ref{main3} we assume that $N$ is a
positive integer with $\pi(N) \leq k$. For Theorem \ref{main4} we assume that $N$ is 
a positive integer with $\O(N) \leq k$. By Proposition \ref{alt2}, it suffices to prove 
the two theorems for finite simple classical groups $S$ of sufficiently large rank (and defined 
over a sufficiently large field $\F_q$, in the case of Theorem \ref{main3}). So we assume 
that $S = G/Z$, where $Z := \bfZ(G)$ and $G = \Cl_n(q)$ with 
$\Cl \in \{ \SL, \SU, \Sp, \O^\e\}$ (and $\e = \pm$). Let $V := \F_q^n$ (if $\Cl \neq \SU$) and 
$V := \F_{q^2}^n$ (for $\Cl = \SU$) denote the natural $G$-module. 
Also set $\e = +$ if $\Cl = \SL$ and $\e = -$ when $\Cl = \SU$.

\smallskip
(i) Apply Proposition \ref{tori} to $G$ and consider $n \geq A$. Since $\pi(N) \leq k$, by 
\ref{tori}(a) there is some $i_0$ between 
$1$ and $k+1$ such that the orders of $s := s_{i_0}$ and $t := t_{i_0}$ are coprime to $N$.  
Define 
$$Q = Q(k) := (B_1B_2^2)^{481}.$$  

We claim that if $q \geq Q$, then every $g \in G \setminus Z$ belongs to $s^G \cdot t^G$, so 
it is a product of two $N$th powers; in particular, we are done with Theorem \ref{main3}. 
Indeed, since $g \notin Z$, its {\it support} $\supp(g)$, as defined in \cite[Definition 4.1.1]{LarShT}, is at least 1. 
It follows by \cite[Theorem 4.3.6]{LarShT} and the condition on $q$ that 
$$\frac{|\chi(g)|}{\chi(1)} < q^{-1/481} \leq \frac{1}{B_1B_2^2}$$
for every $1_G \neq \chi \in \Irr(G)$. Now condition \ref{tori}(b) implies that
$$\sum_{1_G \neq \chi \in \Irr(G)}\frac{|\chi(s)\chi(t)\chi(g)|}{\chi(1)} < \frac{B_1B_2^2}{B_1B_2^2} 
    = 1,$$
so $g \in s^G \cdot t^G$ as desired.    

\smallskip
(ii) Now we consider the case $2 \leq q < Q$ and $\O(N) \leq k$. Suppose that $g \in G$ satisfies 
$$\supp(g) \geq C = C(k) := (\log_2 Q)^2.$$
By \cite[Theorem 4.3.6]{LarShT}, 
$$\frac{|\chi(g)|}{\chi(1)} < q^{-\sqrt{\supp(g)}/481} \leq 2^{-(\log_2 Q)/481} = \frac{1}{B_1B_2^2}$$
for every $1_G \neq \chi \in \Irr(G)$. Hence, as in (i), $g \in s^G \cdot t^G$, so
$g$ is a product of two $N$th powers. 

\smallskip
(iii) It remains to consider the non-central $g \in G$ with $\supp(g) < C$. Recall the 
integer $D = D(k,Q)$ defined in the proof of Lemma \ref{center}, according to which 
\begin{equation}\label{div}
  8|D, ~~(q-\e)|D.
\end{equation}   
We also choose 
\begin{equation}\label{large}
  n \geq \max(A,2C+(9k+4)D). 
\end{equation}  
Since $\supp(g) < C \leq n/2$, by \cite[Proposition 4.1.2]{LarShT} we see that $g$ has 
a {\it primary eigenvalue} $\lambda$, where $\lambda^{q-\e} = 1$ in the case $\Cl = \SL^\e$ and 
$\lambda= \pm 1$ in the case $\Cl = \Sp, \O$. Moreover, arguing as in the proof of 
\cite[Lemma 6.3.4]{LarShT}, we get that $g$ fixes an (orthogonal if $\Cl \neq \SL$) decomposition
$$V = U \oplus W,$$
where $\dim U \geq n-2C$ and $U \supseteq \Ker (g-\lambda \cdot 1_V)$.

Now we consider a chain of (non-degenerate if $\Cl \neq \SL$) subspaces 
$$U_1 \subset U_2 \subset \ldots \subset U_{k+1} \subset U$$
with $\dim U_j = jD$, and moreover $U_j$ is of type $+$ if $\Cl = \O^\pm$ 
(this can be achieved since $\dim U \geq n-2C \geq (9k+4)D$ by (\ref{large})).
We also define 
$$W_j:= W \oplus (U_j^\perp \cap U),$$
so 
$$V = U_j \oplus W_j,~~d_j := \dim W_j = n-jD.$$
Setting $\cR_j := \cR(\Cl(W_j))$, the set of primes defined in Theorem \ref{prime1}, we claim that
\begin{equation}\label{2primes}
   \cR_i \cap \cR_j = \emptyset
\end{equation}
whenever $1 \leq i \neq j \leq k+1$. Assume the contrary: so $\ell \in \cR_i \cap \cR_j$ for some 
$i < j$. By the construction of $\cR_i$, 
$$\ell |(q^{2d_i}-1)(q^{2d_i-2}-1)(q^{2d_i-4}-1),$$
and similarly for $j$. Note that 
$$kD \geq d_i-d_j = (j-i)D \geq D \geq 8$$
by (\ref{div}).
It follows that $\ell |(q^e-1)$, where 
$$12 \leq 2d_i-4-2d_j \leq e \leq 2d_i-2d_j+4 \leq 2kD+4.$$
On the other hand, (\ref{large}) implies that 
$$(d_j-2)/4 \geq (n-(k+1)D-4)/4 > 2kD+4.$$
We have shown that some $\ell \in \cR(\Cl_{d_j}(q))$ divides $p^{ef}-1$ with 
$12 \leq e < (d_j-2)/4$ and $q = p^f$. This contradicts the construction of $\cR(\Cl_{d_j}(q))$ in 
Theorem \ref{prime1}, according to which $\ell = \p(p,af)$ for some 
$a \geq (d_j-1)f/4$.  

Since $\pi(N) \leq k$, (\ref{2primes}) now implies that 
there is some $i$ such that $N$ is not divisible by any prime in $\cR_i$.  
Hence, by Theorem \ref{prime1}, $H:=\Cl(W_i)$ admits two regular semisimple elements 
$s',t'$, whose orders are coprime to $N$, and such that 
$(s')^H \cdot (t')^H \supseteq H \setminus \bfZ(H)$. 

Next, $G$ contains a subgroup $\Cl(U_i) \times \Cl(W_i)$. Note that 
$g$ acts on $U_i$ as the scalar $\lambda$. Condition (\ref{div}) now implies that
$x= g|_{U_i} \in \bfZ(\Cl(U_i))$, whence $x = u^N$ for some $u \in \Cl(U_i)$ by Lemma 
\ref{center} (as $\O(N) \leq k$). 
Since $g$ fixes $W_i$, it follows that $g = xy$ with $y = g|_{W_i} \in H=\Cl(W_i)$.
Note that $y$ has $\lambda$ as 
an eigenvalue but does not act as the scalar $\lambda$. It  follows that 
$y \in H \setminus \bfZ(H) \subseteq (s')^H \cdot (t')^H$, so $y = v^Nw^N$ with $v,w \in H$. As 
$x = u^N$ with $u \in \Cl(U_i)$ centralizing $v \in H$, we conclude that $g = (uv)^Nw^N$, as 
desired. 
\hal

\end{document}